\documentclass[letterpaper,11 pt]{article}
\usepackage[margin=1in]{geometry}
\usepackage[T1]{fontenc}    
\usepackage{multirow}
\usepackage{wrapfig}
\usepackage{booktabs}       
\usepackage{enumerate}
\usepackage{amsmath}
\usepackage{lipsum}
\usepackage{mathtools}
\usepackage{float}
\usepackage[dvipsnames,table,xcdraw]{xcolor}
\usepackage{bbm}
\usepackage{varioref}
\usepackage{hyperref}
\usepackage{float}
\usepackage{placeins}
                                
\usepackage{amssymb} 
\usepackage{rotating}
\usepackage{graphicx}
\usepackage{algpseudocode}
\usepackage{mathrsfs}  
\usepackage{algcompatible}
\usepackage{amsmath}
\usepackage{cases}
\usepackage{graphicx}
\usepackage{multirow}
\usepackage{comment}
\usepackage{subcaption}
\usepackage[normalem]{ulem}

\usepackage{amsthm}
\DeclareUnicodeCharacter{00A0}{~}
\usepackage{algorithm}
\usepackage{algpseudocode}
\algnewcommand{\Initialize}[1]{%
	\State \textbf{Initialization:}
	\Statex {\raggedright #1}
}
\usepackage{xcolor} 

\usepackage{tcolorbox}

\usepackage{graphicx} 
\usepackage{amsmath}
\usepackage{amsthm}
\usepackage{amsfonts}
\usepackage{amssymb}
\usepackage{extpfeil}  
\usepackage{geometry}
\usepackage{comment}
\usepackage{algorithm}
\usepackage{algpseudocodex}
\usepackage{enumerate}
\usepackage{enumitem}
\usepackage{booktabs, multirow} 
\usepackage{soul}
\usepackage{xcolor,colortbl} 
\usepackage{changepage,threeparttable} 
\usepackage[normalem]{ulem}
\usepackage{tcolorbox}
\usepackage{makecell}

\usepackage[textsize=footnotesize,textwidth=0.8in]{todonotes}

\definecolor{aqua}{rgb}{0.0, 0.5, 1.0}

\newcolumntype{L}[1]{>{\raggedright\arraybackslash}p{#1}}

\usepackage{xcolor}

\newcommand{\prox}{\mathrm{prox}}
\newcommand{\dom}{\mathrm{dom}}
\newcommand{\argmin}[1]{\operatorname*{argmin}_{#1}}

\def\R{\mathbb{R}}
\def\N{\mathbb{N}}

\def\E{\mathbb{E}}
\def\bP{\mathbb{P}}
\def\tf{\tilde{f}}
\def\tg{\tilde{g}}
\def\sF{\mathcal{F}}
\def\Real{\mathbb{R}}
\usepackage{bm}
\usepackage{tcolorbox}
\newcommand{\bxi}{\boldsymbol{\xi}}

\allowdisplaybreaks
\usepackage{tikz}
\usetikzlibrary{decorations.pathreplacing}

\NewDocumentEnvironment{alignb}{b}{%
  \begin{align*}
  \refstepcounter{equation} #1 \tag{\theequation}
  \end{align*}
}{\ignorespacesafterend}

\allowdisplaybreaks

\usepackage{pgfplots}
\pgfplotsset{compat=1.17}

\theoremstyle{plain}%

\newtheorem{theorem}{Theorem}
\newtheorem{proposition}{Proposition}
\newtheorem{lemma}{Lemma}
\newtheorem{remark}{Remark}
\newtheorem{corollary}{Corollary}
\newtheorem{assumption}{Assumption}

\usepackage[sort&compress]{natbib}
 \bibpunct[, ]{[}{]}{,}{n}{}{,}%

\usepackage{geometry}
\title{A Stochastic LineseArch Method (SLAM) for Minimizing Expectation Residuals
}
 
\author{Qi Wang\thanks{Department of Industrial and Operations Engineering, University of Michigan, United States. {\tt\small  qiwangqi@umich.edu}} \and Uday V. Shanbhag\thanks{Department of Industrial and Operations Engineering, University of Michigan, United States. {\tt\small  udaybag@umich.edu}
} \and Yue Xie\thanks{Musketeers Foundation Institute of Data Science and Department of Mathematics, University of Hong Kong, China. {{\tt\small yxie21@hku.hk} (Corresponding author)}
} 
}

\begin{document}
\sloppy
\maketitle
\thispagestyle{empty}
\pagestyle{plain}

\maketitle
 \begin{abstract}
 Most existing rate and complexity guarantees for stochastic gradient methods in $L$-smooth settings mandates that such sequences be non-adaptive, non-increasing, and upper bounded by $\tfrac{a}{L}$ for $a > 0$. This requires knowledge of $L$ and may preclude larger steps. Motivated by these shortcomings, we present an Armijo-enabled stochastic linesearch framework with common random number (CRN)-enabled stochastic zeroth- and first-order oracles. 
 The resulting steplength sequence is non-monotone and requires neither knowledge of $L$ nor any other problem parameters. We then prove that the expected stationarity residual diminishes at a rate of $\mathcal{O}(1/\sqrt{K})$, where $K$ denotes the iteration budget. Furthermore, the resulting iteration and sample complexities for computing an $\epsilon$-stationary point are $\mathcal{O}(\epsilon^{-2})$ and $\mathcal{O}\left(\epsilon^{-4}\right)$. The proposed method allows for a simple nonsmooth convex component in the objective, addressed through proximal gradient updates. Analogous guarantees are provided in the Polyak-{\L}ojasiewicz (PL) setting and convex regimes.  Preliminary numerical experiments are seen to be promising.
\end{abstract}


\fontsize{11}{12}\selectfont

\section{Introduction}\label{sec:intro}
In this paper, we present a proximal stochastic gradient method with linesearch for minimizing the nonsmooth function $\phi$ over $\R^n$, as specified by 
\begin{align}
\min_{x \in \R^n} \ \phi(x) \, \triangleq \, f(x) + r(x), \quad \text{where } f(x) \, \triangleq \, \E \left[\, F(x, \bxi)\, \right], \label{eq:composite_opt}
\end{align}
where $\bxi: \Omega \to \Xi \subseteq  \mathbb{R}^d$ denotes a random
variable with realizations denoted by $\xi$,  $\Xi \,
\triangleq \, \left\{ \, \xi(\omega) \, \mid \, \omega \, \in
\, \Omega\,\right\}$,  $F: \Real^{n}  \times \Xi \to \Real$
is a real-valued function, $(\Omega,\mathcal{F},\mathbb{P})$ denotes the associated
probability space, and $r:\R^n \to \R\cup \left\{\, +\infty\, \right\}$ is a proper, 
closed, and convex function. Our linesearch method is overlaid on
the standard stochastic gradient descent (SGD)
method~\cite{Robbins1985} to aid in selecting a suitable
stepsize using some additional (random)
function evaluations. Employing standard unbiased stochastic zeroth- and first-order oracles, our proposed scheme generates a sequence $\{x_k\}$ that  achieves a
non-asymptotic rate guarantee for an iteration budget $K$ (see Theorem~\ref{thm.nonaymptotic constant nk pj}), given by
\begin{align*}
    \E\left[\, \left\|\, G_s(x_{R_K})\, \right\|^2\, \right] \, \le \, c/K,
\end{align*}
where $c$ is a prescribed positive scalar, $\|G_s(\bullet)\|$ is a stationary metric, and
$x_{R_K}$ is the solution at a randomly selected index across
$K$ iterations.

\medskip
\noindent \textbf{Deterministic nonconvex optimization.} In deterministic
unaccelerated first-order settings, when $f$ is $L$-smooth, i.e., when the gradient map $\nabla f$ is $L$-Lipschitz continuous, convergence and rate
guarantees for gradient descent with constant stepsize $t$ require that  $t
\in (0, \tfrac{2}{L})$. $L$-smoothness  is an oft-used requirement imposed in  the convergence analysis of first-order methods for smooth optimization problems. This property guarantees the standard quadratic upper model used to establish sufficient descent. Although methods under weaker regularity conditions exist, global Lipschitz smoothness will be assumed by essentially all methods considered in this paper. However, since $L$ is problem-dependent and typically
unknown, such choices are far too conservative, leading to {slow convergence} in
practice.  Consequently, after a descent direction is determined, a backtracking
Armijo linesearch framework has been widely employed for selecting a stepsize
that ensures sufficient descent and thereby  guaranteeing global convergence to
a stationary point~\cite{Nocedal2006,Beck2014}; such an avenue requires
$\mathcal{O}(\epsilon^{-2})$ iterations (and gradient evaluations) to find an
iterate $x$ for which $\|\nabla f(x)\| \le \epsilon$~\cite[Theorem
2.2.2]{Cartis2022} in smooth nonconvex settings. Although such complexity bounds
are similar to those obtained in constant stepsize settings where stepsizes are
bounded by $\tfrac{2}{L}$, linesearch avenues are advantageous in that (i)
knowledge of $L$ is not needed;  and  (ii) larger steps are acceptable, leading
to improved behavior and broader applicability. With the introduction of
nonsmooth convex components (given by the function $r$) to allow for the presence of convex constraints or
nonsmooth convex regularizers, proximal gradient methods with an
Armijo-based linesearch have been studied in
\cite{Bertsekas1976,Beck2009,Salzo2017,Beck2017,Yang2024a}. The Armijo condition
may be modified to accommodate the nonsmooth component and this condition
may differ when $f$ is convex rather than when it is nonconvex
(see~\cite{Beck2017}). Our proposed method employs such conditions in the
stochastic regime with some subtle modifications.

\medskip
\begin{table}[htbp]
        \caption{Comparison of stochastic {linesearch} methods}\label{tab: review}
        \scriptsize
        \begin{tabular}{L{1.1cm}L{3.7cm}L{2.8cm}L{1.6cm}L{2cm}L{2.7cm}}\toprule
            ref.& 
            prob. assumptions& 
            oracle assumptions (CRN required)& 
            batch size sequence& 
            {upper bound for stepsizes}& 
            complexities\\
            \midrule
            \cite{Vaswani2019} & 
\vspace{-0.5\baselineskip}
\begin{itemize}[leftmargin=1em, nosep]
    \item nonconvex
    \item $L_i$-smooth {\bf (known $L$)}
    \item {\bf strong growth cond.}
    \item finite sum
\end{itemize}
& 
            unbiased $f$ and $g$ & 
            constant & 
            $s \le \tfrac{2}{\rho L}\ (\rho > 1)$ & 
            iter. and samp. $\mathcal{O}(\epsilon^{-2})$\\
            \midrule
            \cite{Vaswani2019}& 
\vspace{-0.5\baselineskip}
\begin{itemize}[leftmargin=1em, nosep]
    \item convex
    \item $L_i$-smooth
    \item {\bf interpolation}
    \item finite sum
\end{itemize}
&
            unbiased $f$ and $g$& 
            constant& 
            n/a& 
            iter. and samp. $\mathcal{O}(\epsilon^{-1})$ \\
            \midrule
            \cite{Jiang2023}& 
\vspace{-0.5\baselineskip}
\begin{itemize}[leftmargin=1em, nosep]
    \item convex
    \item $L_i$-smooth
    \item {\bf compact feasible set}
    \item finite sum
\end{itemize}
&
            unbiased $f$ and $g$& 
            constant& 
            n/a& 
            iter. and samp. $\mathcal{O}(\epsilon^{-2})$\\
            \midrule
            \cite{Tsukada2024}& 
\vspace{-0.5\baselineskip}
\begin{itemize}[leftmargin=1em, nosep]
    \item nonconvex
    \item $L_i$-smooth {\bf (known $L$)}
\end{itemize}&
            unbiased $f$ and $g$, $g$ has bounded variance &$\mathcal{O}(\epsilon^{-2})$ & 
            $s \le \bar{s}_{\alpha,\beta,L}$\, (see (v)) & 
            iter. $\mathcal{O}(\epsilon^{-2})$, samp. $\mathcal{O}(\epsilon^{-4})$\\
            \midrule
            \textbf{SLAM} & 
\vspace{-0.5\baselineskip}
 \begin{itemize}[leftmargin=1em, nosep]
    \item nonconvex
    \item $L_i$-smooth
    \item {\bf bounded true gradient}
\end{itemize}           
&
            unbiased $f$ and $g$, $g$ has bounded variance & 
            $\mathcal{O}(\epsilon^{-2})$ or $\{k+1\}$& 
            n/a & 
            iter. $\mathcal{O}(\epsilon^{-2})$, samp. $\mathcal{O}(\epsilon^{-4})$ (\textbf{Thm.~\ref{thm.nonaymptotic constant nk pj}}, \textbf{Cor.~\ref{cor.nonaymptotic increasing nk pj}})\\
            \midrule
            \textbf{SLAM$^{\rm con}$}                                   & 
\vspace{-0.5\baselineskip}
 \begin{itemize}[leftmargin=1em, nosep]
    \item convex
    \item $L_i$-smooth
    \item {\bf compact feasible set}
\end{itemize}           
&
            unbiased $f$ and $g$, $g$ has bounded variance & 
            $\mathcal{O}(\epsilon^{-1})$&
            n/a & 
            iter. $\mathcal{O}(\epsilon^{-1})$, samp. $\mathcal{O}\left(\epsilon^{-2}\right)$ \textbf{Thm.~\ref{thm.convex.main}}\\
            \bottomrule
        \end{tabular}
    \scriptsize

    (i) $L_i$-smooth: $\nabla F(\cdot, \xi)$ is $L$-Lipschitz continuous for each $\xi$;
    (ii) strong growth cond.: $\E[\|\nabla F(x, \bxi)\|^2] \le \rho \|\nabla f(x)\|^2$ for all $x \in \R^n$;
    (iii) interpolation: if $x^*$ is a minimizer of $f$, then $\nabla F(x^*, \xi) = 0$ for each $\xi \in \Xi$;
    (iv) complexities are measured for nonconvex problems by requiring $\frac{1}{K}\sum_{k=0}^{K-1} \E[\|\nabla f(x_k)\|^2] \le \epsilon^2$, and for convex problems by requiring $\frac{1}{K}\sum_{k=0}^{K-1} \E[f(x_k) - f(x^*)] \le \epsilon$ (suppose that $r \equiv 0$); (v) for \cite{Tsukada2024}, $\bar{s}_{\alpha,\beta,L} \triangleq \frac{1+\sqrt{1+8\beta(1-\alpha)}}{2L}$, where $\beta\in(1/4,1)$ and
$\alpha\in(0,1-1/(4\beta))$ are linesearch parameters.
Under these conditions, $\bar{s}_{\alpha,\beta,L}<2/L$.
\end{table}

\noindent \textbf{Stochastic optimization.} In stochastic settings, the
situation is significantly more complicated. For instance, when considering SGD with unbiased oracles, constant stepsize schemes with steplengths of $\mathcal{O}\left(\frac{1}{L}\right)$ allow for convergence to a neighborhood of the stationary point~\cite[Theorem 4.8]{Bottou2018}, i.e. $\E[\, \tfrac{1}{K} \sum_{k=0}^{K-1} \left\|\nabla
f(x_k)\right\|^2\,] \le \sigma^2 + \mathcal{O}(1/K)$, where $\sigma^2$
denotes a uniform bound on the variance of the gradient estimator $\nabla F(x,\bxi)$. Asymptotic
convergence~\cite{Bottou2018,Bertsekas2000}  is guaranteed by leveraging a
diminishing stepsize sequence $\{t_k\}$, where $\sum_{k=0}^{\infty} t_k = \infty$
and  $\sum_{k=0}^{\infty} t_k^2 < \infty$. If $D$ denotes a bound associated with initial suboptimality, Ghadimi and
Lan~\cite{Ghadimi2013} employ constant stepsizes bounded by $\min\{\tfrac{1}{L},
\tfrac{D}{\sigma \sqrt{K}}\}$ to obtain an iteration and sample-complexity
of $\mathcal{O}(\epsilon^{-4})$ to ensure an $\epsilon$-solution. Again, these schemes require knowledge of $L$, employ non-increasing stepsize sequences, and are generally viewed as too conservative in practice even if $L$ is available.

\medskip

\noindent \textbf{Adaptive steplength schemes.} Adaptive schemes can be loosely grouped as follows. 

\noindent (a) \emph{Step search with probabilistic oracles.}
In~\cite{Cartis2018}, under deterministic objectives with noisy gradient
estimates that are sufficiently accurate in a probabilistic sense, the method proceeds as follows. When the
Armijo condition fails, the stepsize is reduced and newly-generated gradient
estimates are employed to check  the Armijo condition; else the candidate  point
becomes the new iterate and the stepsize is increased. Since independent samples are used across oracle calls rather than reusing a fixed set of realizations, this procedure is referred to as a \emph{step-search method}. The method provably
takes at most $\mathcal{O}(\epsilon^{-2})$ iterations  to produce an
$\epsilon$-solution. Extensions to stochastic objectives~\cite{Paquette2020}
as well as weakened Armijo conditions~\cite{Berahas2021,Jin2021,Jin2024,Jin2025} have
been considered, in which the function estimates at the current point
and the trial point are defined using two independent random seeds;
consequently, the Armijo condition is relaxed by adding at least two
multiples of the objective noise level. This oracle permits estimates with
large bias on low-probability events and is more appropriate for black-box
settings in which the random seed defining the oracle cannot be controlled. However, such an avenue necessitates that the oracle abide by suitably defined probabilistic requirements, in contrast with standard first/zeroth-order oracles. 

\noindent (b) \emph{Line search
with unbiased stochastic oracles requiring common random numbers.}
This class of methods uses unbiased stochastic objective and gradient
estimates, allowing the same
sampled realization to be reused at the current point and trial points via common random numbers (CRN) framework.
Hence, an Armijo-based backtracking linesearch can be performed on the sample
function analogously to the deterministic setting and terminates after
finitely many backtracking steps. CRN oracles are readily available across a breadth of stochastic optimization and machine learning pipelines, since a batch can be cached and reused during an Armijo line search at negligible cost. However, this assumption may be unsuitable for settings such as privacy-preserving federated learning. In (i) nonconvex settings, the
first linesearch method using CRN oracles was proposed in
\cite{Vaswani2019}, which requires additional assumptions, including
the strong growth assumption of the form
$\E[\|\nabla F(x, \bxi)\|^2] \le \rho \|\nabla f(x)\|^2$ for all $x \in \R^n$
and requires steplengths to be bounded by $\tfrac{2}{\rho L}$ for $\rho \ge 1$.
Under these conditions, their method achieves $\mathcal{O}(\tfrac{1}{K})$
convergence rate to first-order stationarity. Without imposing the strong growth condition, yet under the stepsize cap $s \le \bar{s}_{\alpha,\beta,L} \triangleq \frac{1+\sqrt{1+8\beta(1-\alpha)}}{2L}$, where $\beta\in(1/4,1)$ and $\alpha\in(0,1-1/(4\beta))$ are linesearch parameters, the analysis in {\cite{Tsukada2024}} derives a convergence rate of the form 
$\mathcal{O}(\tfrac{1}{K}) + \mathcal{O}(\tfrac{1}{N})$, where $N$ denotes the fixed
sample size per iteration.  In (ii) convex regimes, \cite{Vaswani2019,Vaswani2025}, expected
suboptimality for an averaged iterate diminishes at $\mathcal{O}(\tfrac{1}{K})$,
under an {\em interpolation condition}; such a condition requires that $f$ is
convex, admits a finite-sum structure, and its minimizer is also the
minimizer for each component $F(\bullet,\xi_i)$, an assumption that holds in select
settings but is far less applicable in more general stochastic optimization regimes.  Jiang and Stich~\cite{Jiang2023} derive a rate of
$\mathcal{O}(\tfrac{1}{K} + \tfrac{1}{\sqrt{K}})$, when this interpolation
condition is relaxed. These methods have been further studied in
\cite{Galli2023} by incorporating  the Polyak stepsize to set the initial stepsize
and in \cite{Fan2023} for a bilevel optimization setting. While the strong growth condition (nonconvex settings), the restriction $s \le 2/L$ (assuming $L$ is known), and the interpolation condition (convex settings) may hold for a select set of learning problems, we believe that such conditions hold rarely, if at all, in other settings. Our method does not necessitate imposing such strong conditions. Beyond the CRN oracle assumption, we impose boundedness of the true gradient, and the corresponding bound is used only in the analysis. We provide a tabular comparison of related stochastic linesearch schemes  with our schemes \textbf{SLAM} and \textbf{SLAM}$^{\rm con}$ in Table~\ref{tab: review}, emphasizing the requirements and clarifying the distinctions.

\noindent (c) \emph{Stochastic trust-region methods.} Stochastic trust region methods have been
developed~\cite{Chen2018,Blanchet2019} with probabilistic oracles while
improvements in sample complexity were provided in~\cite{Shashaani2018, Ha2025}
by using common random numbers. Convergence guarantees under standard stochastic
first-order oracles have been provided in \cite{Curtis2019,Curtis2022}, albeit
with deterministic sequences to bound the steplength. 

\noindent (d) \emph{Related schemes.} In  \cite{Lan2024}, the authors propose an ``auto-conditioned'' framework wherein  the stepsize is not set via a linesearch but is
adapted based on a steplength scheme which aligns steplengths with the reciprocal of the running maximum of the estimated Lipschitz constants at each
iteration. Apart from the above problem classes, {a stochastic linesearch scheme has also been considered for the pseudomonotone stochastic variational inequality problem by leveraging techniques from empirical process theory~\cite{iusem2019variance}.}

\medskip

\noindent \textbf{Problem-parameter-free gradient-based schemes.}  Most popular gradient-based
schemes necessitate leveraging problem parameters (such as smoothness constants, strong convexity parameter etc.)  in designing convergent schemes that achieve optimal rate and complexity
guarantees~\cite{Beck2014,Beck2017,Carmon2022}. However, such parameters are
rarely available and one needs to often tune the scheme to  achieve
desirable empirical behavior. This has necessitated the development of
``Problem-parameter-free'' schemes that aim at developing  algorithms with little or no prior knowledge of the problem parameters and
consequently, little to no tuning of these parameters. In deterministic optimization,  there is a breadth of
``parameter-free'' gradient-based schemes that  include Armijo-based
linesearch schemes~\cite{nocedal2006numerical} and accelerated counterparts~\cite{Beck2017}, amongst others. In stochastic regimes, designing problem-parameter-free schemes with convergence guarantee is more challenging. This has led to advances in developing SGD  schemes for stochastic optimization~\cite{Carmon2022,orabona2020parameterfree,medyakov2026signsgd} and online learning problems~\cite{loizou2021stochastic,mcmahan2017survey}. Similarly, problem-parameter-free counterparts of Polyak stepsize rules in stochastic regimes~\cite{loizou2021stochastic} have been developed. In fact, problem-parameter-free schemes have assumed relevance for related problems such as  decentralized and bilevel optimization~\cite{li2024problem,zhai2026problem,yan2025problem}. There has also been an effort to develop stochastic linesearch schemes~\cite{Vaswani2019,Tsukada2024}, but these schemes appear to require knowledge of
the Lipschitz constant (and are therefore not problem-parameter-free by our stated definition). 

\medskip

\begin{tcolorbox}
    \textbf{Gap}. Existing stochastic linesearch schemes often rely on
an a priori knowledge of $L$, require   initial steplengths to
satisfy stringent upper bounds given by $\mathcal{O}(1/L)$, and employ
non-increasing steplength sequences. There appear to be no
problem-parameter-free Armijo-based stochastic linesearch schemes with
standard stochastic oracles that can contend with nonconvex settings with
rate and complexity guarantees without imposing either non-increasing
requirements or conservative bounds on steplengths ($\frac{a}{L}$ for some $a > 0$).
\end{tcolorbox}

Motivated by this gap, we present a Stochastic LineseArch  Method (\textbf{SLAM}) reliant on conducting an Armijo-based backtracking linesearch using sampled {objective} and gradient evaluations. Our method requires no problem-dependent parameter as a prior and imposes neither a stepsize cap such as $s \le 2/L$ nor restrictive conditions (such as interpolation or strong growth condition). Our method divides the $K$ iterations into $J(K)+1$ periods, where the $j$-th period has length $p_j$ for $j \in \left\{\, 0, 1, 2, \cdots, J(K)\, \right\}$. At the outset of each period, our method resets the initial steplength to its maximum level $s$ to conduct a linesearch, where $s$ is a user-defined parameter; otherwise, the initial steplength in the linesearch is set to the value of accepted steplength from the prior iteration. By using multiple cycles, the algorithm can periodically explore using a large stepsize while adapting to previously accepted stepsizes within each restart window. Importantly, when the batch size and the period length are set as fractional of $K$, the iteration and sample-complexity for computing an $\epsilon$-stationary point are provably $\mathcal{O}(\epsilon^{-2})$ and $\mathcal{O}\left(\epsilon^{-4}\right)$ for nonconvex function $f(\bullet)$. In fact, when algorithm parameters are selected in accordance with problem parameters,  the specialized counterparts of the sample-complexities closely resemble canonical sample-complexities in terms of their dependence on $L$ and moment bounds.  The claims are subsequently refined under both the PL-condition and convexity. Succinctly, our scheme can be summarized as follows.

\begin{tcolorbox} {\bf Summary of contributions}. Our proposed stochastic linesearch method ({\bf SLAM}) is a proximal stochastic gradient method equipped with an \emph{Armijo-enabled backtracking scheme} that uses stochastic zeroth/first-order oracles and resets initial linesearch stepsize periodically.

\smallskip

\noindent (a) {\em Smooth nonconvex settings.} When $f$ is smooth and
potentially nonconvex, to achieve $\epsilon$-stationarity of the expected
residual, {\bf SLAM} is characterized by  iteration and sample-complexity
bounds of $\mathcal{O}(\epsilon^{-2})$ and
$\mathcal{O}\left(\epsilon^{-4}\right)$, respectively. Importantly, this
guarantee does not necessitate imposing the strong growth condition, allows for
non-monotone steplength sequences,  requiring neither knowledge of $L$ nor
the imposition of bounds of the form $\mathcal{O}(1/L)$ on the initial
steplength. Furthermore, under suitable requirements on sample-size sequence
$\{N_k\}$ and period sequence $\{ p_j\}$, asymptotic convergence in mean
ensues.

\smallskip

\noindent (b) {\em PL and convex regimes.} We refine the above scheme to obtain  sample-complexity guarantees of $\mathcal{O}(\epsilon^{-2})$  for computing an $\epsilon$-optimal solution in both the convex and PL settings. Crucially, the claims in the convex regime do not require the interpolation condition. {In the convex setting, the {\bf SLAM$^{\rm con}$}, variant of {\bf SLAM} is proposed, which allows for the nonsmooth convex component in the objective.}

\smallskip

\noindent (c) {\em Numerical experiments. } We compare the performance of {\bf SLAM} with another stochastic line-search method from \cite{Vaswani2019} as well as \texttt{Adam} and \texttt{SGD} with tuned steplengths, observing that {\bf SLAM}  competes well across all problem types without any need for tuning.
\end{tcolorbox}

Notably, the key to our analysis lies in the observation that the error introduced by the (random) stepsize sequence (arising from the Armijo-enabled linesearch)  can be effectively controlled via steplength resetting frequency  and by leveraging the property of uniform Lipschitz smoothness. Indeed, the error of sufficient decrease in the objective function per iteration in expectation is proportional to the probability of  stepsize change between consecutive iterates. Uniform Lipschitz smoothness together with the linesearch mechanism provides an upper bound for  such a probability within each resetting window. Therefore, by appropriately choosing the resetting period, the canonical complexity guarantees may be recovered. Such an analysis technique can also be leveraged in the convex setting to recover the corresponding canonical complexity result. 

\medskip
\noindent \textbf{Comparing SLAM with prior work.} {In this paragraph, we provide  a comparison between  \textbf{SLAM} and \textbf{SLAM$^{\rm con}$} with methods in Table~\ref{tab: review}.} Comparing
\cite[nonconvex]{Vaswani2019}, \cite{Tsukada2024}, and our \textbf{SLAM}
method in the nonconvex setting, the \emph{strong growth condition} in
\cite[nonconvex]{Vaswani2019} appears to be a highly restrictive 
assumption. It implies that the stochastic gradient $\nabla F(x,\xi)$, and consequently the variance of the gradient estimator, vanishes at every stationary point. In contrast,
\cite{Tsukada2024} and \textbf{SLAM} assume boundedness of the gradient variance. While this
assumption may be restrictive in applications with heavy-tailed data (such
as financial returns), {it holds for a broad range of problems. For example, in finite-sum empirical risk minimization, this assumption holds whenever the iterates remain in a bounded set.}
Second, \textbf{SLAM} requires the \emph{bounded true gradient}
assumption, whereas \cite[nonconvex]{Vaswani2019} and \cite{Tsukada2024} do
not. The latter methods instead impose stepsize caps of $2/(\rho L)$ or
$2/L$, respectively, implying that such schemes require knowledge of $L$. From an analytical standpoint, controlling the noise induced by the random step $-t_kg_k$ is more delicate than in SGD with
prescribed stepsizes, because the backtracking stepsize $t_k$ is a random variable and
depends on the stochastic gradient $g_k$. Imposing an ex-ante bound such as $t_k \le \,2/L$
facilitates the control of this noise, but makes the stepsize conservative (and in some cases, miniscule)
and requires knowledge of $L$, severely limiting the applicability and impeding the performance of such schemes. Instead, in \textbf{SLAM}, we use the bounded
true-gradient assumption to control this noise. This condition holds, for
example, for empirical logistic regression and problems with a bounded
feasible set. For convex problems, \cite[convex]{Vaswani2019} requires imposing the
\emph{interpolation} condition. Note that this condition is also rather restrictive, implying that the stochastic
gradient variance vanishes at a minimizer. Consequently, \cite{Vaswani2019}
does not impose the bounded gradient assumption. Both \cite{Jiang2023} and
\textbf{SLAM$^{\rm con}$} assume a bounded feasible set, a mild condition
when the proximal term $r$ has compact effective domain. Although
\cite{Jiang2023} does not explicitly assume bounded variance, its compact
feasible set and finite-sum assumptions imply the bounded variance, making
its assumptions close to those of \textbf{SLAM$^{\rm con}$}. However, the
iteration complexity in \cite{Jiang2023} is worse than that of
\textbf{SLAM$^{\rm con}$}. Finally, the batch-size column shows that
\cite{Tsukada2024}, \textbf{SLAM}, and \textbf{SLAM$^{\rm con}$} use
$\epsilon$-dependent batch sizes to control gradient noise; consequently,
their sample complexity is a square of the iteration complexity.

\smallskip

The remainder of the paper is organized into five sections. Section~\ref{sec:2} provides some background, outlines the assumptions, and formally presents the {\bf SLAM} algorithm. In Section~\ref{sec:3}, we provide the main convergence theory and specialize it to convex and PL settings in Section~\ref{sec:4}. In Section~\ref{sec:5}, we compare the performance of {\bf SLAM} with a related stochastic linesearch scheme as well as tuned variants of \texttt{Adam} and \texttt{SGD} and conclude in Section~\ref{sec:6}.

\medskip 

\noindent\textbf{Notation} We use $\R$ and $\N$ to denote the set of real
numbers and the set of nonnegative integers, respectively. We use $\dom(r)$
to denote the effective domain of function $r$, i.e., $\dom(r) \triangleq
\{x \in \R^n \mid r(x) < \infty\}$.  $\|\bullet\|$ denotes a general norm,
$\lceil \bullet \rceil$ denotes the ceiling function, and $[\bullet]_{+}$
denotes the function $\max\{\bullet,0\}$. For a convex function $r$, we use
$\partial r(x)$ to denote the subdifferential of $r$ at $x$. We use
${\bf 1}_{A}$ to denote the indicator function of event $A$. For a closed
convex set $X \subseteq \R^{n}$, we use ${\bf 1}_{X}(x)$ to denote the
extended-real-valued indicator function of set $X$, which takes the
value $0$ if $x \in X$, and $+\infty$ otherwise.

\section{Background, assumptions, and algorithmic framework} \label{sec:2}
In this section, we provide some background, introduce our assumptions,  and present our stochastic linesearch  framework.

\subsection{Background and assumptions}
We impose the following assumptions on $\phi$.  

\begin{assumption}\label{ass:prox} \em
The function $\phi$ in \eqref{eq:composite_opt} satisfies the following properties.
\begin{enumerate}
    \item[(i)] The function $r:\R^n \to (-\infty, \infty]$ is proper, closed, and convex. Additionally, $r$ is $L_r$-Lipschitz continuous over $\dom(r)$.
    \item[(ii)]  The function $f$ is proper and closed, $\dom(f)$ is convex, $\dom(r) \, \subseteq \, \mathrm{int(\dom(}f))$, and $f$ is continuously differentiable on $\mathrm{int(\dom(}f))$.  Its gradient map $\nabla f: \mathrm{int(\dom(}f)) \to \R^n$, denoted by $g(\bullet) \triangleq \nabla f(\bullet)$, is uniformly bounded on $\dom(r)$; i.e., there exists $B_g > 0$ such that $\left\|\, g (x) \, \right\| \, \le \, B_g $ for any $x \, \in \, \dom(r)$.  
    \item[(iii)] The function $\phi$ is lower bounded by $\phi_{\inf}$ on $\dom(r)$, where $\phi_{\inf} \, \triangleq \, \inf \left\{ \, \phi(x) \, \mid \, x \in \dom(r)\, \right\}$.$\hfill \Box$
\end{enumerate}
\end{assumption}

\medskip
We now impose requirements on the the integrand of the expectation $\mathbb{E}\left[\, F(x,\bxi) \,\right]$. Specifically, we require that  for any $\xi \, \in \, \Xi$, the function $F(\bullet,\xi)$  satisfies suitable smoothness requirements.

\medskip

\begin{assumption}\label{ass: F} \em
    For every $\xi \in \Xi$, $F(\bullet,\xi): \mathrm{int(\dom(}f)) \to \R$ is continuously differentiable and its gradient is Lipschitz continuous with constant $L$ on $\dom(r)$. $\hfill \Box$
\end{assumption}

\medskip

Several aspects of this assumption necessitate comment. First, we may instead require that $\nabla_x F(\bullet,\xi)$ is $\tilde{L}(\xi)$-Lipschitz continuous and under a suitable integrability requirement on $\tilde{L}(\bxi)$, we may obtain an analog of the above result. For purposes of simplicity, we directly assume that $F(\bullet,\xi)$ is $L$-smooth for any $\xi \, \in \, \Xi$. Second and more importantly, we never need to know $L$ for purposes of the algorithm. This is in sharp contrast with most existing SGD-based schemes and their variants, where the theoretical guarantees are predicated on both the knowledge of $L$ and its subsequent usage in the selection of the steplength. For instance, a common choice in constant steplength regimes  emerges in the form of choosing $\gamma = \frac{a}{L}$, where $a$ is a suitably defined positive scalar. Such choices have profound implications in the empirical behavior beyond the onerous informational burden of knowing such a constant, often highly unlikely in any practical setting. In particular, such steplengths are aligned with the global Lipschitz constant and can be exceedingly small and may have a debilitating impact on the performance of the scheme. Our scheme does {\bf not} necessitate knowledge of the Lipschitz constant $L$, being a linesearch scheme.\\

Next, we discuss the requirements on the stochastic zeroth and first-order
oracle. In particular, we assume the existence of stochastic zeroth and
first-order oracles that given $x$, provide $F(x,\xi)$ and $\nabla
F(x,\xi)$, respectively. The conditional unbiasedness and conditional
boundedness of the second moment are imposed on our stochastic zeroth- and
first-order oracles, as captured by the following assumption where our prescribed first/zeroth-order oracles are defined.

\begin{assumption}[Stochastic zeroth and first-order oracles] \label{ass: var}
    For any given realization $\tilde{\xi} \in \Xi$ and $x \in \mbox{dom}(r)$, the oracles provide $F(x,\tilde{\xi})$ and $\nabla F(x,\tilde{\xi})$. In addition, the oracle also provides $F(\bullet,\tilde{\xi})$, i.e., the  oracle provides $F(u,\tilde{\xi})$ for any $u \, \in \, \mbox{dom}(r)$.
Moreover, for any $x \in \dom(r)$,  there exists a constant $\sigma > 0$ such that the following hold almost surely.
    \begin{align*}
        \E\left[\, F(x,\bxi)\mid x \, \right] \, = \, f(x), \,\, \E\left[\, \nabla F(x,\bxi)\mid x \,\right] \, = \, g(x), \,\, \text{and} \,\, \E\left[\, \| \nabla F(x,\bxi) - g(x) \|^2 \mid x\, \right] \, \le \, \sigma^2. 
    \end{align*} $\hfill \Box$
\end{assumption}

Observe that the above framework employs a CRN framework, akin to the schemes presented in~\cite{Vaswani2019,Jiang2023,Tsukada2024}. 
 The CRN oracle supports two types of evaluations under the same realization $\xi \in \Xi$, (i) joint zeroth- and first-order estimates $F(x,\xi)$ and $\nabla F(x,\xi)$ at any $x\in\dom(r)$, and (ii) zeroth-order estimates $F(u,\xi)$ at multiple points $u\in\dom(r)$. This differs from a standard black-box oracle, which uses a fresh realization $\xi$ to generate a zeroth- or first-order estimate each time it is queried. (cf.~\cite{Jin2021}). We now note that under the prior assumptions suitable smoothness requirement is induced.
  
\begin{remark}\label{rmk: 1}
    Under Assumptions \ref{ass:prox}--\ref{ass: var}, the following hold. 

   \noindent (i) The function $f$ is $L$-smooth over $\dom(r)$, since for any $(x, y) \in \dom(r) \times \dom(r)$,
              \begin{align*}
                  \left\|\, g(x) - g(y)\, \right\| & \, =\, \left\| \, \E\left[ \, \nabla F(x, \bxi)\, \right] - \E\left[\, \nabla F(y, \bxi)\, \right]\, \right\|                           \,  \le \, \E \left[\left\|\, \nabla F(x, \bxi) - \nabla F(y, \bxi) \,\right\|\, \right] \, \le \, L\left\|\, x - y\, \right\|.
              \end{align*}
    \noindent (ii)  Since $f$ and $F(\bullet, \xi)$ are $L$-smooth over $\dom(r)$, for any $(x, y) \in \dom(r) \times \dom(r)$, 
              \begin{align}
\begin{aligned}
\label{rmk: lipschitz}
                  |f(y) - f(x) - g(x)^\top(y-x)| &\le \tfrac{L}{2} \|y - x\|^2 \\
                  |F(y, \xi) - F(x,\xi) - \nabla F(x, \xi)^\top(y-x)| &\le \tfrac{L}{2} \|y - x\|^2.
              \end{aligned}
\end{align} $\hfill \Box$
\end{remark}

 We proceed to consider sequences generated by Algorithm~\ref{alg: sls pk prox} and will provide a comprehensive discussion subsequently. We use $\tilde{f}_k$ and $\tilde{g}_k$ to denote the sample-mean objective and gradient at iteration $k = 0, 1, \hdots$, possibly using a mini-batch of size $N_k$ of i.i.d samples $\xi_{k,1},\cdots, \xi_{k,N_k}$, i.e.,
\begin{align}\label{sampledfg}
    \tf_k(\bullet) \, \triangleq \, \frac{1}{N_k} \sum_{i = 1}^{N_k} F(\bullet,\xi_{k,i})\quad \text{ and } \quad \tg_k(\bullet) \, \triangleq \, \frac{1}{N_k} \sum_{i = 1}^{N_k} \nabla F(\bullet,\xi_{k,i}).
\end{align}
Correspondingly, $\tilde{\phi}_k$ is defined as $\tilde{\phi}_k(\bullet) \triangleq \tf_k(\bullet) + r(\bullet)$.
    It follows from Assumption \ref{ass: var}, that, given $x_k \in \dom(r)$, the following hold almost surely.
    \begin{align} \label{eq: sample var}
        \E\left[\, \tf_k(x_k)\, \big|\, x_k\, \right] \, & = \, f(x_k),\  \E\left[\tg_k(x_k)\, \big|\, x_k\right] = g(x_k), \text{ and}\;  \\
        \E\left[\, \|\tg_k(x_k) - g(x_k)\|^2\, \big|\, x_k\, \right] \, & \le \, \frac{\sigma^2}{N_k}. \,\, 
    \end{align}

We use proximal gradient-type method to deal with the nonsmooth function $r(\bullet)$ in \eqref{eq:composite_opt}. For a constant $t > 0$, the proximal operator of $t\times r(\bullet)$ defined as
\begin{align*}
    \prox_{tr}(x) \triangleq  \argmin{y \in \R^n} \left\{ r(y) + \frac{1}{2t} \| y - x \|^2 \right\} = \argmin{y \in \R^n} \left\{ tr(y) + \frac{1}{2} \| y - x \|^2 \right\}.
\end{align*}
For a given constant $t \in (0, \infty)$, $\prox_{tr}(x)$ is a singleton for any $x \in \R^n$ \cite[Theorem 6.3]{Beck2017}. The following lemma is established in \cite{Beck2017} for general proper, closed and convex function. We restate it here for completeness.
\begin{lemma}\em \cite[Theorem 6.39]{Beck2017}\label{lem.prox}
    For any $x, u \in \R^n$, and $t > 0$, the following claims are equivalent:\\
        (i) $u \, = \, \prox_{tr}(x)$;
        (ii) $x - u \, \in \, t \partial r(u)$;
        (iii) $(x-u)^\top (y-u) \, \le \, t r(y) - t r(u)$ for any $y \, \in \, \R^n$. $\hfill \Box$
\end{lemma}

\subsection{Algorithmic description}

Our algorithm is inspired by the decades of success seen in the leveraging of globalization strategies such as linesearch and trust-region techniques within large-scale nonlinear programming solvers such as \texttt{SNOPT}~\cite{Gill2005}, \texttt{KNITRO}~\cite{Byrd2006}, amongst others. In contrast, when contending with expectation-valued problems, the challenge is that the directions afforded by stochastic gradients may be quite poor. Consequently, conducting a linesearch using a poor direction (not necessarily a descent direction) may prove less beneficial. We address this concern using mini-batch gradients with a \emph{prescribed} batch size given by either a constant of order $\Theta(K)$, where $K$ is the total number of iterations, or the increasing sequence $\{k+1\}$, yielding accurate or increasingly accurate gradient estimates, respectively. Second, we introduce a sequence of cycles where the steplength sequence is non-increasing within a cycle but is reinitialized to the maximum steplength of $s$ at the beginning of a new cycle. As a consequence, it can be further observed that steplength sequences do not decay to zero as is traditionally the case in classical SGD theory. We now present our scheme in Algorithm~\ref{alg: sls pk prox} and  highlight some distinguishing features of our scheme.
\begin{enumerate}
    \item \underline{\em Conducting linesearch with a fixed set of  samples.} At each iteration $k$, a batch of samples, denoted by $\{\xi_{k,j}\}_{j=1}^{N_k}$, is generated a priori and the objective and gradient estimates used in linesearch are all evaluated using this fixed batch. In effect, the search process and sampling schemes are distinct and non-overlapping; i.e., while searching for an acceptable steplength, we do not sample. This approach ensures that the linesearch is well-defined; namely, the \texttt{while} loop in Line~\ref{line.backtrack} of Algorithm~\ref{alg: sls pk prox} will terminate after a finite number of reductions in $t_k$. This is analyzed in Lemma~\ref{lm.prox:t_lb}. 
    \item \underline{\em Periodically resetting the initial stepsize.} With an input sequence $\{p_j\}$ where $p_j \in \N\setminus\{0\}$, the $K$ iterations is divided into $J(K)+1$ cycles, where $J(K)$ is defined in \eqref{def:JK}. The period length sequence $\{p_j\}$ is not chosen adaptively but is a \emph{prescribed input} to the algorithm. We preview several admissible choices here: a constant sequence whose value is a fixed fraction of $K$, or a geometric or arithmetic sequence independent of $K$. The formal conditions guaranteeing convergence are given in Lemma~\ref{lem.pj sum}, Theorem~\ref{thm.nonaymptotic constant nk pj}, and Theorem~\ref{thm.aymptotic}.
    At the beginning of each cycle, the initial stepsize for backtracking linesearch is reset to $s$, while within each cycle, the initial stepsize is set to be the accepted stepsize from the prior iteration. Consequently, the stepsize sequence within each restart period is non-increasing and may increase when a new period is started. Figure~\ref{fig:cycle} demonstrates an example of the accepted stepsizes and the associated periods. 

\begin{figure}[htb]
\centering
\includegraphics[width=5in]{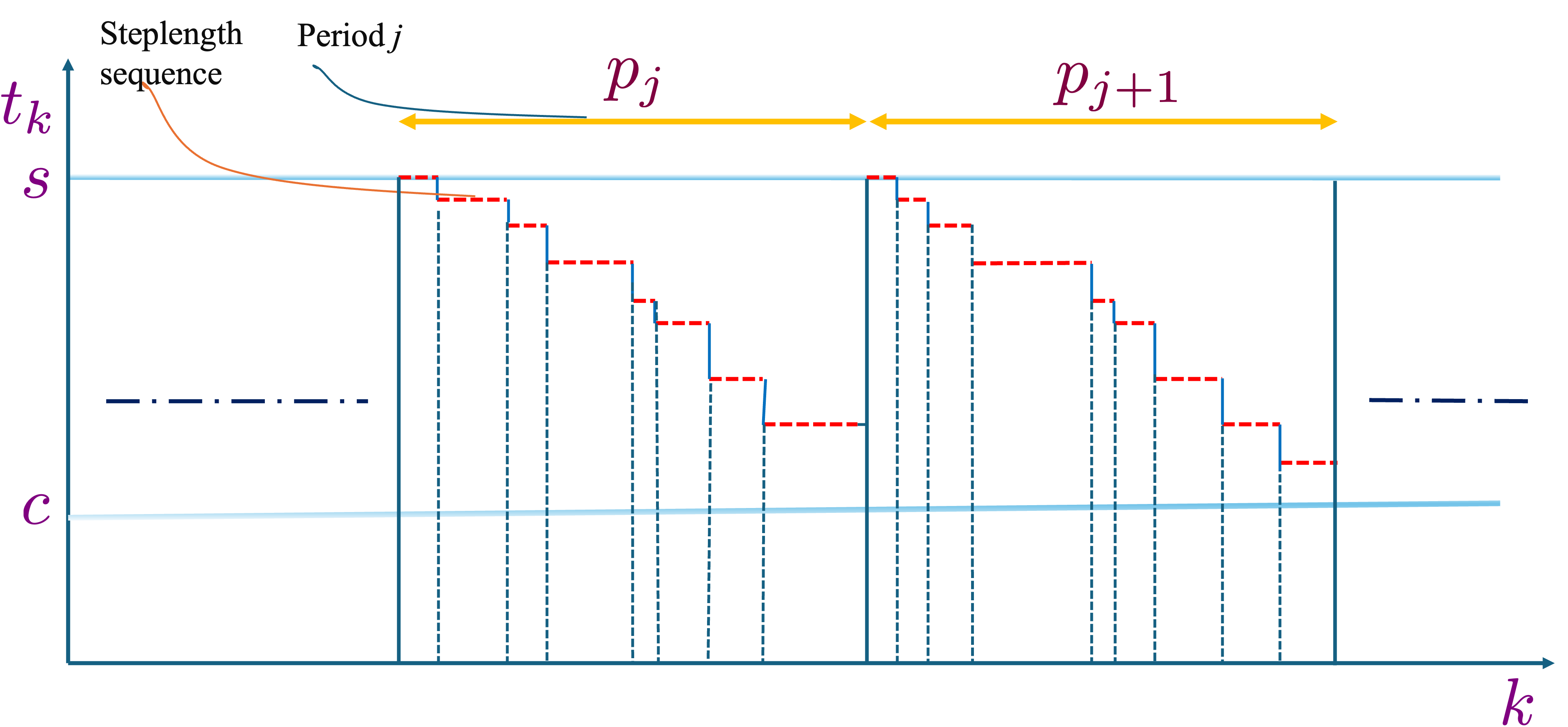}
\caption{Example of period sequence $\{p_j\}$ and steplength sequence $\{t_k\}$. The constant $c$ is a theoretical lower bound on the stepsize $t_k$, as established in Lemma~\ref{lm.prox:t_lb}. The existence of this lower bound follows from the CRN oracle and $L_i$-smoothness assumptions.}
\label{fig:cycle}
\end{figure}
    This strategy allows the algorithm to periodically explore using a large stepsize while remaining adapting to previously accepted stepsizes within the prior iteration within each restart window. This is motivated by the observation that through the linesearch process, the steplength adapts to the curvature of the problem in a given iteration and it may help to retain that alignment in the subsequent iteration. As shown in Figure~\ref{fig:cycle}, the stepsize sequence $\left\{ \, t_k \, \right\}$ within a period is nonincreasing. Notably, when there is only a single cycle (i.e., $p_0 = K$), the initial stepsize in the linesearch is never reset to $s$ again and is always set by $t_{k-1}$, resulting in a non-increasing stepsize sequence. This special case is the scheme in \cite{Vaswani2019} with \texttt{opt=0} in the \texttt{reset} options. However, the analysis in \cite{Vaswani2019} requires the stepsizes to be bounded by $\tfrac{2}{\rho L}$ for some $\rho > 1$ and the strong growth condition; neither requirement is needed in our analysis and it bears reminding that the analysis in \cite{Vaswani2019} necessitates knowledge of $L$ and may lead to miniscule steps if the global Lipschitz constant $L$ is massive. 
    \item \underline{\em Armijo linesearch condition for proximal gradient methods}. At iteration $k \in \N$, with current iterate $x_k \in \R^n$, trial stepsize $t_k \in (0, \infty)$, and mini-batch stochastic gradient $\tg_k(x_k) \in \R^n$, we compute the candidate of new iterate as follows,
\begin{align*}
    x_{k}(t_k) \, \gets \, \prox_{t_k r}\left(\, x_k - t_k \tg_k(x_k)\, \right) = \argmin{y \in \R^n}  \left\{ \, r(y) + \frac{1}{2t_k} \left\| \, x_k - t_k \tg_k(x_k) - y \right\|^2 \, \right\}.
\end{align*}
We then test the following generalized Armijo condition  \eqref{eq.prox.armijo} (see \cite{Bertsekas1976,Beck2017}), where $\alpha \in (0,1)$ is an input algorithm parameter. If this condition is violated, we then reduce $t_k$ by a factor of $\beta \in (0,1)$ and recompute $x_k(t_k)$, repeating this process until the condition is satisfied.
\begin{align}
    \tilde{\phi}_k(x_{k}(t_k) ) - \tilde{\phi}_k (x_k) \le - \frac{\alpha}{t_k} \|x_k - x_k(t_k)\|^2. \label{eq.prox.armijo}
\end{align}
\end{enumerate}

Recall that a point $x^* \in \dom(r)$ is a stationary point of \eqref{eq:composite_opt} if and only if $- \nabla f(x^*) \in \partial r(x^*)$. For a point $x \in \R^n$ and $t > 0$, we introduce the following residual function, akin to the one described in \cite{Beck2017}: 
\begin{align} \label{eq.G.prox}
    G_{t}(x) \, \triangleq \, \frac{1}{t}\left(\, x - \prox_{tr}(x - t g(x))\, \right).
\end{align}
Correspondingly, for a stochastic gradient mapping associated with a sample-size $N_k$, as denoted by $\tg_k(\bullet)$, and $t > 0$, we define the residual $\tilde G^k_{t}$ as 
\begin{align}
    \tilde G^k_{t}(x) \, \triangleq \, \frac{1}{t}\left(\, x - \prox_{tr}(x - t \tg_k(x))\, \right). \label{eq.Gtilde.prox}
\end{align}
Note that with the above notation, the stochastic variant of the Armijo condition \eqref{eq.prox.armijo}  at step $k$ requires that the acceptable steplength $t_k$ satisfies a suitably defined reduction in sampled function value, given a prescribed scalar $\alpha > 0$. This condition can be equivalently stated as follows, by leveraging the residual function $\tilde{G}^k$, defined in  \eqref{eq.prox.armijo.G}. 
\begin{align} \label{eq.prox.armijo.G}
    \tilde \phi_k(x_{k}(t_k)) - \tilde \phi_k(x_k) \, \le \, -\alpha t_k \left\|\, \tilde G^k_{t_k}(x_k)\, \right\|^2.
\end{align}
In the special case of projected gradient methods, $G_t$ is often referred to as the generalized projected gradient~\cite{Ghadimi2016}. It can be readily seen that when $r(\bullet) \equiv 0$,
$\tilde G_t^k(x_k)$ and $G_t(x_k)$ reduce to the sampled
gradient $\tg_k(x_k)$ and the true gradient $g(x_k)$,
respectively. 
The following lemma establishes a relationship between $\|G_t(x)\|$ and the stationary condition. It was established in \cite{Beck2017} and we restate it here for completeness.
\begin{lemma}\em \cite[Theorem 10.7(b)]{Beck2017}
    For any $x^* \in \mathrm{int}(\dom(f))$ and $t > 0$, $G_t(x^*) = 0$ if and only if $x^*$ is a stationary point of \eqref{eq:composite_opt}, i.e., $-\nabla f(x^*) \, \in \, \partial r(x^*)$. $\hfill \Box$
\end{lemma}

We will derive convergence guarantees and rates for the stationarity metric  $\|G_s(x_{R_K})\|$, where $R_K$ denotes an index randomly selected via a uniform distribution from
$\left\{\, 0,1,\cdots,K-1\, \right\}$, while $s$ denotes the user-defined initial steplength. Finally, we define the history $\sF_k$  generated by the Algorithm~\ref{alg: sls pk prox} at iteration $k$ as 
\begin{align*}
    \sF_0 \, \triangleq \, \{x_0\}\, \text{ and }\,\, \sF_k \, \triangleq \, \left\{ \, x_0,\, (\xi_{0,j})_{j=1}^{N_0}, (\xi_{1,j})_{j=1}^{N_1}, \cdots, (\xi_{k-1,j})_{j=1}^{N_{k-1}} \, \right\}, k \ge 1.
\end{align*}

\begin{tcolorbox}
\begin{algorithm}[H]\caption{Stochastic LineseArch  Method ({\bf SLAM})}\label{alg: sls pk prox}
    \begin{algorithmic}[1]
        \Require $x_0,~\alpha \in (0,1),~\beta \in (0,1),~\{p_j\}$ where $p_j \in \N$, $s \in {(0, \infty)},~\{N_k\}$ where $N_k \in \N$.
        \State Set $j \gets 0$; Set $d \gets 1$;
        \For{ $k\in\{0, 1, \ldots, K-1\}$}
        \State Generate samples $\{\xi_{k,j}\}_{j=1}^{N_k}$ and define $ \tf_k(\bullet)$ and $\tg_k(\bullet)$ as in \eqref{sampledfg};
        \If{$k = 0$ or $d = p_j$} \label{line.start set initial t}
        \State Set $t_{k}^{\mathrm{init}} \gets s$; \label{line.set_tk start cycle}
        \State Set $d \gets 1$, $j \gets j + 1$;
        \Else
        \State Set $t_{k}^{\mathrm{init}} \gets t_{k-1}$; \label{line.set_tk within cycle}
        \State Set $d \gets d + 1$;
        \EndIf \label{line.end set initial t}
        \State Set $t_k \gets t_{k}^{\mathrm{init}}$;
        \While{ {\eqref{eq.prox.armijo} is not satisfied}} \label{line.backtrack}
        \State Set $t_k \gets \beta t_k$; \label{line.prox backtrack}
        \EndWhile
        \State Update $x_{k+1} \gets x_k(t_k)$; 
        \EndFor
    \end{algorithmic}
\end{algorithm}

\end{tcolorbox}
\section{Convergence analysis} \label{sec:3}
In this section, we first state and prove a set of useful Lemmas in Section~\ref{sec:supporting lemmas}, which then pave the way for the development of rate statements and iteration/sample-complexity guarantees in Section~\ref{subsec.nonaymptotic}.

\subsection{Supporting lemmas}\label{sec:supporting lemmas}
The first lemma derives a deterministic lower and upper bound for the (random) sequence of steplengths $\left\{\, t_k \, \right\}$ generated by Algorithm~\ref{alg: sls pk prox}. The proof follows an approach used for deterministic settings in~\cite[Lemma 4.23]{Beck2014}. 

\medskip

\begin{lemma}[Bounds on $t_k$]\label{lm.prox:t_lb}\em
    Consider the sequences generated by Algorithm~\ref{alg: sls pk prox} and suppose that Assumptions~\ref{ass:prox}, \ref{ass: F}, and \ref{ass: var} hold. Then, for any $k \in \N$, 
    \begin{align}
        \min\left\{\, s,\frac{2(1-\alpha)\beta}{L}\, \right\} \, \triangleq \,  c \, \le \, t_k \, \le \, s. \label{eq:prox.t_bound}
    \end{align}
\end{lemma}
\begin{proof}
The upper bound requirement, $t_k \le s$ for all $k \in \N$, follows immediately from Lines \ref{line.set_tk start cycle}, \ref{line.set_tk within cycle}, and \ref{line.prox backtrack} of Algorithm~\ref{alg: sls pk prox}. Note that within each cycle $j \in \N$, the stepsizes $\{t_k\}$ for  are non-increasing due to Lines \ref{line.set_tk within cycle} and  \ref{line.prox backtrack} in Algorithm~\ref{alg: sls pk prox}.\\

\noindent In establishing the lower bound, we
observe that  if the initial stepsize $s$ satisfies the Armijo
condition \eqref{eq.prox.armijo}, then $t_k = s$. Otherwise, the \texttt{while} loop in the backtracking
linesearch  in Algorithm~\ref{alg: sls pk prox}
has at least one inner iteration and the accepted $t_k$ satisfying Armijo condition
\eqref{eq.prox.armijo} would be less than $s$. Therefore, the prior trial steplength, given by $t_k/\beta$,
does not satisfy \eqref{eq.prox.armijo}, i.e.,
    \begin{align}
        \tilde\phi_k\left(\, x_k\left(\tfrac{t_k}{\beta}\right)\, \right) - \tilde\phi_k(x_k) \, > \, - \frac{\alpha \beta}{t_k} \left\|\, x_k - x_k\left(\tfrac{t_k}{\beta}\right)\, \right\|^2, \label{eq.lem1.armijo upper bound}
    \end{align}
    where $\tilde G_{\frac{t_k}{\beta}}^k(x_k)$ is defined in \eqref{eq.Gtilde.prox}. On the other hand, since $\tf_k$ is $L$-smooth (by Assumption \ref{ass: F}), we have that the following inequality holds.
    \begin{align}
        \tf_k\left(x_k\left(\tfrac{t_k}{\beta}\right)\right) - \tf_k(x_k) \le - \tg _k(x_k) ^\top \left(x_k - x_k\left(\tfrac{t_k}{\beta}\right)\right)  + \frac{L}{2}\left \| \, x_k - x_k\left(\tfrac{t_k}{\beta}\right) \, \right \|^2, \label{eq: lem1.f_smooth}
    \end{align}
By Lemma \ref{lem.prox}, for a general $t>0$, we have
\begin{align*}
    &&\left(x_k - t\tg_k(x_k) - x_k(t)\right)^\top \left(x_k - x_k(t) \right) &\le t \left(r(x_k) - r(x_k(t))\right) \\
    \iff && \tfrac{1}{t}\left\|x_k - x_k(t) \right\|^2 - \tg_k(x_k)^\top \left(x_k - x_k(t) \right) &\le r(x_k) - r(x_k(t))
\end{align*}
By substituting $t = \tfrac{t_k}{\beta}$, adding this inequality to \eqref{eq: lem1.f_smooth}, and rearranging, we obtain
\begin{align}
\nonumber    \tilde \phi_k \left(x_k\left(\tfrac{t_k}{\beta}\right)\right) - \tilde \phi_k(x_k)= &\tf_k\left(x_k\left(\tfrac{t_k}{\beta}\right)\right) + r\left(x_k\left(\tfrac{t_k}{\beta}\right)\right) - \tf_k(x_k) - r(x_k) \\
    \le &  \left(\tfrac{L}{2} - \tfrac{\beta}{t_k} \right) \left \| x_k - x_k\left(\tfrac{t_k}{\beta}\right) \right \|^2 \label{eq.lem1.function reduction upper bound}
\end{align}
Combining the above inequality with \eqref{eq.lem1.armijo upper bound}, we obtain
    \begin{align*}
       - \frac{\alpha \beta}{t_k}  \, < \, \frac{L}{2} - \frac{\beta}{t_k} \quad \iff \quad t_k > \frac{2(1-\alpha)\beta}{L}.
    \end{align*}
    The two possibilities lead to $t_k \ge \min\left\{\, s, \frac{2(1-\alpha)\beta}{L}\, \right\}$.
\end{proof}
\begin{remark}
Note that in the result of \eqref{eq.lem1.function reduction upper bound}, the right-hand-side may be positive, i.e., we do not require $t_k \le \frac{2\beta}{L}$ in order to guarantee a positive reduction between $\tilde \phi_k \left(x_k\left(\tfrac{t_k}{\beta}\right)\right)$ and $\tilde \phi_k(x_k)$. $\hfill \Box$
\end{remark}
Given the stepsize sequence $\{t_k\}_{k=0}^{K-1}$ generated by Algorithm~\ref{alg: sls pk prox}, we define the random variable $T(K)$ as 
\begin{align}
    T(K) \triangleq \sum_{k=0}^{K-1} {\bf 1}_{\{t_{k} \neq t_{k-1}\}}\text{ with } t_{-1} \triangleq s, \label{eq.TK}
\end{align}
i.e., $T(K)$ represents the number of times that the steplength changes during $K$ iterations. We have the following lemma that shows that $T(K)$ is bounded from above by a deterministic quantity, beyond an obvious bound that $T(K) \le K$. This bound is derived by utilizing an a priori specified sequence of cycle lengths, denoted by $\left\{ p_j\right\}$, and considering the worst-case of the number of updates of $t_k$ withing any given cycle.

\begin{lemma}\label{lem.T bound prox}
 \em   Consider the sequences generated by Algorithm~\ref{alg: sls pk prox} and suppose that Assumptions~\ref{ass:prox}, \ref{ass: F}, and \ref{ass: var} hold. Then for any $K \ge 0$, the random variable $T(K)$, defined in \eqref{eq.TK}, is bounded from above by a deterministic scalar, as prescribed next. 
              \begin{align*}
    T(K) \, \le \, (J(K)+1) \left\lceil\, \left[\, \log_{1/\beta} \left(\, \frac{sL}{2(1-\alpha)} \, \right)\, \right]_{+} + 2 \, \right\rceil,
\end{align*}
where \begin{align}
    \label{def:JK}
J(K) \, \triangleq \, \min\left\{\, l \in \mathbb{N} \ \biggm| \, \ {\displaystyle \sum_{j=0}^l p_j} \ge K\, \right\},\end{align} i.e., $J(K)+1$ represents the total number of cycles arising from cycle length sequence $\left\{p_j\right\}$ and number of iterations $K$.
\end{lemma}
\begin{proof}
 Consider the first $p_0$ iterations, i.e., the first cycle. We bound the number of times that $t_k$ changes ($t_k \neq t_{k-1}$) for $k \in \{1, \cdots, p_0-1\}$ as follows.
    By Lemma \ref{lm.prox:t_lb}, we know that $c\le {t_k} \le {t_{k-1}}$ for $k \in \{1, \cdots, p_0-1\}$.  Therefore, the set $\{k \in \{1, \cdots, p_0-1\}
\mid {t_{k} < t_{k-1}} \}$ assumes its maximum cardinality when ${t_{0}} = s$ and
${t_{k}} = s\beta^k$ for each $k \in \{1, \cdots, p-1\}$ until $s
\beta^{k_{\max}} \le \frac{2(1-\alpha) \beta}{L}$, where $k_{\max}$ defined as follows in \eqref{def: km}
\begin{align}\label{def: km}
k_{\max} \triangleq \min\left\{\, k \in \{1, \cdots, p_0-1\} \biggm| s\beta^{k} \le \frac{2(1-\alpha)\beta}{L}\, \right\}.                             
\end{align}
If $k \ge k_{\max}$ and $k_{\max}$ does not exceed $p_0-1$, then ${t_{k}}$ stays unchanged since
the Armijo condition always holds. As a result, the maximum number of times that the steplength
$t_{k}$ changes is given by
$k_{\max}$, where
            \begin{align}
                \notag
                  \quad k_{\max} \le \min\left\{\, p_0-1,  \left\lceil\left[\log_{1/\beta} \left(\frac{sL}{2(1-\alpha)} \right)\right]_{+} + 1 \right\rceil \, \right\}.
              \end{align}
              Without loss of generality, the result holds for the $j$-th cycle with $p_0$ replaced by $p_j$.  When we include the possible change at the first iteration in the cycle,  the number of times that $t_k$ changes in the $j$-th cycle is upper bounded by
              \begin{align*}
                  \min\left\{ \left\lceil\left[\log_{1/\beta} \left(\frac{sL}{2(1-\alpha)} \right)\right]_{+} + 2 \right\rceil, p_j \right\} \le \left\lceil\left[\log_{1/\beta} \left(\frac{sL}{2(1-\alpha)} \right)\right]_{+} + 2 \right\rceil.
              \end{align*}
Consequently, when running Algorithm~\ref{alg: sls pk prox} for $K$ iterations, there are $J(K)+1$ cycles, where the result follows.
\end{proof}    

\medskip

The next lemma provides some properties of the true residual and stochastic residual mappings $G_t$ and $\tilde{G}_t^k$, which are defined in \eqref{eq.G.prox} and \eqref{eq.Gtilde.prox}, respectively.

\begin{lemma}\label{lm:bounds_G} \em Consider Algorithm~\ref{alg: sls pk prox}. Suppose that Assumption~\ref{ass:prox} holds and $k \in \N$ and $x \in \dom(r)$.    
\begin{enumerate}[label=(\roman*)]
    \item \label{lem3.i} For any $\eta_1 \ge \eta_2 > 0$, we have that 
    \begin{align} \label{rand-resid1}
        \left\|\, \tilde{G}^k_{\eta_2}(x)\, \right\| \, \ge \, \left\|\, \tilde{G}^k_{\eta_1}(x)\, \right\|
    \end{align}
    and 
    \begin{align}\label{rand-resid2}
        \eta_2\left\|\, \tilde{G}^k_{\eta_2}(x)\, \right\| \, \le \, \eta_1\left\|\, \tilde{G}^k_{\eta_1}(x)\, \right\|.
    \end{align}
    The analogs of \eqref{rand-resid1} and \eqref{rand-resid2} also hold for the true residuals $G_{\eta_1}$ and $G_{\eta_2}$, respectively.
    \item \label{lem3.ii}Suppose $\tilde{g}_k(x) \triangleq g(x) + \tilde{e}_k(x)$. For any $t > 0$,
    \begin{align*}
        - 2\|\tilde{G}^k_t(x) \|^2 & \le -\| \, G_t(x) \, \|^2 + 2 \| g(x) -  \tilde{g}_k(x)\|^2  .  \end{align*}
    \item \label{lem3.iii} For any $t > 0$, we have 
    \begin{align*}
        \|tG_t(x)\| \le t (B_g + L_r).
    \end{align*}
\item \label{lem3.iv} For any $\eta_{\min} \le \eta_1 \le \eta_{\max}$, and $\eta_{\min} \le \eta_2 \le \eta_{\max}$, where $\eta_{\min}$ and $\eta_{\max}$ are some positive constants, one has
\begin{align*}
    \left\| \, \prox_{\eta_1 r}(x - \eta_1 g(x)) - \prox_{\eta_2 r}(x - \eta_2 g(x)) \, \right\| \, \le \,  (\eta_{\max}-\eta_{\min})(L_r + 3B_g).
\end{align*}
\end{enumerate}
\end{lemma}
\begin{proof} 
\ref{lem3.i} follows directly from \cite[Theorem 10.9]{Beck2017} and our definition of $\tilde{G}_t^k(x)$ in \eqref{eq.Gtilde.prox}. Note that $\eta_i$ in our definition of $\tilde G_{\eta_i}^k(x)$ is chosen as the reciprocal of $L_i$ in gradient mapping definition of \cite{{Beck2017}}. Next, we prove \ref{lem3.ii}. For any $x\in \R^{n}$ and $k\in \N$, we have that
    \begin{align*}
        \| \, G_t(x) \, \|^2 & = \| G_t(x) - \tilde{G}_t^k (x) + \tilde{G}_t^k(x) \|^2  \le  2\| G_t(x) - \tilde{G}_t^k (x)\|^2 + 2\|\tilde{G}_t^k(x) \|^2 \\  & \le  2\| g(x) -  \tilde{g}_k(x)\|^2 + 2\|\tilde{G}_t^k(x) \|^2,
    \end{align*}
    where the second inequality follows from the nonexpansivity of $\prox_{tr}(\bullet)$.\\

    \noindent We now prove \ref{lem3.iii}. By denoting $\prox_{tr}(x - tg(x))$ by $x(t)$, by the definition of $\prox_{tr}$, one has $x(t)$ minimizes $r(y) + \frac{1}{2t}\|x - tg(x) - y\|^2$ in $y$. Thus, there exists a subgradient $z(x(t))\, \in \, \partial r(x(t))$ for a minimizer $x(t)$ such that 
\begin{align}
    z(x(t)) + \frac{1}{t}(x(t) - x + tg(x)) = 0. \label{eq.lem3iii1}
\end{align}
Furthermore, since $r$ is convex, one has for any $x \, \in \, \mbox{dom}(r)$, 
\begin{align}
    z(x(t))^\top (x - x(t)) \le r(x) - r(x(t)). \label{eq.lem3iii2}
\end{align}
By definition of $G_t(x)$, one has
\begin{align*}
    & \|tG_t(x)\|^2 = \|x - x(t)\|^2 \overset{\eqref{eq.lem3iii1}}{=} t(x - x(t))^\top (g(x) + z(x(t))) = t g(x)^\top (x - x(t)) + t z(x(t))^\top (x - x(t)) \\
    &\overset{\eqref{eq.lem3iii2}}{\le} t g(x)^\top (x - x(t))+ t (r(x) - r(x(t)) \overset{\mbox{\tiny Ass.~\ref{ass:prox}}}{\le} t B_g \|x - x(t)\| + t L_r \|x - x(t)\| = t^2(B_g + L_r) \left\|\, G_t(x)\, \right\|.
\end{align*}
We now prove \ref{lem3.iv}. Without loss of generality, we assume $\eta_1  \, \ge \,  \eta_2$. {By adding and subtracting terms, invoking the triangle inequality, and leveraging non-expansivity of $\prox_{\eta r}(\bullet)$, we observe that $\|\prox_{\eta_1r}(x - \eta_1 g(x)) - \prox_{\eta_2r}(x - \eta_2g(x))\|$ may be bounded as follows.} 
\begin{alignb}
    &\quad \left\|\, \prox_{\eta_1r}(x - \eta_1 g(x)) - \prox_{\eta_2r}(x - \eta_2g(x)) \, \right\| \\
    &= \left\|\, \prox_{\eta_1r}(x - \eta_1 g(x)) - \prox_{\eta_1r}(x - \eta_2g(x)) + \prox_{\eta_1r}(x - \eta_2g(x)) -  \prox_{\eta_2r}(x - \eta_2g(x))\, \right\|\\
    &\le \|\, (\eta_1 - \eta_2)g(x)\, \|+\left\|\, \prox_{\eta_1r}(x - \eta_2g(x)) -  \prox_{\eta_2r}(x - \eta_2g(x))\, \right\|.\label{eq.lem5.itemiv proof}
\end{alignb}
Consider the second term in the final inequality. By defining $y_1$ and $y_2$ as  $y_1 \triangleq \prox_{\eta_1r}(x-\eta_2g(x))$ and $y_2 \triangleq \prox_{\eta_2r}(x-\eta_2g(x))$, respectively, and utilizing the optimality conditions of convex optimization problem associated with the proximal operator, one has
\begin{align*}
    0 \, \in \,  \partial r(y_1) + \frac{1}{\eta_1} (y_1-(x-\eta_2g(x))) \ \text{ and } \  0 \, \in \, \partial r(y_2) + \frac{1}{\eta_2} (y_2-(x-\eta_2g(x))),
\end{align*}
which, together with convexity of $r$, implies that
\begin{align*}
    \frac{1}{\eta_1}\left((x-\eta_2g(x)) - y_1\right)^\top (y_2 - y_1) &\le r(y_2) - r(y_1), \\
\text{and}\quad    \frac{1}{\eta_2}\left((x-\eta_2g(x)) - y_2\right)^\top (y_1 - y_2) &\le r(y_1) - r(y_2).
\end{align*}
Summing up the two inequalities gives us the following sequence of equivalence statements.
\begin{align*}
  & \frac{1}{\eta_1}\left((x-\eta_2g(x)) - y_1\right)^\top (y_2 - y_1) +   \frac{1}{\eta_2}\left((x-\eta_2g(x)) - y_2\right)^\top (y_1 - y_2)  \le 0 \\
  \iff \quad & \frac{1}{\eta_1}\left\|\, y_1 - y_2\, \right\|^2 + \left( \frac{1}{\eta_2} - \frac{1}{\eta_1}\right) (y_1 - y_2)^\top ((x-\eta_2g(x)) - y_2) \le 0 \\
  \iff \quad & \quad \quad \quad \quad \|y_1 - y_2\|^2 \le \left(\, \frac{\eta_1}{\eta_2} - 1\, \right) (y_1 - y_2)^\top (y_2-(x-\eta_2g(x))) \\
  & \quad \quad \quad \quad \quad \quad \quad \quad \ \ \le \left(\, \frac{\eta_1}{\eta_2} - 1\, \right) \left\|\, y_1 - y_2 \, \right\| \left\|\, y_2-(x-\eta_2g(x))\, \right\|.
  \end{align*}
  Therefore, we have that $\|y_1 - y_2\| \le \left(\, \frac{\eta_1}{\eta_2} - 1\, \right) \left\|\, y_2-(x-\eta_2g(x))\, \right\|.$
  By invoking the definitions of $y_1$ and $y_2$, we obtain the following bound.
\begin{align}
\notag
    \|\prox_{\eta_1r}(x - \eta_2g(x)) -  \prox_{\eta_2r}(x - \eta_2g(x))\| &\le \frac{\eta_1 - \eta_2}{\eta_2} \| \prox_{\eta_2r}(x - \eta_2g(x)) - (x - \eta_2 g(x))  \| \\
    \notag
    & = \frac{\eta_1 - \eta_2}{\eta_2} \| -\eta_2G_{\eta_2}(x) + \eta_2 g(x))  \|\\
    \label{ineq: nn8}
    & \le (\eta_1 - \eta_2) (2B_g + L_r),
\end{align}
where we use the result in \ref{lem3.iii} and $\|g(x)\| \le B_g$ to get the last inequality. The inequality \eqref{eq.lem5.itemiv proof} further can be bounded by
\begin{align}
\notag
    &\quad \left\|\, \prox_{\eta_1r}(x - \eta_1 g(x)) - \prox_{\eta_2r}(x - \eta_2g(x))\, \right\| \\
    \notag
    &\le \left\|\, (\eta_1 - \eta_2)g(x)\, \right\|+\left\|\, \prox_{\eta_1r}(x - \eta_2g(x)) -  \prox_{\eta_2r}(x - \eta_2g(x))\, \right\| \\
    \notag
    & \le (\eta_1 - \eta_2) (L_r + 3B_g) \\
     \notag
    & \le (\eta_{\max} - \eta_{\min}) (L_r + 3B_g),
\end{align}
where the second inequality is from \eqref{ineq: nn8} and $\| g(x) \| \le B_g$.
\end{proof}

\medskip
We introduce an auxiliary iterate, denoted by $\bar{x}_{k+1}$ and defined as 
\begin{align*}
    \bar x_{k+1} \, \triangleq \, \prox_{t_{k-1}r}(x_k - t_{k-1}g(x_k))\quad \text{for all}\quad k \in \N \quad \text{and}\quad t_{-1} = s.
\end{align*}
It may be observed that $\bar x_{k+1}$ is \emph{independent of} the
samples generated at iteration $k$ and is
$\mathcal{F}_k$-measurable.  Therefore, the sequence
$\left\{\bar{x}_{k+1}\right\}$ is adapted to $\left\{\sF_k\right\}$.  This sequence helps us to establish the following proposition, which derives a bound on the expected difference between $\phi(x_{k+1})$ and $\tilde{\phi}_k(x_{k+1})$ for any $k$.
\begin{proposition}\label{lm: errbd prox}\em
    Consider the sequences generated by Algorithm~\ref{alg: sls pk prox} and suppose that Assumptions~\ref{ass:prox}, \ref{ass: F}, and \ref{ass: var} hold. Then the following holds for any $k \in \N$. 
    \begin{align}
        \E \left[\, \phi(x_{k+1}) - \tilde{\phi}_k(x_{k+1})\, \right]
         & \, \le \,  \left(\tfrac{1}{2L} + 3Ls^2\right)\frac{\sigma^2}{N_k} + 3L(s-c)^2 (L_r+3B_g)^2 \bP\{t_{k} \neq t_{k-1}\}. \label{prop1:ineq}
    \end{align}
\end{proposition}
\begin{proof}
    We begin by expressing the term on the left as the sum of the following terms
    \begin{align}
         \E\left[\, \phi(x_{k+1}) - \tilde{\phi}_k(x_{k+1})\, \right] & = \E\left[f(x_{k+1}) - \tf_k(x_{k+1})\right] \notag                                                                                               \\
         = \E & \left[f(x_{k+1}) - f(\bar x_{k+1}) + \tf_k(\bar x_{k+1}) - \tf_k(x_{k+1}) + f(\bar x_{k+1}) - \tf_k(\bar x_{k+1})\right]                    \label{ineq: nn1 prox}
    \end{align}
    Note that by Lipschitz continuity of $g(\bullet)$ and
    $\tg_k(\bullet)$ and Remark \ref{rmk: 1}, we immediately have that
    \begin{align}
        f(x_{k+1}) - f(\bar x_{k+1}) & \le g(\bar x_{k+1})^\top (x_{k+1} - \bar x_{k+1}) + \frac{L}{2} \| x_{k+1} - \bar x_{k+1} \|^2 
        \label{ineq: nn2 prox}
    \end{align}
    and
    \begin{align}
        \tf_k(\bar x_{k+1}) - \tf_k(x_{k+1}) & \le \tg_k(\bar x_{k+1})^\top (\bar x_{k+1} - x_{k+1}) + \frac{L}{2} \| x_{k+1} - \bar x_{k+1} \|^2. 
        \label{ineq: nn3 prox}
    \end{align}
    Therefore, from \eqref{ineq: nn1 prox} and by adding \eqref{ineq: nn2 prox}-\eqref{ineq: nn3 prox}, and using the Young's inequality $( g(\bar x_{k+1}) - \tg_k(\bar x_{k+1}) )^\top ( x_{k+1} -  \bar x_{k+1} ) \, \le \,  \frac{1}{2L}\|g(\bar x_{k+1}) - \tg_k(\bar x_{k+1})\|^2 + \frac{L}{2}\|x_{k+1} -  \bar x_{k+1}\|^2$, we may derive the following bound. 
    \begin{align}
    \notag
        \E\left[\, f(x_{k+1}) - \tf_k(x_{k+1})\, \right] & \le \E\left[\, ( g(\bar x_{k+1}) - \tg_k(\bar x_{k+1}) )^\top ( x_{k+1} -  \bar x_{k+1} ) \, \right] \\
        \notag
        & ~~~+ L \E\left[ \| x_{k+1} -  \bar x_{k+1}  \|^2 \,\right] + \E\left[ f(\bar x_{k+1}) - \tf_k(\bar x_{k+1}) \, \right] \\
        \notag
        & \le \tfrac{1}{2L} \E\left[ \| g(\bar x_{k+1}) - \tg_k (\bar x_{k+1}) \|^2 \right] + \tfrac{3L}{2} \E\left[ \| x_{k+1} -  \bar x_{k+1}  \|^2 \right] \\
        \label{ineq: nn4 prox}
        & ~~~+ \E\left[ f(\bar x_{k+1}) - \tf_k(\bar x_{k+1}) \right].
    \end{align}
The second term in \eqref{ineq: nn4 prox} may be bounded as
follows.
    \begin{alignb}
\quad & \E\left[ \, \| \bar x_{k+1} - x_{k+1} \|^2 \,\right]   \\
         & = \E\big[\, \| \prox_{t_{k-1}r}\left(x_k - t_{k-1} g(x_k)\right) - \prox_{t_{k}r}\left(x_k - t_{k} g(x_k)\right)    \\ 
         & \quad + \prox_{t_{k}r}\left(x_k - t_{k} g(x_k)\right)  - \prox_{t_k r}\left(x_k - t_k \tg_k(x_k) \right) \, \|^2\, \big]                                \\
         & \le \E\big[\, 2\| \prox_{t_{k-1}r}\left(x_k - t_{k-1} g(x_k)\right) - \prox_{t_{k}r}\left(x_k - t_{k} g(x_k)\right)  \|^2  \left({\bf 1}_{\{t_k = t_{k-1}\}} + {\bf 1}_{\{t_k \neq t_{k-1}\}}\right) \big]\\ 
         & \quad + \E\big[\, 2\|\prox_{t_{k}r}\left(x_k - t_{k} g(x_k)\right)  - \prox_{t_k r}\left(x_k - t_k \tg_k(x_k) \right) \, \|^2\, \big]                                \\
         & = \E\big[\, 2\| \prox_{t_{k-1}r}\left(x_k - t_{k-1} g(x_k)\right) - \prox_{t_{k}r}\left(x_k - t_{k} g(x_k)\right)  \|^2 {\bf 1}_{\{t_k \neq t_{k-1}\}}\big] \\
        & \quad + \E\big[\, 2\| \prox_{t_{k-1}r}\left(x_k - t_{k-1} g(x_k)\right) - \prox_{t_{k}r}\left(x_k - t_{k} g(x_k)\right)  \|^2 {\bf 1}_{\{t_k = t_{k-1}\}}\big] \\
         & \quad + \E\big[\, 2\|\prox_{t_{k}r}\left(x_k - t_{k} g(x_k)\right)  - \prox_{t_k r}\left(x_k - t_k \tg_k(x_k) \right) \, \|^2\, \big]                                \\
         & \le 2(s-c)^2(L_r + 3B_g)^2 \E[ {\bf 1}_{\{t_k \neq t_{k-1}\}}] + 0  \\
         & \quad +  \E\big[\, 2\|\prox_{t_{k}r}\left(x_k - t_{k} g(x_k)\right)  - \prox_{t_k r}\left(x_k - t_k \tg_k(x_k) \right)\, \|^2\, \big]\\
         & \le 2(s-c)^2(L_r + 3B_g)^2\mathbb{P}\{t_k \neq t_{k-1}\} + \E\big[\, 2 t_k^2 \| g(x_k) - \tg_k(x_k) \|^2\, \big]   \\
         & \le 2(s-c)^2(L_r + 3B_g)^2\mathbb{P}\{t_k \neq t_{k-1}\} + 2s^2\E\big\| g(x_k) - \tg_k(x_k) \|^2\, \big].   
         \label{eq.xbar x difference}
    \end{alignb}
The second inequality uses Lemma~\ref{lm:bounds_G}\ref{lem3.iv} and the bounds $t_{k-1},t_k \in [c,s]$. The third inequality employs nonexpansiveness of the proximal operator.
By combining \eqref{ineq: nn4 prox} and \eqref{eq.xbar x difference}, one has
\begin{align*}
     \E\left[\, \phi(x_{k+1}) - \tilde{\phi}_k(x_{k+1})\, \right] 
    & \le  \tfrac{1}{2L} \E\left[ \| g(\bar x_{k+1}) - \tg_k (\bar x_{k+1}) \|^2 \right] + 3Ls^2\E\big[\| g(x_k) - \tg_k(x_k) \|^2\, \big]  \\
    & ~~~+ 3L(s-c)^2 (L_r+3B_g)^2 \bP\{t_{k} \neq t_{k-1}\} + \E\left[ f(\bar x_{k+1}) - \tf_k(\bar x_{k+1}) \right] \\
    & \le  \left(\frac{1}{2L} + 3Ls^2\right)\tfrac{\sigma^2}{N_k} + 3L(s-c)^2 (L_r+3B_g)^2 \bP\{t_{k} \neq t_{k-1}\}
\end{align*}
where the second inequality follows from
        \begin{align*}
            \E\left[ \, \| g(\bar x_{k+1}) - \tg_k(\bar x_{k+1}) \|^2 \, \right] & = \E\left[ \, \E\left[ \, \| g(\bar x_{k+1}) - \tg_k(\bar x_{k+1}) \|^2 \mid \sF_k \, \right]\, \right] \le \tfrac{\sigma^2}{N_k}, \\
             \E\left[ \, \| g(x_k) - \tg_k(x_k) \|^2 \, \right] & = \E\left[ \, \E\left[ \, \| g(x_k) - \tg_k( x_k ) \|^2 \mid \sF_k \, \right]\, \right] \, \le \, \tfrac{\sigma^2}{N_k},\\
            \mbox{ and }  \E\left[ \, f(\bar x_{k+1}) - \tf_k(\bar x_{k+1}) \, \right] & = \E\left[\, \E\left[ \, f(\bar x_{k+1}) - \tf_k(\bar x_{k+1}) \mid \sF_k \, \right]\, \right] = 0,
        \end{align*}
        a consequence of observing that $x_k$ and $\bar x_{k+1}$ are $\mathcal{F}_{k}$-measurable and Assumptions~\ref{ass: var}. The conclusion follows.
\end{proof}

The following proposition combines the results obtained from prior results to derive a lower bound on  the expected reduction in function value across any two consecutive iterations.
\begin{proposition}\label{prop: sufficient decrease prox}\em
    Consider the sequences generated by Algorithm~\ref{alg: sls pk prox} and suppose that Assumptions~\ref{ass:prox}, \ref{ass: F}, and \ref{ass: var} hold. For any $k \in \N$, the following holds.
    \begin{equation} \label{eq.sufficient decrease prox}
        \begin{split}
        \E\left[\, \phi(x_{k+1}) \, \right]
     & \le \E\left[\, \phi(x_k)\, \right] - \frac{\alpha c}{2} \E\left[ \,\|G_{s}(x_k)\|^2 \, \right] + \left({\frac{1}{2L} + 3Ls^2} + \alpha c\right)\frac{\sigma^2}{N_k} \\
     & ~~~ + 3L(s-c)^2 (L_r+3B_g)^2\bP\{t_{k} \neq t_{k-1}\}.    
        \end{split}
    \end{equation}
\end{proposition}
\begin{proof}
    Taking unconditional expectations on both sides of \eqref{eq.prox.armijo} which is equivalent to \eqref{eq.prox.armijo.G}, we obtain
\begin{align*}
     \E\left[\, \tilde{\phi}_k(x_{k+1}) \, \right] & \le \E\left[\, \tilde{\phi}_k(x_k)\, \right] - \alpha \E\left[ \, t_k \|\tilde G_{t_k}^k(x_k)\|^2 \, \right] \\
     & =  \E\left[\,\E \left[ \, \tilde{\phi}(x_k)\, \mid \, \sF_k \, \right] \,\right] - \alpha \E\left[ \, t_k \|\tilde G_{t_k}^k(x_k)\|^2 \, \right] \\
     & =  \E\left[\, \phi(x_k)\, \right] - \alpha \E\left[ \, t_k \|\tilde G_{t_k}^k(x_k)\|^2 \, \right] \\
     & \le  \E\left[\, \phi(x_k)\, \right] - \alpha c \E\left[ \,\|\tilde G_{t_k}^k(x_k)\|^2 \, \right] \\
     & \le  \E\left[\, \phi(x_k)\, \right] - \tfrac{\alpha c}{2} \E\left[ \,\|G_{s}(x_k)\|^2 \, \right] + \alpha c \E\left[\, \|g(x_k) - \tg_k(x_k)\|^2 \, \right]  \\
     & \le  \E\left[\, \phi(x_k)\, \right] - \tfrac{\alpha c}{2} \E\left[ \,\|G_{s}(x_k)\|^2 \, \right] + \frac{\alpha c \sigma^2}{N_k},    
\end{align*}
where the first equality is due to $x_k$ being $\mathcal{F}_{k}$-measurable and the unbiasedness assumption, the second inequality holds by $t_k \ge c$, the third inequality follows from Lemma~\ref{lm:bounds_G}\ref{lem3.i}, $t_k \le s$, and Lemma~\ref{lm:bounds_G}\ref{lem3.ii}, while the last inequality is a consequence of Assumption~\ref{ass: var}. Then by Proposition~\ref{lm: errbd prox} and through a rearrangement of terms, we obtain the following sequence of inequalities.
\begin{align*}
     \E\left[\, \phi(x_{k+1}) \, \right] & \le  \E\left[\, \left(\, \phi(x_{k+1}) \,  - \,  \tilde{\phi}_k(x_{k+1}) \, \right) \, \right] + \E\left[\, \phi(x_k)\, \right] - \frac{\alpha c}{2} \E\left[ \,\|G_{s}(x_k)\|^2 \, \right] + \frac{\alpha c \sigma^2}{N_k} \\
     & \overset{ \mbox{ \tiny \eqref{prop1:ineq}}}{\le} \E\left[\, \phi(x_k)\, \right] - \tfrac{\alpha c}{2} \E\left[ \,\|G_{s}(x_k)\|^2 \, \right] + \left(\tfrac{1}{2L} + 3Ls^2 + \alpha c\right)\frac{\sigma^2}{N_k} \\
     &  ~~~+ 3L(s-c)^2 (L_r+3B_g)^2 \bP\left\{t_{k} \neq t_{k-1}\right\}.
\end{align*}
\end{proof}
By combining the prior results, we are now in a position to derive a rate guarantee of the expected residual of a randomly selected iterate. 
\begin{proposition}\label{proposition:main prox}\em
    Consider the sequences generated by Algorithm~\ref{alg: sls pk prox} and suppose that Assumptions~\ref{ass:prox}, \ref{ass: F}, and \ref{ass: var} hold, $K \ge 1$, and $R_K$ is a random variable taking values in $\left\{\, 0, 1, \cdots, K-1 \,\right\}$ with a probability mass function given by $\mathbb{P}\left[ \, R_K = j\, \right] = 1/K$. Then the following holds for any $K\ge 1$.
    \begin{alignb} 
         \E\left[\, \left\|\, G_{s}(x_{R_K}) \,\right\|^2\, \right] &\, \le \, \frac{2( \phi(x_0) - \phi_{\inf} )}{\alpha c K} +   \frac{ \frac{1}{L} + 6Ls^2+ 2\alpha c}{\alpha c K} \sum_{k=0}^{K-1} \frac{\sigma^2}{N_k} \\
         & ~~~+  \frac{6L(s-c)^2 (L_r+3B_g)^2(J(K)+1) \left\lceil\left[\log_{1/\beta} \left(\frac{sL}{2(1-\alpha)} \right)\right]_{+} + 2 \right\rceil}{\alpha c K}  \label{eq.rate bound}
    \end{alignb}
\end{proposition}
\begin{proof}
    By Proposition \ref{prop: sufficient decrease prox} and by summing \eqref{eq.sufficient decrease prox} from $k = 0$ to $K-1$,  we obtain
\begin{align*}
    \phi_{\inf} - \phi(x_0) &\, \le \, \E\left[\, \phi(x_K)\, \right]  - \E\left[\, \phi(x_0)\, \right] \\
    &\le  - \frac{\alpha c}{2} \sum_{k=0}^{K-1}\E\left[\, \left\|\,   G_s(x_k) \, \right\|^2\, \right] +  \left( \, \frac{1}{2L} + 3Ls^2 + \alpha c\, \right) \sum_{k=0}^{K-1}\frac{\sigma^2}{N_k}
         \\
    & ~~~+ 3L(s-c)^2 (L_r+3B_g)^2\sum_{k=0}^{K-1}\bP\left\{\, t_{k} \neq t_{k-1}\, \right\}. 
\end{align*}
It remains to derive a bound on $\sum_{k=0}^{K-1}\bP\left\{\, t_{k} \neq t_{k-1}\, \right\}$, which is provided next. Recall that the random variable $T(K)$ represents the total number of iterations corresponding to $t_k \neq t_{k-1}$ when running Algorithm \ref{alg: sls pk prox} for $K$ iterations. Consequently, we may derive the required probability as follows.
\begin{alignb}
    \sum_{k=0}^{K-1}\bP\left\{\, t_{k} \neq t_{k-1}\, \right\} & = \sum_{k=0}^{K-1}\E\left[\, {\bf 1}_{\{t_{k} \neq t_{k-1}\}}\, \right] = \E\left[\, \sum_{k=0}^{K-1}{\bf 1}_{\{t_{k} \neq t_{k-1}\}}\, \right] = \E\left[\, T(K)\, \right]\\
    & \le (J(K)+1) \left\lceil\, \left[\, \log_{1/\beta} \left(\frac{sL}{2(1-\alpha)} \right)\, \right]_{+} + 2 \, \right\rceil, \label{eq.p to T}
\end{alignb}
where the last inequality follows from Lemma~\ref{lem.T bound prox} and $J(K)+1$ represents the total number of cycles in $K$ iterations which is a deterministic quantity. Therefore, if $R_K$ is a random variable taking values in $\left\{\, 0, 1, \cdots, K-1\, \right\}$ with $\mathbb{P}\left[ \, R_K = j\, \right] = 1/K$, then the result follows from $\E\left[ \, \left\| \, G_{s}(x_{R_K}) \right\|^2\, \right] = \tfrac{1}{K}\sum_{k=0}^{K-1} \E\left[\, \left\|\, G_s(x_k) \, \right\|^2\, \right]$ and the above inequalities.
\end{proof}

From the result in Proposition~\ref{proposition:main prox}, we can establish the result of $\E[ \|  G_{s}(x_{R_{K}})  \|^2]= \mathcal{O}(1/K)$ under certain choices of $\{N_k\}$ and $\{p_j\}$. We summarize these results in the following two subsections.

\subsection{Complexity guarantees}\label{subsec.nonaymptotic}
In this subsection, we analyze the iteration and sample complexities. For a real number $\epsilon > 0$, we define the iteration complexity as the smallest number of iterations $K_\epsilon$ that guarantees
\begin{align}\label{eq.convergence measure}
    \E\left[ \, \left\| \,  G_{s}(x_{R_{K_\epsilon}})  \, \right\|\, \right] \, \le \, \epsilon,
\end{align}
and define the sample complexity as  $S_\epsilon \, \triangleq \, {\displaystyle \sum_{k=0}^{K_\epsilon - 1}} N_k$, representing the minimum number of samples necessary to compute an $\epsilon$ solution satisfying \eqref{eq.convergence measure}.
In the following Lemma~\ref{lem.Nk sum} and ~\ref{lem.pj sum}, we provide the bounds for $\sum_{k=0}^{K-1}\frac{1}{N_k}$ and $J(K) + 1$ for some specific choices of $\{N_k\}$ and $\{p_j\}$. Note that $J(K)$ defined in Lemma~\ref{lem.T bound prox} is a deterministic quantity since $\{p_j\}$ and $K$ are prescribed. We begin by providing some fairly rudimentary bounds on $\sum_{k=0}^{K-1} \frac{1}{N_k}$ for differing choices of the sequence $\{ N_k\}$ and find utilization in our convergence analysis. The proof is relegated to Appendix~\ref{app: prfs0}. 

\medskip

\begin{lemma}\label{lem.Nk sum} \em Consider a non-decreasing positive sequence $\left\{\, N_k \, \right\}$, where $N_k \in \N$ and some $K \in \N \setminus \{1\}$. Then the following holds. 
        (i) If $N_k = \lceil c_b K \rceil $ where $c_b > 0$, then 
        ${\displaystyle \sum_{k=0}^{K-1}} \frac{1}{N_k} \le  \frac{1}{c_b}$;
        (ii) If $N_k = k+1 $ then 
        ${\displaystyle \sum_{k=0}^{K-1}} \frac{1}{N_k} \le  \ln(K) + 1$;
        (iii) If $N_k = (k+1)^b $ for some $b \in (1, \infty)$, then ${\displaystyle \sum_{k=0}^{K-1}} \frac{1}{N_k} \le \frac{b}{b-1}$.
\end{lemma}
\medskip
We now examine the sequence $\left\{ p_j \right\}$ from the standpoint of bounding $J(K)+1$. We consider three settings where the first setting sets $p_j$ as a constant defined by a fraction of the overall iteration limit. We then consider an increasing sequence where $p_j$ increases at an exponential rate with $j$ while the last choice chooses a linear growth in $p_j$. Collectively, the bound on $J(K)+1$ emerges in deriving the complexity guarantee.

\medskip

\begin{lemma}\label{lem.pj sum}
    \em Consider a positive integer $K \in \N \setminus \{1\}$ and a sequence $\left\{p_j\right\}$, where $p_j \in \N$ for every $j$. Then the following hold. 
    \begin{enumerate}[label=(\roman*)]
         \item\label{item.pj1} If $p_j = \lceil c_p K \rceil $ where $c_p \in (0,1]$, then $J(K)+1 \le \frac{1}{c_p}+1$.
         \item\label{item.pj2} If $p_j = 2^jl_0$ where $l_0 \in \N\setminus\{0\}$, then $J(K)+1 = \left\lceil \log_2\left(\frac{l_0 + K}{l_0}\right) \right\rceil = \mathcal{O}\left(\ln(K)\right)$.
         \item\label{item.pj3} If $p_j = l_0 + j$ where $l_0 \in \N\setminus\{0\}$ and $l_0 \le K$, then $$J(K) + 1 =\left\lceil \sqrt{(\tfrac{1}{2} + l_0)^2 + 2(K-l_0)} - (\tfrac{1}{2} + l_0) + 1\right\rceil = \mathcal{O}\left(\sqrt{K}\right)$$.
    \end{enumerate}
\end{lemma}
The proof of Lemma~\ref{lem.pj sum} is relegated to Appendix~\ref{app: prfs0}. We are now in a position to derive a non-asymptotic rate guarantee in terms of $K$, which then facilitates the presentation of complexity guarantees.
\medskip
\begin{theorem}\label{thm.nonaymptotic constant nk pj} \em 
Suppose Assumptions~\ref{ass:prox}, \ref{ass: F}, and \ref{ass: var} hold.
Consider the sequence generated by Algorithm \ref{alg: sls pk prox}, where for a given $K$, $N_k = \lceil c_b K\rceil$, and $p_j = \lceil c_p K \rceil$, where $c_b > 0$ and $c_p \in (0,1]$ as defined in Lemma~\ref{lem.Nk sum}(i) and \ref{lem.pj sum}\ref{item.pj1}, respectively. If $R_K$ is a random index as defined in Proposition~\ref{proposition:main prox},  then the following hold. 
\noindent
\begin{enumerate}[label=(\roman*)]
    \item For any  $K \ge 1$,  \begin{alignb}
        & \quad \E\left[ \, \left\|\, G_{s}(x_{R_K}) \, \right\|^2\, \right] \\
        &\le \frac{2( \phi(x_0) - \phi_{\inf} )}{\alpha c K} +   \frac{ \left(\frac{1}{L} + 6Ls^2 + 2\alpha c\right) \sigma^2}{\alpha c c_b K}  \\
         & \quad \quad +  \frac{6L(s-c)^2 (L_r+3B_g)^2  \left\lceil\left[\log_{1/\beta} \left(\frac{sL}{2(1-\alpha)} \right)\right]_{+} + 2 \right\rceil (c_p + 1)}{\alpha c c_p K} \\
         & = \mathcal{O}\left(\, \frac{ ( \phi(x_0) - \phi_{\inf} ) + (L^{-1} + Ls^2 + c) \sigma^2 /c_b + L (s-c)^2 (B_g + L_r)^2 \left[\ln(Ls)\right]_{+}}{c K} \, \right),\label{eq.rate bound main}
    \end{alignb}
implying that the iteration complexity $K_{\epsilon}$ and sample-complexity $S_{\epsilon}$ that guarantee \eqref{eq.convergence measure} are 
  \begin{align*}   
        K_{\epsilon}  = \mathcal{O}\left(\, 
        \frac{ ( \phi(x_0) - \phi_{\inf} ) + (L^{-1} + Ls^2 + c) \sigma^2 /c_b + L (s-c)^2 (B_g + L_r)^2 \left[\ln(Ls)\right]_{+}}{c \epsilon^2}\, \right)
    \end{align*}
    and
    \begin{align*}
    S_\epsilon = \mathcal{O}\left(\, 
        \frac{ ( \phi(x_0) - \phi_{\inf} )^2 c_b + (L^{-1} + Ls^2 + c)^2 \sigma^4/c_b + L ^2(s-c)^4 (B_g + L_r)^4 \left[\ln(Ls)\right]_{+}^2 c_b}{c^2 \epsilon^4}\, \right).
\end{align*}
\item Suppose, in addition, $s = \frac{2(1-\alpha)\beta\zeta }{L}$ for some constant $\zeta \in (0,1]$. Then, for any $K \ge 1$, 
\begin{alignb}
        \E\left[ \, \left\|\, G_{s}(x_{R_K}) \, \right\|^2\, \right] &\le \frac{L\left(\frac{(\phi(x_0) - \phi_{\inf})}{(1-\alpha)\beta\zeta\alpha}\right)}{K} +   \frac{ \left(\frac{1}{2\alpha(1-\alpha)\beta \zeta} + \frac{12(1-\alpha)\beta \zeta}{\alpha} + 2\right) \sigma^2}{c_b K} \\
        &= \mathcal{O}\left(\, \frac{L(\phi(x_0) - \phi_{\inf}) + \sigma^2/c_b}{K} \, \right) \label{eq.rate bound s 1overL}
    \end{alignb}
and the iteration and sample-complexities that guarantee \eqref{eq.convergence measure} are 
 \begin{align*}   
        K_{\epsilon}  = \mathcal{O}\left(\, 
        \frac{ L( \phi(x_0) - \phi_{\inf} ) + \sigma^2/c_b}{\epsilon^2}\, \right) \text{ and }     S_\epsilon = \mathcal{O}\left(\, 
        \frac{ L^2( \phi(x_0) - \phi_{\inf} )^2 c_b + \sigma^4/c_b }{\epsilon^4} \, \right), \text{respectively}.
    \end{align*}
\end{enumerate}
\end{theorem}
\begin{proof} \noindent (i)  By Lemma~\ref{lem.Nk sum}(i) and \ref{lem.pj sum}\ref{item.pj1}, one has 
    \begin{align*}
        \sum_{k=0}^{K-1} \frac{1}{N_k} \le  \frac{1}{c_b} \quad \text{and} \quad J(K)+1 \le \frac{1}{c_p}+1.
    \end{align*}
    Therefore, \eqref{eq.rate bound main} is a consequence of Proposition~\ref{proposition:main prox}. The iteration complexity follows by invoking Jensen's inequality, leading to the claim that
    \begin{align*}
     \E\left[ \, \left\|\, G_{s}(x_{R_K}) \, \right\|\, \right]   \le \sqrt{\E\left[ \, \left\|\, G_{s}(x_{R_K}) \, \right\|^2\, \right]},
    \end{align*}
    and subsequently rearranging to find the smallest $K_{\epsilon}$ such that $\E\left[ \, \left\|\, G_{s}(x_{R_{K_{\epsilon}}}) \, \right\|\, \right] \le \epsilon$. The sample-complexity bound, given by $S_{\epsilon}$, follows by $S_\epsilon = K_\epsilon \lceil c_b K_{\epsilon}\rceil$ and invoking the inequality of $(a+b+c)^2 \le 3a^2 + 3b^2 + 3c^2$ for real numbers $a,b,c$. 
    \noindent (ii) By assuming that $s = \frac{2(1-\alpha)\beta\zeta }{L}$ for some constant $\zeta \in (0,1]$, we have that $c = s$ from the definition of $c$ in \eqref{eq:prox.t_bound}. Therefore, the results follow directly from part (i) by substituting $s$ and $c$ accordingly.
\end{proof}

\medskip
\begin{remark}  \em

\noindent (I) For the standard SGD with $\mathcal{O}\left(\frac{1}{L}\right)$-constant stepsize and $\mathcal{O}(1)$ batch size, the rate guarantee is $\E\left[\|g(x_{R_k})\|^2\right] = \mathcal{O}\left(\frac{L(\phi(x_0) - \phi_{\inf})}{K}+\sigma^2 \right)$ as established in \cite[Theorem 2.1]{Ghadimi2013} and \cite[Theorem 4.8]{Bottou2018}. This is matched by our result in the case of $s = \frac{2(1-\alpha)\beta\zeta}{L}$ in Theorem~\ref{thm.nonaymptotic constant nk pj}(ii) in \eqref{eq.rate bound s 1overL} and using batch size of $N_k = 1$ {,i.e., $c_b  = 1/K $}.\\ 

\noindent (II) {When $s = \frac{2(1-\alpha)\beta\zeta}{L}$, if $c_b$ is chosen as $c_b = \sigma^2/(L(\phi(x_0) - \phi_{\inf}))$, then $S_\epsilon$ can be sharpened to $\mathcal{O}\left(\frac{ L(\phi(x_0)-\phi_{\inf})\sigma^2 }{\epsilon^4} \right)$, consistent with the canonical sample complexity result implied by \cite[Corollary 2.2]{Ghadimi2013}. In this corollary, let $\tilde D = D_f$, which is the optimal selection, and we can deduce that the sample complexity is $\mathcal{O} \left( \frac{L(f(x_1)- f^*)}{\epsilon^2} + \frac{L(f(x_1)- f^*)\sigma^2}{\epsilon^4} \right)$.}\\

\noindent (III) It bears reminding that naively setting $s$ to the the minimum feasible value (when $L$ is available) may provide benefit in terms of an improved constant factor in the worst-case error bound; however, this can have a debilitating impact on the performance of the scheme, significantly limiting the ability of the scheme to adapt to the topography of the function by adapting steplengths. This impact becomes far more pronounced in ``peaky'' functions. In fact, the empirical benefits of our adaptive schemes emerge in our numerical studies.\\

\noindent (IV) Finally, our result addresses some key shortcomings present in prior work~\cite{Vaswani2019} in that our guarantees do not necessitate the satisfaction of the strong growth condition and allow for non-monotone steplength sequences without initial steplengths being stringently bounded in terms of the reciprocal of $L$. 
$\hfill \Box$
\end{remark}
\medskip

The batch-size sequence $\left\{N_k\right\}$ and the period sequence $\{p_j\}$ in the previous theorem are dependent on $K$, i.e., dependent on $\epsilon$.  We are also interested in the case where both $\{N_k\}$ and $\{p_j\}$ are independent of $K$ and we provide such a result in the following corollary.

\begin{corollary}\label{cor.nonaymptotic increasing nk pj}\em
For any $K \in \N\setminus\{0\}$, consider the sequence generated by Algorithm \ref{alg: sls pk prox} using $N_k =k+1$ and $p_j = 2^j l_0$, where $l_0 \in \N\setminus\{0\}$ as defined in Lemma~\ref{lem.pj sum}\ref{item.pj2}. Suppose Assumptions~\ref{ass:prox}, \ref{ass: F}, and \ref{ass: var} hold. Suppose $R_K$ is defined in Proposition~\ref{proposition:main prox}. Then, the following hold.
\begin{enumerate}[label=(\roman*)]
    \item For any  $K \ge 1$,
\begin{align*}
         \E\left[ \, \left\| G_{s}(x_{R_K}) \right\|^2\, \right] &\le \frac{2( \phi(x_0) - \phi_{\inf} )}{\alpha c K} +   \frac{ \left(\frac{1}{L} + 6Ls^2 + 2\alpha c\right) \sigma^2 (\ln(K)+1)}{\alpha c K}  \\
         & ~~~+  \frac{6L(s-c)^2(L_r + 3B_g)^2\left\lceil\left[\log_{1/\beta} \left(\frac{sL}{2(1-\alpha)} \right)\right]_{+} + 2 \right\rceil \left\lceil \log_2\left(\frac{l_0 + K}{l_0}\right) \right\rceil}{\alpha c K} \\
         &  = \mathcal{O}\left(\, \frac{\Phi\ln(K) }{K} \, \right),
\end{align*}
where
\begin{align*}
\Phi \triangleq c^{-1}(\phi(x_0) - \phi_{\inf}) + (L^{-1}c^{-1} + 6Ls^2 c^{-1} + 2)\sigma^2 + 6L(s-c)^2(L_r + B_g)^2c^{-1}\left[\ln(sL)\right]_{+},
\end{align*}
implying that the iteration complexity $K_{\epsilon}$ and sample-complexity $S_{\epsilon}$ that guarantee \eqref{eq.convergence measure} are 
  \begin{align*}   
        K_{\epsilon}  = \mathcal{O}\left( \ln\left(\frac{\Phi}{\epsilon^2} \right)\frac{\Phi}{\epsilon^2} \right)
    \quad \text{and}\quad  
    S_\epsilon = \mathcal{O}\left( \left(\ln\left(\frac{\Phi}{\epsilon^2} \right)\right)^2\frac{\Phi^2}{\epsilon^4} \right).
\end{align*}
\item Suppose, in addition, $s = \frac{2(1-\alpha)\beta\zeta }{L}$ for some constant $\zeta \in (0,1]$. Then, for any $K \ge 1$, 
\begin{align*}
         \E\left[ \, \left\| G_{s}(x_{R_K}) \right\|^2\, \right] &\le \frac{L( \phi(x_0) - \phi_{\inf} )}{(1-\alpha)\beta \zeta\alpha K} +   \frac{ \left(\frac{1}{2\alpha(1-\alpha)\beta \zeta} + \frac{12(1-\alpha)\beta\zeta}{\alpha} + 2\right) \sigma^2 (\ln(K)+1)}{K} \\
         & = \mathcal{O}\left(\, \frac{\Phi_2 \ln(K)}{K} \, \right),
\end{align*}
where $\Phi_2 \triangleq L( \phi(x_0) - \phi_{\inf} ) + \sigma^2.$ 
\end{enumerate}
\end{corollary}
\begin{proof}
    The rate result in (i) follows by Lemma~\ref{lem.Nk sum}(ii), \ref{lem.pj sum}\ref{item.pj2}, and Proposition~\ref{proposition:main prox}. To derive the iteration complexity, one has to find the largest $K_{\epsilon}$ such that 
     \begin{align*}
\frac{\epsilon^2}{\Phi}   \, \le \,   \E\left[\, \left \| \, G_{s}\left(x_{R_{K_{\epsilon}}}\right) \, \right \|^2\, \right]/\Phi = \mathcal{O}\left(\frac{\ln (K_{\epsilon})}{K_\epsilon}\right).
     \end{align*}
     The solution of $K_\epsilon$  can be obtained via the Lambert function (cf.~\cite{Jalilzadeh2022}), and the sample-complexity follows from 
     \begin{align*}
        S_\epsilon = \sum_{k=0}^{K_{\epsilon}-1} (k+1) = \frac{K_{\epsilon}(K_{\epsilon}+1)}{2}.
     \end{align*}
The result of (ii) follows immediately with $s = c = \frac{2(1-\alpha)\beta\zeta}{L}$.
\end{proof}
Following from Corollary~\ref{cor.nonaymptotic increasing nk pj}, we consider the performance of the proposed algorithm when $K$ goes to infinity. As such, the choices of $\{N_k\}$ and $\{p_j\}$ being dependent on $K$ are not applicable. Instead, we choose a sequence $\{p_j\}$ such that the resulting $J(K)/K \to 0$ and $\{N_k\}$ grows.
\begin{proposition}\label{thm.aymptotic}\em
Consider the sequence generated by Algorithm \ref{alg: sls pk prox}. Suppose Assumptions~\ref{ass:prox}, \ref{ass: F}, and \ref{ass: var} hold. Suppose the choice of $\{p_j\}$ satisfies 
\begin{align*}
    \lim_{K \to \infty} \frac{J(K)}{K}  = 0.
\end{align*}
Suppose for any $K > 0$, $R_K$ is as defined in Proposition~\ref{proposition:main prox}. 
\begin{enumerate}[label=(\roman*)]
    \item If $N_k = b$ for some constant $b \in \N\setminus\{0\}$, then (i.a) and (i.b) hold. 
\begin{align*}
      \mbox{(i.a)} & \quad  \limsup_{K \to \infty} \E\left[ \, \left\|\, G_{s}(x_{R_K}) \, \right\|^2\, \right] \, \le \, \frac{ (\frac{1}{L} + 6Ls^2 + 2\alpha c) \sigma^2}{\alpha c b} \\
       \mbox{(i.b)}&  \quad  \liminf_{k \to \infty} \E\left[ \, \left\|\, G_{s}(x_{k}) \, \right\|^2\, \right] \, \le \, \frac{ (\frac{1}{L} + 6Ls^2 + 2\alpha c) \sigma^2}{\alpha c b}.
    \end{align*}
    If $b \ge  \epsilon^{-1}$, then ${\displaystyle \limsup_{K \to \infty}} \ \E\left[ \, \left\|\, G_{s}(x_{R_K}) \, \right\|^2\, \right] \, = \, \mathcal{O}(\epsilon).$ 
    \item If $\{N_k\}$ satisfies 
        $\lim_{K \to \infty}\frac{1}{K}\sum_{k=0}^{K-1} \frac{1}{N_k} = 0$, 
    then (ii.a) and (ii.b) hold.   
    \begin{align*}
        \mbox{(ii.a)} \quad \lim_{K \to \infty} \E\left[\, \left\|\, G_{s}(x_{R_K}) \, \right\|^2\, \right]\,  = \, 0; \quad 
        \mbox{(ii.b)} \quad \liminf_{k \to \infty} \E\left[\, \left\|\, G_{s}(x_{k})\, \right\|^2\, \right] = 0.
    \end{align*}
\end{enumerate}
\end{proposition}
\begin{proof}
    From \eqref{eq.rate bound} and ${\displaystyle \lim_{K \to \infty}} \frac{J(K)}{K}  = 0$, we have
        \begin{align} 
         \limsup_{K\to \infty}\E[ \| G_{s}(x_{R_K}) \|^2] &\le  \frac{( \tfrac{1}{L} + 6Ls^2 + 2\alpha c) \sigma^2}{\alpha c} \lim_{K \to \infty} \sum_{k=0}^{K-1} \frac{1}{KN_k}.
    \end{align}
    The results of two cases of $\left\{N_k\right\}$ follows in the form of (i.a) and (ii.a), while (i.b) and (ii.b) follows from (i.a) and (ii.a), respectively.
\end{proof}
\begin{remark}\em
    We provide sufficient conditions for $\{p_j\}$ and $\{N_k\}$ to guarantee the asymptotic convergence result in Proposition~\ref{thm.aymptotic}. To be specific, for example, the choice of $\{p_j\}$ in Lemma~\ref{lem.pj sum}\ref{item.pj2} and \ref{item.pj3} satisfy the condition of $\lim_{K \to \infty} \frac{J(K)}{K}  = 0$, and $\{N_k\}$ in Lemma~\ref{lem.Nk sum}\ref{item.pj2} and \ref{item.pj3} satisfy $\lim_{K \to \infty}\frac{1}{K}\sum_{k=0}^{K-1} \frac{1}{N_k} = 0$.$\hfill \Box$
\end{remark}

\section{Specializations to PL and convex regimes} \label{sec:4}

{In this section, we refine our guarantees to two important settings. Specifically, we present rates and guarantees for the PL setting, which subsumes the strongly convex regime, in Section~\ref{sec:pl}. Subsequently, by appealing to a somewhat modified linesearch criterion, we present a result for convex settings in Section~\ref{sec:convex}.}

\subsection{Guarantees under the PL condition}\label{sec:pl}
The next result applies when the problems satisfy the following condition for some $\mu > 0$,
\begin{align}\label{eq.plcondition}
    \left\|\, G_s(x)\, \right\|^2 \, \ge \, \mu (\phi(x) - \phi_{\inf}) \text{ for all } x \in \dom(r).
\end{align}
Condition~\eqref{eq.plcondition} is referred to as the \emph{generalized Lezanski-Polyak-Lojasiewicz (LPL)} condition, reducing to standard LPL condition in the smooth setting (i.e., $r(x) \equiv 0$). Note that in the literature, this condition is often referred to as the PL condition. Problems satisfying the  LPL condition, given by \eqref{eq.plcondition}, include strongly convex problems, amongst others (cf.~\cite{Karimi2020}). Our PL condition \eqref{eq.plcondition} is identical to the one used in \cite{Li2018}, which is considered more natural, and differs from the one adopted in \cite{Karimi2020}. The equivalence of the two conditions when using $s = 1/L$ is discussed in \cite{Li2018}. We now provide a rate and complexity statement by leveraging our prior analysis.

\medskip
\begin{corollary} \label{cor.pl convergence}
\em
    Suppose the conditions in Theorem~\ref{thm.nonaymptotic constant nk pj} hold. Additionally, suppose that the problem satisfies \eqref{eq.plcondition}. Consider the sequence $\left\{\, x_k \, \right\}$ generated by Algorithm~\ref{alg: sls pk prox} and let $K$ and $R_K$ be as defined in Theorem~\ref{thm.nonaymptotic constant nk pj}. Then the following hold.
\begin{enumerate}[label=(\roman*)]
    \item For any  $K \ge 1$,  
    \begin{align*}
     \E\left[\,  \phi(x_{R_K}) - \phi_{\inf} \, \right] & \le \frac{2( \phi(x_0) - \phi_{\inf} )}{\mu\alpha c K} +   \frac{ \left(\frac{1}{L} + 6Ls^2 + 2\alpha c\right) \sigma^2}{\mu \alpha c c_b K}  \\
        &  ~~~+  \frac{6L(s-c)^2 (L_r+3B_g)^2  \left\lceil\left[\log_{1/\beta} \left(\frac{sL}{2(1-\alpha)} \right)\right]_{+} + 2 \right\rceil (c_p + 1)}{\mu\alpha c c_p K}.
    \end{align*}
For any $\epsilon > 0$, suppose $K_{\epsilon}$ is the number of iterations that guarantees 
\begin{align} \label{eq.obj convergence measure}
 \E\left[ \, \phi(x_{R_{K(\epsilon)}}) - \phi_{\inf} \, \right] \, \le \,  \epsilon.  \end{align}
Then the iteration-complexity $K_{\epsilon}$ and the sample-complexity $S_{\epsilon}$ are given by
  \begin{align*}   
        K_{\epsilon}  = \mathcal{O}\left(\, 
        \frac{ ( \phi(x_0) - \phi_{\inf} ) + (L^{-1} + Ls^2 + c) \sigma^2 + L (s-c)^2 (B_g + L_r)^2 \left[\ln(Ls)\right]_{+}}{\mu c \epsilon}\, \right)
    \end{align*}
    and
    \begin{align*}
    S_\epsilon = \mathcal{O}\left(\, 
        \frac{ ( \phi(x_0) - \phi_{\inf} )^2 + (L^{-1} + Ls^2 + c)^2 \sigma^4 + L ^2(s-c)^4 (B_g + L_r)^4 \left[\ln(Ls)\right]_{+}^2}{\mu^2 c^2 \epsilon^2}\, \right),
\end{align*}
respectively.
\item Suppose, in addition, $s = \frac{2(1-\alpha)\beta\zeta }{L}$ for some constant $\zeta \in (0,1]$. Then, for any $K \ge 1$, 
\begin{align*}
         \E\left[\,  \phi(x_{R_K}) - \phi_{\inf} \, \right] = \mathcal{O}\left(\, \frac{L(\phi(x_0) - \phi_{\inf}) + \sigma^2}{\mu K} \, \right) 
    \end{align*}
and the iteration and sample-complexities that guarantee \eqref{eq.obj convergence measure} are 
 \begin{align*}   
        K_{\epsilon}  = \mathcal{O}\left(\, 
        \frac{ L( \phi(x_0) - \phi_{\inf} ) + \sigma^2}{\mu \epsilon}\, \right) \text{ and }     S_\epsilon = \mathcal{O}\left(\, 
        \frac{ L^2( \phi(x_0) - \phi_{\inf} )^2 + \sigma^4}{\mu^2 \epsilon^2} \, \right) \text{respectively}.
    \end{align*}
\end{enumerate}
\end{corollary}

\begin{proof}
(i) Applying the PL condition \eqref{eq.plcondition} at $x_{R_K}$ and taking expectations gives
\begin{align*}
\E\left[\, \left\|\, G_{s}(x_{R_K}) \,\right\|^2\, \right] \ge \mu \E\left[ \phi(x_{R_K}) - \phi_{\inf} \right].
\end{align*}
Here $x_{R_K} \in \dom(r)$ because every iterate $x_k$ belongs to $\dom(r)$ and $R_K\in\{0,\ldots,K-1\}$.
The bound in (i) follows by combining this inequality with \eqref{eq.rate bound main} in Theorem~\ref{thm.nonaymptotic constant nk pj} and dividing by $\mu$. The corresponding complexity bounds and part (ii) follow immediately.
\end{proof} 
\begin{remark} \em
Under the PL-condition and by using a sufficiently small constant steplength of $\mathcal{O}(1/L)$, a linear (last-iterate) convergence rate is constructed in \cite{Li2018,Karimi2020,Lei2020a} and the resulting iteration complexity is bounded by $\mathcal{O}\left(\tfrac{L}{\mu}\ln(\epsilon^{-1})\right)$. Among these results, the authors in~\cite{Lei2020a} consider a stochastic setting where the sample-size grows at a suitably defined geometric rate in order to achieve the linear convergence. Notably, these guarantees necessitate knowledge of $L$, employing a constant steplength scheme; consequently, this steplength can be exceedingly small if chosen in alignment with the theoretical prescription. Instead, our result in Corollary~\ref{cor.pl convergence}(i) leverages adaptive steplengths and does not necessitate knowledge of either $\mu$ or $L$. 
Our analysis is based on the expected function decrease in \eqref{eq.sufficient decrease prox}, whose upper bound contains the term $\mathbb{P}\{t_k \neq t_{k-1}\}$. This probability does not appear to decay geometrically, preventing a direct application of the geometric-recursion argument used for prescribed constant stepsizes and geometrically increasing batch sizes used in \cite{Lei2020a} to establish a linear convergence rate. $\hfill \Box$
\end{remark}

\subsection{Addressing convex settings}\label{sec:convex}
We now specialize our findings to the convex case and impose the following assumption.
\begin{assumption}[Convexity and compactness] \label{ass: convexity}\em
The function $f$ is convex and the effective domain $\dom(r)$ is compact with diameter $D_x$, i.e., $D_x \triangleq {\displaystyle \sup_{x,y \in \dom(r)}} \|x-y\|$. 
\end{assumption}

Under compactness requirements, we may claim that a minimizer $x^*$ is attained over set $\dom(r)$, and $\phi(x_*) = \phi_{\inf}$.
Instead of imposing the Armijo condition in the nonconvex setting \eqref{eq.prox.armijo}, we impose a slightly modified condition for convex problems for some $\alpha \in {[1/2,1]}$, as defined as follows.
\begin{align}\label{eq.ls_convex}
 \tf_k(x_k(t_k)) & \, \le \, \tf_k(x_k) - \tg _k(x_k) ^\top( x_k - x_k(t_k) ) + \frac{1-\alpha}{t_k} \left\| \, x_k - x_k(t_k) \, \right\|^2. 
\end{align}
The resulting algorithm is formally defined in Algorithm~\ref{alg: sls pk prox convex} and referred to as {\bf SLAM}$^{\rm con}$ to capture that such a scheme is specialized to settings where $f$ is convex.
\begin{tcolorbox}
\begin{algorithm}[H]\caption{A Stochastic LineseArch Method ({\bf SLAM$^{\rm con}$}) under convexity of $f$}\label{alg: sls pk prox convex}
    \begin{algorithmic}[1]
        \Require same with Algorithm~\ref{alg: sls pk prox} except $\alpha \in [1/2, 1)$
        \State Replacing the linesearch condition \eqref{eq.prox.armijo} in Line~\ref{line.backtrack} Algorithm~\ref{alg: sls pk prox} with \eqref{eq.ls_convex}.
    \end{algorithmic}
\end{algorithm}
\end{tcolorbox}
\begin{remark}[Relation between nonconvex and convex linesearch conditions]\em
    We consider the linesearch condition \eqref{eq.ls_convex}, considered in a setting with convex composite functions,  as proposed by \cite{Beck2009} with $\alpha = \frac{1}{2}$   and further studied in \cite{Salzo2017,Beck2017}. When $r \equiv 0$, both condition \eqref{eq.ls_convex} and the nonconvex condition \eqref{eq.prox.armijo} reduces to the standard Armijo condition in the smooth unconstrained setting. Otherwise, when $r \neq 0$, the analysis for nonconvex setting still holds for the following reasons.
    \begin{enumerate}
        \item[(i)] The lower bound of $t_k$ generated by Algorithm~\ref{alg: sls pk prox convex} is the same as that in Lemma~\ref{lm.prox:t_lb} and the proof is similar.
        \item[(ii)]  The convex line search condition \eqref{eq.ls_convex} implies the nonconvex one \eqref{eq.prox.armijo}. Observing points $x_k$, $x_k - t_k \tg_k(x_k)$, and $x_k(t_k)$, by utilizing convexity of $r$ and invoking Lemma~\ref{lem.prox}, one has
\begin{alignb}    
    & (x_k - t_k \tg_k(x_k) - x_k(t_k))^\top (x_k - x_k(t_k)) \le t_k (r(x_k) - r(x_k(t_k))) \\
    \iff &  \frac{\|x_k - x_k(t_k)\|^2}{t_k} - \tg_k(x_k)^\top (x_k - x_k(t_k)) \le r(x_k) - r(x_k(t_k))  \label{eq.prox inequality}
\end{alignb}
Adding inequality \eqref{eq.ls_convex} and \eqref{eq.prox inequality} yields \eqref{eq.prox.armijo}. 
    \end{enumerate} 
Therefore, all supporting lemmas in Section~\ref{sec:supporting lemmas} remain valid when using Algorithm~\ref{alg: sls pk prox convex}. The subsequent analysis diverges from this point, focusing on the suboptimality measure \eqref{eq.obj convergence measure} for convex problems, rather than on stationarity. $\hfill \Box$
\end{remark}

We begin our analysis by deriving a bound on $ \E\left[\|x_{k+1} - x^*\|^2\left(\tfrac{1}{t_{k+1}} - \tfrac{1}{t_k}\right)\right]$.
\medskip
\begin{lemma}\label{lm: errbd convex}
 \em   Consider Algorithm~\ref{alg: sls pk prox convex} and suppose that Assumptions \ref{ass:prox} -\ref{ass: convexity} hold.  Then for any $k \in \N$, the following hold.
    \begin{align*}
        \E\left[\, \left\|\, x_{k+1} - x^*\, \right\|^2\left(\frac{1}{t_{k+1}} - \frac{1}{t_k}\right)\, \right] \,  \le D_x^2 \left(\, \frac{1}{c} - \frac{1}{s}\, \right) \bP\left\{\, t_{k+1} \, \neq \, t_k\, \right\}.
    \end{align*}
\end{lemma}
\begin{proof}
    For $k \in \N$, we define three events, $I_{k,1}, I_{k,2},$ and $I_{k,3}$, defined as
    $$I_{k,1}\, \triangleq\, \left\{\,t_{k+1}> t_k\, \right\}, I_{k,2}\, \triangleq \, \left\{\, t_{k+1} = t_k\, \right\}, \mbox{ and } I_{k,3}\, \triangleq\, \left\{\, t_{k+1} < t_k\, \right\}$$ respectively. Then by law of total expectation, we have
    \begin{align*}
        \E\left[\left\|x_{k+1} - x^*\right\|^2\left(\frac{1}{t_{k+1}} - \frac{1}{t_k}\right)\right] &= \underbrace{\E\left[\left\|x_{k+1} - x^*\right\|^2\left(\frac{1}{t_{k+1}} - \frac{1}{t_k}\right)\, \Big| \, I_{k,1}\right]}_{\, \le \, 0 } \bP[\, I_{k,1}\, ] \\
        & + \underbrace{\E\left[\left\|x_{k+1} - x^*\right\|^2\left(\frac{1}{t_{k+1}} - \frac{1}{t_k}\right)\, \Big| \, I_{k,2}\, \right]}_{\, = \, 0}\bP[\, I_{k,2}\, ] \\
        &+ \E\left[\left\|x_{k+1} - x^*\right\|^2\left(\frac{1}{t_{k+1}} - \frac{1}{t_k}\right)\, \Big| \, I_{k,3}\, \right]\bP[I_{k,3}] \\
        & \le  \E\left[\left\|\, x_{k+1} - x^*\, \right\|^2\left(\frac{1}{t_{k+1}} - \frac{1}{t_k}\right)\, \Big|\, I_{k,3}\, \right]\bP\left[\, I_{k,3}\, \right] \\
        & \le D_x^2 \left(\, \frac{1}{c} - \frac{1}{s}\, \right) \bP\left\{\, t_{k+1} \neq t_k\, \right\}, 
    \end{align*}
    where the last inequality arises from the bound for $\|x_{k+1}-x^*\|$, by invoking the lower and upper bounds of $t_{k+1}$ and $t_k$, respectively (both from Lemma \ref{lm.prox:t_lb}), and the relation of $\bP[\, I_{k,3}\, ] \le \bP\left\{\, t_{k+1} \neq t_k\, \right\}$.
\end{proof}         

\begin{theorem} \label{thm.convex.main}
\em
    For any $K \in \N\setminus\{0\}$, consider the sequence generated by Algorithm \ref{alg: sls pk prox convex} using $\{N_k\}$ and $\{p_j\}$ as defined in Theorem~\ref{thm.nonaymptotic constant nk pj}. Suppose Assumption \ref{ass:prox}-\ref{ass: convexity} hold. Suppose $\bar{x}_K$ is defined as $\bar{x}_K \, \triangleq \, \frac{1}{K} {\displaystyle \sum_{k=1}^{K}x_k}$.  
\begin{enumerate}[label=(\roman*)]
    \item For any $K \ge 1$,  
    \begin{align*}
& \quad  \E\left[\,  \phi(\bar{x}_K) - \phi_{\inf} \, \right]  \\    & \le \frac{D_x^2 }{2cK} + \frac{ \left(\frac{1}{2L} + 3Ls^2 \right)\sigma^2}{c_b K} \\
        &~~~+  \frac{\left(\frac{D_x^2}{2} \left(\, \frac{1}{c} - \tfrac{1}{s}\, \right)  +  3L(s-c)^2(L_r + 3B_g)^2\, \right)  (c_p+1) \left\lceil\left[\log_{1/\beta} \left(\frac{sL}{2(1-\alpha)} \right)\right]_{+} + 2 \right\rceil}{c_pK}.\\
    \end{align*}
The resulting iteration and sample-complexities that guarantee \eqref{eq.obj convergence measure} are 
  \begin{align*}   
        K_{\epsilon}  &= \mathcal{O}\left(\, \frac{D_x^2c^{-1} + (L^{-1}+ Ls^2)\sigma^2 /c_b}{\epsilon}  \, \right) \\
        &~~~ + \mathcal{O}\left(\, \frac{\left(D_x^2(c^{-1} - s^{-1}) + L(s-c)^2(L_r + B_g)^2\right) \left[\ln(Ls)\right]_{+} }{\epsilon} \, \right)
    \end{align*}
    and
    \begin{align*}
    S_\epsilon &= \mathcal{O}\left(\, \frac{D_x^4c^{-2}c_b + (L^{-1}+ Ls^2)^2\sigma^4/c_b}{\epsilon^2} \, \right) \\
    &\quad + \mathcal{O}\left(\, \frac{\left(D_x^2(c^{-1} - s^{-1}) + L(s-c)^2(L_r + B_g)^2\right)^2 \left(\left[\ln(Ls)\right]_{+}\right)^2c_b }{\epsilon^2} \, \right),
\end{align*}
respectively.
\item Suppose, in addition, $s = \frac{2(1-\alpha)\beta\zeta }{L}$ for some constant $\zeta \in (0,1]$. Then, for any $K \ge 1$, 
\begin{align*}
         \E\left[\,  \phi(\bar{x}_{K}) - \phi_{\inf} \, \right] = \mathcal{O}\left(\, \frac{D_x^2L + L^{-1}\sigma^2/c_b }{K} \, \right) 
    \end{align*}
and the iteration and sample-complexities that guarantee \eqref{eq.obj convergence measure} are 
 \begin{align*}   
        K_{\epsilon}  = \mathcal{O}\left(\, 
        \frac{ D_x^2L + L^{-1}\sigma^2/c_b}{\epsilon}\, \right), \text{ and }
        S_\epsilon = \mathcal{O}\left(\, 
        \frac{ D_x^4L^2 c_b + L^{-2}\sigma^4/c_b}{\epsilon^2} \, \right), \text{ respectively}.
    \end{align*}
\end{enumerate}
\end{theorem}
\begin{proof}
   (i) For any $k \in \N$, by utilizing the update rule, we have
    \begin{align*}
        \left\|\, x_{k+1} - x^*\, \right\|^2 &= \left\| \, x_{k+1} - x_k + x_k -x^* \,\right\|^2 \\ 
        & = \|  x_{k+1} - x_k \|^2 +  \|  x_k - x^* \|^2 + 2(x_{k+1}- x_k)^\top (x_k - x^*) \\
        & = \|  x_{k+1} - x_k \|^2 +  \|  x_k - x^* \|^2 + 2(x_{k+1}- x_k + t_k \tg_k(x_k) - t_k \tg_k(x_k) )^\top(x_k - x^*) \\
        & = \|  x_{k+1} - x_k \|^2 +  \|  x_k - x^* \|^2 + 2(x_{k+1}- x_k + t_k \tg_k(x_k) )^\top(x_k - x_{k+1} + x_{k+1} - x^*) \\
        & - 2 t_k \tg_k(x_k)^\top (x_k - x^*) \\
        & = \|  x_k - x^* \|^2 - \|  x_{k+1} - x_k \|^2 + 2t_k\tg_k(x_k)^\top(x_k - x_{k+1} ) - 2 t_k \tg_k(x_k)^\top (x_k - x^*) \\
        & + 2(x_{k+1}- x_k + t_k \tg_k(x_k) )^\top( x_{k+1} - x^*)\\
        & \le \|  x_k - x^* \|^2 - \|  x_{k+1} - x_k \|^2 + 2t_k\tg_k(x_k)^\top(x_k - x_{k+1} ) - 2 t_k \tg_k(x_k)^\top (x_k - x^*) \\
        & + 2t_k \left(r(x^*) - r(x_{k+1})\right)
    \end{align*} 
where the last inequality is a consequence of Lemma~\ref{lem.prox}.
From the above relation, condition \eqref{eq.ls_convex}, and $\alpha \in [1/2,1]$, we have
    \begin{align*}
        & \quad r(x_{k}) - r(x^*) + \tg_k(x_k)^\top(x_k - x^*) \\
        & \le \frac{\|x_k - x^*\|^2}{2t_k} - \frac{\|x_{k+1} - x^*\|^2}{2t_k} + \tg_k(x_k)^\top(x_k - x_{k+1} ) - \frac{ \| x_{k+1} - x_k \|^2}{2t_k} + r(x_k) - r(x_{k+1}) \\
 & \le \frac{\|x_k - x^*\|^2}{2t_k} - \frac{\|x_{k+1} - x^*\|^2}{2t_k} + \tf_k(x_k) - \tf_k(x_{k+1}) + \left(\frac{1-2\alpha}{2t_k} \right) \| x_{k+1} - x_k \|^2 + r(x_k) - r(x_{k+1}) \\
 & \le \frac{\|x_k - x^*\|^2}{2t_k} - \frac{\|x_{k+1} - x^*\|^2}{2t_k} + \tilde\phi_k(x_k) - \tilde\phi_k(x_{k+1}). 
    \end{align*}
    It follows that
     \begin{align*}
    \phi(x_k) - \phi_{\inf} & = f(x_k) - f(x^*) + r(x_{k}) - r(x^*) \\
    & \le r(x_{k}) - r(x^*) + g_k(x_k)^\top(x_k - x^*) \\
    & \le  \frac{\|x_k - x^*\|^2}{2t_k} - \frac{\|x_{k+1} - x^*\|^2}{2t_k} + \tilde\phi_k(x_k) - \tilde\phi_k(x_{k+1})  + (g_k - \tg_k(x_k))^\top (x_k - x^*).
    \end{align*}
    Taking expectations conditional on $\mathcal{F}_k$ on both sides and by invoking the convexity of $f$, we have
    \begin{align*}
    \E \left[\, \phi(x_k) - \phi_{\inf} \, \mid \, \mathcal{F}_k \, \right] & \le \E\left[\frac{\|x_k - x^*\|^2}{2t_k} - \frac{\|x_{k+1} - x^*\|^2}{2t_k}\,\Big|\,\mathcal{F}_k\right] + \E \left[\, \tilde \phi_k(x_k) \, \mid \, \mathcal{F}_k \, \right] - \E\left[\tilde\phi_k(x_{k+1}) \ \Big| \mathcal{F}_k \right]\\
    & = \E\left[\frac{\|x_k - x^*\|^2}{2t_k} - \frac{\|x_{k+1} - x^*\|^2}{2t_k}\,\Big|\,\mathcal{F}_k\right] + \phi(x_k) - \E\left[\tilde\phi_k(x_{k+1}) \ \Big| \mathcal{F}_k \right]. 
    \end{align*}   
    Taking unconditional expectations on both sides and by invoking Proposition \ref{lm: errbd prox}, we have
    \begin{align*}
        \E\left[\, \phi(x_k) - \phi_{\inf} \,\right] & \le \E\left[\, \frac{\|x_k - x^*\|^2}{2t_k} - \frac{\|x_{k+1} - x^*\|^2}{2t_k}\, \right] +  \E\left[\, \phi(x_k)\, \right] - \E\left[\, \phi(x_{k+1})\, \right] \\
        & \quad +  \E\left[\, \phi(x_{k+1})\, \right] - \E\left[\, \tilde\phi_{k}(x_{k+1})\, \right]\\
        &\le  \E\left[\, \frac{\|x_k - x^*\|^2}{2t_k} - \frac{\|x_{k+1} - x^*\|^2}{2t_k}\, \right] +  \E[\, \phi(x_k)\, ] - \E[\, \phi(x_{k+1}) \, ] \\
        &\quad + \left(\frac{1}{2L} + 3Ls^2\right)\frac{\sigma^2}{N_k} + 3L(s-c)^2(L_r + 3B_g)^2 \bP\left\{\, t_{k} \neq t_{k-1}\, \right\} \\
      \implies \E\left[\, \phi(x_{k+1}) - \phi_{\inf} \,\right] &\le  \E\left[\, \frac{\|x_k - x^*\|^2}{2t_k} - \frac{\|x_{k+1} - x^*\|^2}{2t_k}\, \right] \\
        &\quad + \left(\frac{1}{2L} + 3Ls^2\right)\frac{\sigma^2}{N_k} + 3L(s-c)^2(L_r + 3B_g)^2 \bP\left\{\, t_{k} \neq t_{k-1}\, \right\}.
    \end{align*}
    By telescoping, Lemma \ref{lm: errbd convex}, $\phi(x) \ge \phi_{\inf}$ for any $x \in \dom(r)$, the {compactness of $\dom(r)$} in Assumption \ref{ass: convexity}, \eqref{eq.p to T}, Lemma~\ref{lem.Nk sum} and Lemma~\ref{lem.pj sum}, we have
    \begin{align*}
        & \quad \frac{1}{K}\sum_{k=0}^{K-1}\E\left[\, \phi(x_{k+1}) - \phi_{\inf} \,\right] \\ 
        &\le  \frac{1}{K}\sum_{k=0}^{K-1}\E\left[\frac{\|x_k - x^*\|^2}{2t_k} - \frac{\|x_{k+1} - x^*\|^2}{2t_{k+1}} + \frac{\|x_{k+1} - x^*\|^2}{2t_{k+1}} - \frac{\|x_{k+1} - x^*\|^2}{2t_k} \right] \\
        &\quad + \frac{1}{K}\sum_{k=0}^{K-1}\left(\left(\tfrac{1}{2L} + 3Ls^2\right)\tfrac{\sigma^2}{N_k} + 3L(s-c)^2(L_r + 3B_g)^2 \bP\{t_{k} \neq t_{k-1}\}\right)\\
        & =  \frac{1}{K}\left(\, \E\left[\, \frac{\|x_0 - x^*\|^2}{2t_0}\, \right] - \E\left[\, \frac{\|x_{K} - x^*\|^2}{2t_{K}}\, \right] \right. \\
            & \left. ~~~~~~+ \sum_{k=0}^{K-1}\E\left[\, \frac{\|x_{k+1} - x^*\|^2}{2t_{k+1}} - \frac{\|x_{k+1} - x^*\|^2}{2t_k} \, \right] \, \right) \\
        &\quad + \frac{1}{K}\sum_{k=0}^{K-1}\left(\, \left(\, \frac{1}{2L} + 3Ls^2\, \right)\frac{\sigma^2}{N_k} + 3L(s-c)^2(L_r + 3B_g)^2 \bP\left\{\, t_{k} \neq t_{k-1}\, \right\}\, \right)\\
        & \le  \frac{1}{K}\left(\, \frac{D_x^2}{2c} + \sum_{k=0}^{K-1}\left(\, \frac{D_x^2}{2} \left(\frac{1}{c} - \frac{1}{s}\, \right) \bP\left\{\, t_{k+1} \neq t_k\, \right\} \, \right) \,\right) \\
       &\quad + \frac{1}{K}\sum_{k=0}^{K-1}\left(\left(\tfrac{1}{2L} + 3Ls^2\right)\frac{\sigma^2}{N_k} +3L(s-c)^2(L_r + 3B_g)^2 \bP\{t_{k} \neq t_{k-1}\} \right) \\
       & \le \frac{D_x^2}{2cK} + \frac{ \left(\frac{1}{2L} + 3Ls^2 \right)\sigma^2}{c_b K} \\
        &~~~+  \frac{\left(\frac{D_x^2}{2} \left(\, \frac{1}{c} - \tfrac{1}{s}\, \right)  + 3L(s-c)^2(L_r + 3B_g)^2\, \right)  (c_p+1) \left\lceil\left[\log_{1/\beta} \left(\frac{sL}{2(1-\alpha)} \right)\right]_{+} + 2 \right\rceil}{c_pK}.
    \end{align*}
    The result follows by leveraging the convexity of $\phi$ and employing the averaged iterate $\bar{x}_K = \frac{1}{K}{\displaystyle \sum_{k=1}^{K} x_k}$.
    (ii) The result follows immediately with $s = c = \frac{2(1-\alpha)\beta\zeta}{L}$.
\end{proof}

\begin{remark}\em
\noindent (I) By \cite[Theorem 2.1]{Ghadimi2013}, in the convex setting, the standard SGD scheme with a constant steplength of $\mathcal{O}(1/L)$ and a batch-size of $\mathcal{O}(1)$ is characterized by a rate guarantee of $\E[\phi(\bar{x}_K) - \phi_{\inf}] \le \frac{L\|x_0 - x^*\|^2}{K} + L^{-1}\sigma^2$. This is comparable with our result in Theorem~\ref{thm.convex.main}(ii) by using a batch-size of $\mathcal{O}(1)$, i.e., $c_b = \mathcal{O}(1/K)$.\\ 

\noindent (II) Notably, our guarantees in the convex regime are free of any form of interpolation condition~\cite{Vaswani2019}, allowing for contending with a far broader class of problems.  $\hfill \Box$
\end{remark}

\section{Numerical results} \label{sec:5}
In this section, we conduct an expansive set of numerical tests and comparisons with competing solution methods such as $\texttt{SGD}$ and $\texttt{Adam}$. We begin in Section~\ref{sec:5.1} by discussing the set of schemes being compared and prescribing the algorithm parameters in each instance. Existing algorithms have to be tuned extensively, often for each problem class and instance, and untuned schemes can perform significantly worse than their tuned counterparts, as observed in Section~\ref{sec:5.2}. We then compare the performance of \texttt{SLAM} with competing schemes on a stochastic variant of the Rosenbrock function (Section~\ref{sec:5.3}), a set of ML applications (Section~\ref{sec:5.4}) and a nonconvex two-stage stochastic programming problem (Section~\ref{sec:5.5}). We conclude this section with Section~\ref{sec:5.6} where we show that the performance of our proposed scheme is relatively robust to choice of cycle lengths. 
\subsection{Related computational schemes}\label{sec:5.1}
We compare the following schemes in our numerical studies. Of these, the first scheme is \texttt{SLAM} with $s=1$ and a constant cycle length of $50$, i.e. $p_j = 50$ for all $j$. The set of schemes being compared include \texttt{sls\_0}, \texttt{sls\_2} (see \cite{Vaswani2019} for both schemes), \texttt{SGD} with constant and diminishing steplengths and \texttt{Adam} \cite{Kingma2017a}. In the case of the latter three schemes, we tune the appropriate stepsizes in these algorithms and discuss the impact of tuning in the next subsection. These schemes and their settings are as follows.
\begin{enumerate}[label=(\arabic*),parsep=0pt,itemsep=1pt]
    \item \texttt{slam\_1\_p50} (our scheme): SLAM with parameters $s=1$, $\{p_j\}=50$;
    \item \texttt{sls0\_1}: \cite[Algorithm 1, \texttt{opt=0}]{Vaswani2019} with $\eta_{\max}=1$  (this is equivalent to SLAM with $s=1$, $p_0=K$);
    \item \texttt{sls2\_1}: \cite[Algorithm 1, \texttt{opt=2}]{Vaswani2019}  with $\gamma=2$, $\eta_{\max}=10$ (for finite sample spaces);
    \item \texttt{sgd\_const\_tuned}: mini-batch SGD with tuned constant stepsize;
    \item \texttt{sgd\_dimin\_tuned}: mini-batch SGD with diminishing stepsize $t_k = s/\sqrt{k+1}$, where initial stepsize $s$ is tuned;
    \item \texttt{adam\_tuned}: Adam \cite{Kingma2017a} with $\beta_1 = 0.9$, $\beta_2 = 0.999$, $\epsilon = 10^{-8}$, the default setting in \cite{Kingma2017a} and the constant stepsize is tuned. 
\end{enumerate}
In our test problems, the nonsmooth term $r$ is either an
$\ell_1$ norm regularizer or the indicator function of a box/affine
constraint set. All of the compared schemes update the iterates or
compute trial points via a proximal-gradient step $\prox_{t_kr}(x_k -
t_kd_k)$ for some search direction $-d_k$. With regard to the linesearch
methods \texttt{slam\_1\_p50}, \texttt{sls0\_1}, and \texttt{sls2\_1}, we
set the parameters $\alpha = 0.1$ and $\beta = 0.9$. We use an identical
constant batch size schedule $\{N_k\}$ across the three schemes. The
batch-size and iteration budget may depend on the number of variables and
will be indicated where appropriate. Note that \texttt{sls2\_1} is a
heuristic scheme according to \cite{Vaswani2019} in that it lacks
convergence theory for certain steplength choices. In fact,
\texttt{sls0\_1} may also be viewed as a heuristic for certain
steplength choices, i.e., is not guaranteed to converge, because
(i) the initial stepsize $s=1$ may not satisfy the stepsize restriction $s
\le 2/(\rho L)$ required by the \cite{Vaswani2019} for the nonconvex
problems; (ii) furthermore, the interpolation or strong growth condition may not hold
either. We choose $s=1$ for \texttt{sls0\_1} since it matches the parameter
choice used in the experiments of the \cite{Vaswani2019}.

\subsection{Tuning \texttt{SGD} and \texttt{Adam}}\label{sec:5.2} Motivated
by the external tuning framework in
\cite{dahlBenchmarkingNeuralNetwork2023}, we design a simple tuning
procedure for the stepsizes for \texttt{sgd\_const\_tuned},
\texttt{sgd\_dimin\_tuned}, and \texttt{adam\_tuned} are tuned as follows.
For each problem, the scheme is run for a given stepsize in the finite set
$\left\{10^{-5}, 10^{-4}, 10^{-3}, 10^{-2}, 10^{-1}, 1 \right\}$ for
one-fifth of the total iteration budget. The stepsize that yields the
smallest final mean objective value over five runs is then selected. In
short, the tuning effort slightly exceeds the computational effort in
solving the original problem. The total iteration budget differs across
problems and will be specified in the subsequent problem descriptions. This
tuning stage is crucial for the performance of these methods, as an
arbitrary choice of stepsize may lead to poor results. This is captured by
Figure~\ref{fig:tunestepsize}, which illustrates the objective values
obtained with different stepsizes for \texttt{sgd\_const\_tuned} and
\texttt{adam\_tuned} on three problems: \texttt{Rosenbrock}
\eqref{eq:rosenbrock}, logistic regression with dataset \texttt{rcv1}, and
Two-stage SP \eqref{2stage}. It can be observed that the algorithm
performance is highly sensitive to the choice of stepsize
 and the best stepsize varies across problems.
\begin{figure}[h]
    \centering
    \begin{tabular}{ccc}
    \texttt{Rosenbrock} $(n=50)$ & \texttt{rcv1} & \texttt{TwoStageSP} $(n=50)$ \\
    \includegraphics[width=0.23\textwidth]{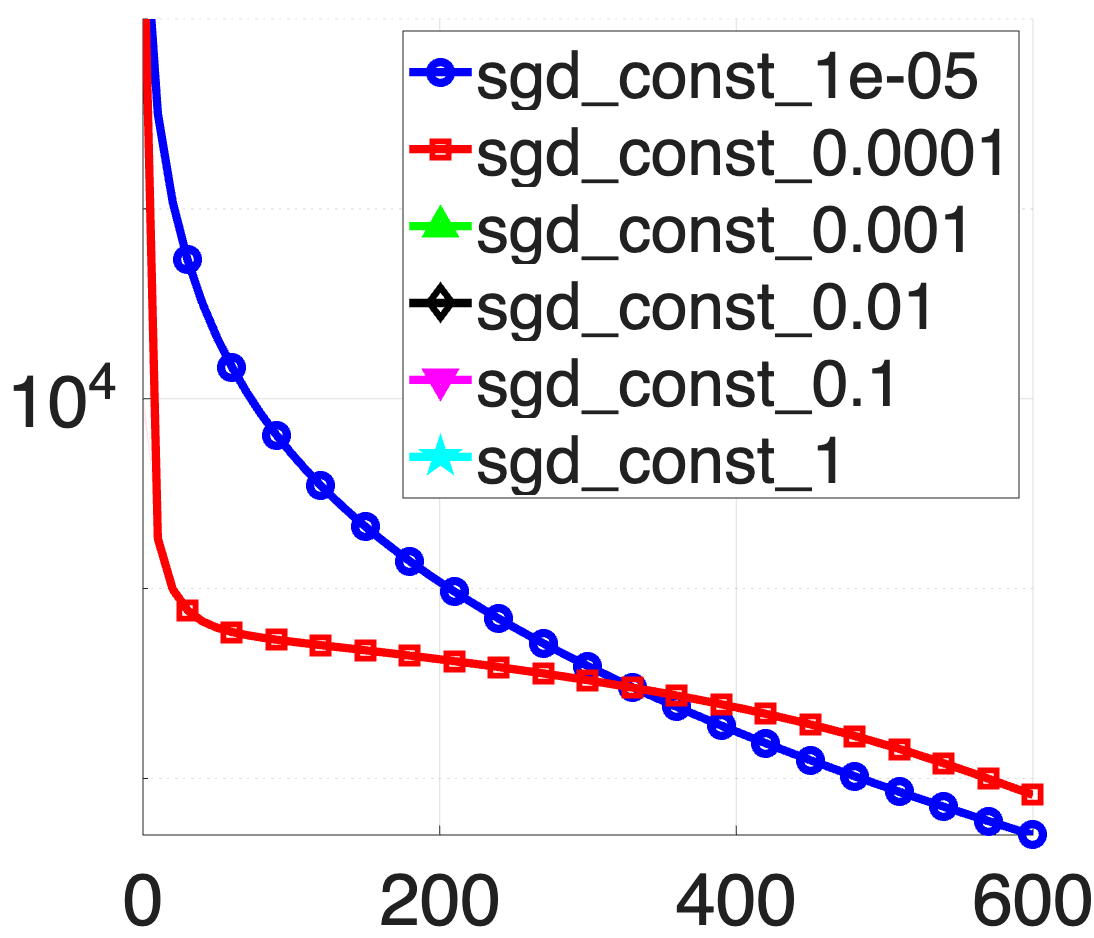}& 
    \includegraphics[width=0.23\textwidth]{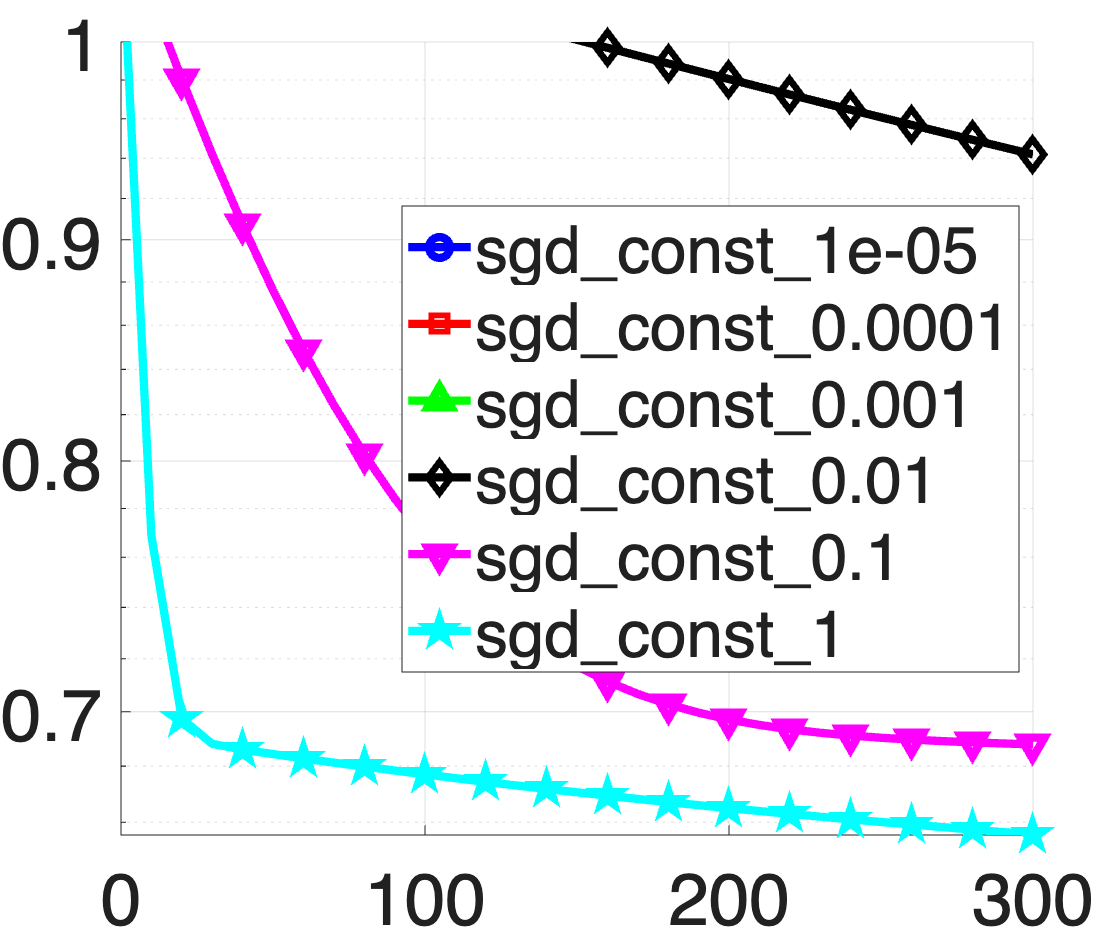}& 
    \includegraphics[width=0.23\textwidth]{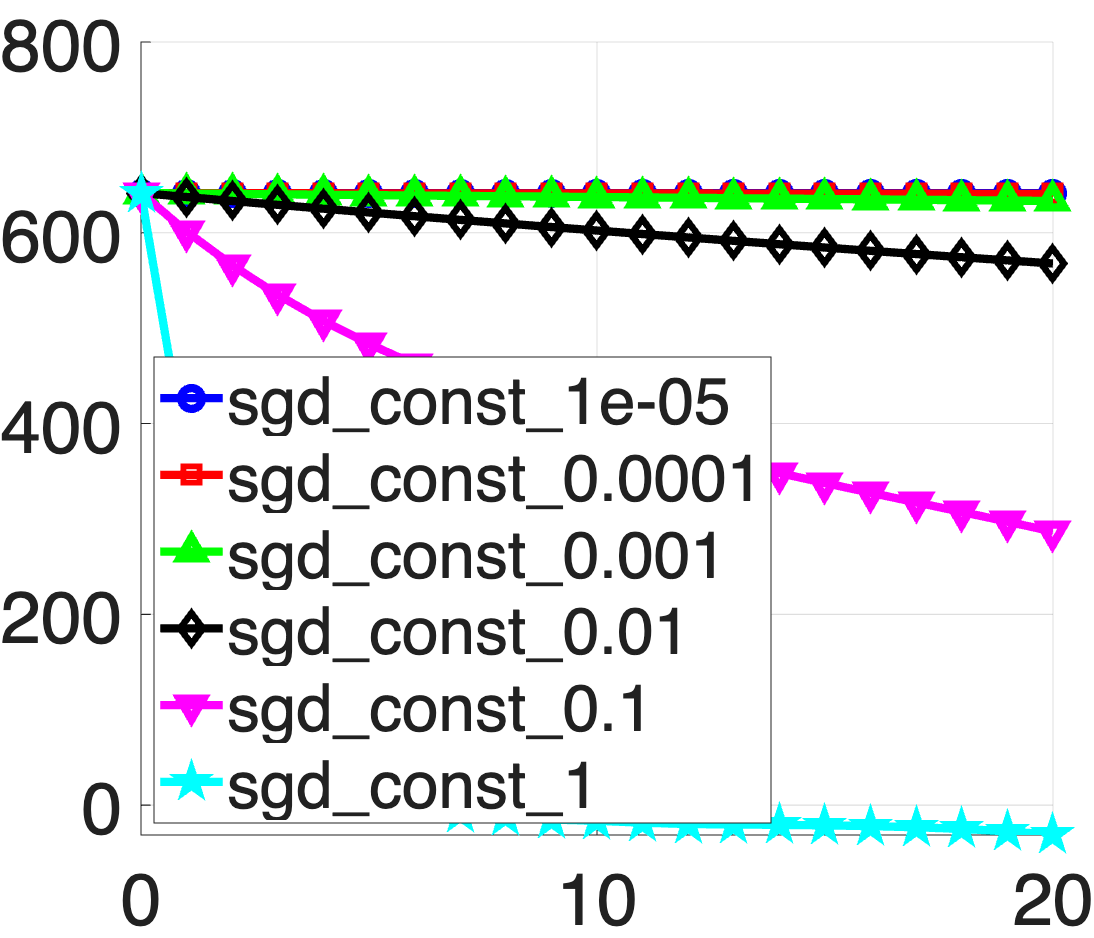}
    \\
    \includegraphics[width=0.23\textwidth]{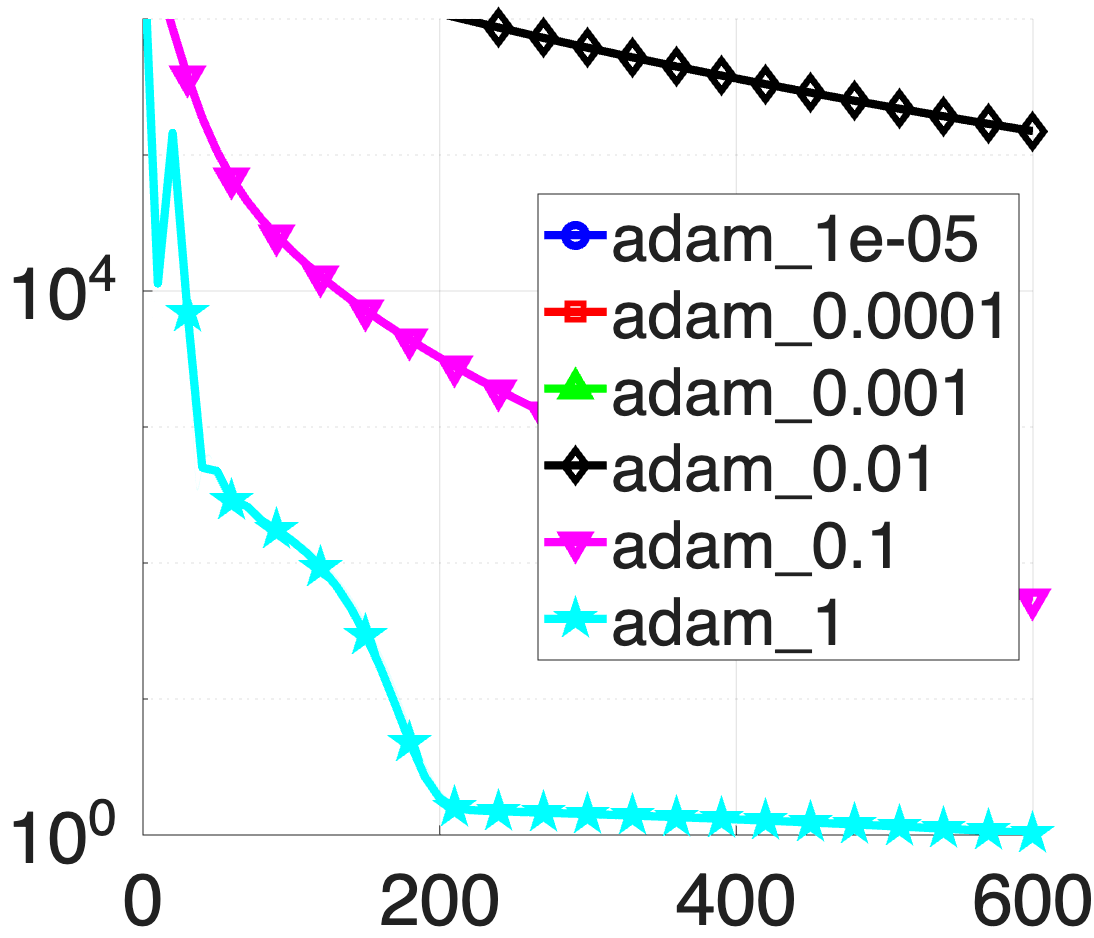} &
    \includegraphics[width=0.23\textwidth]{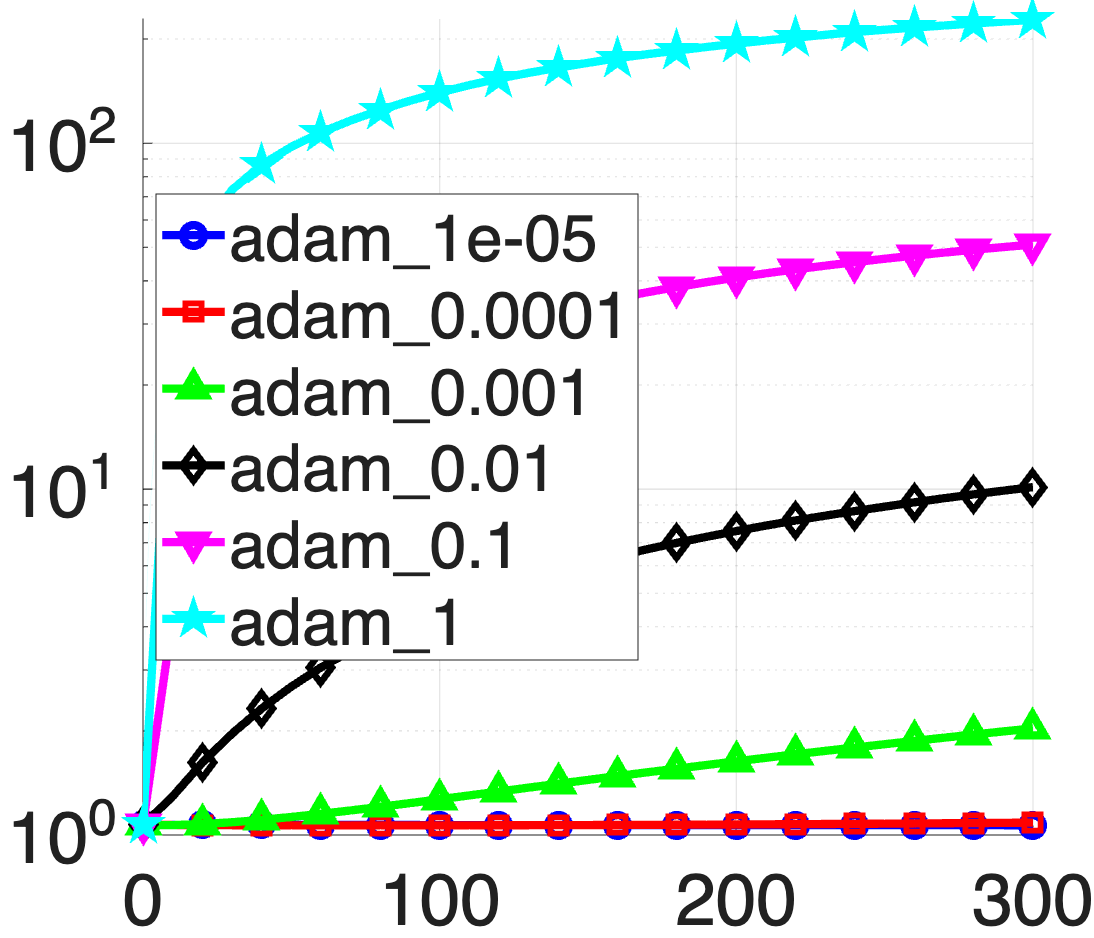} &
    \includegraphics[width=0.23\textwidth]{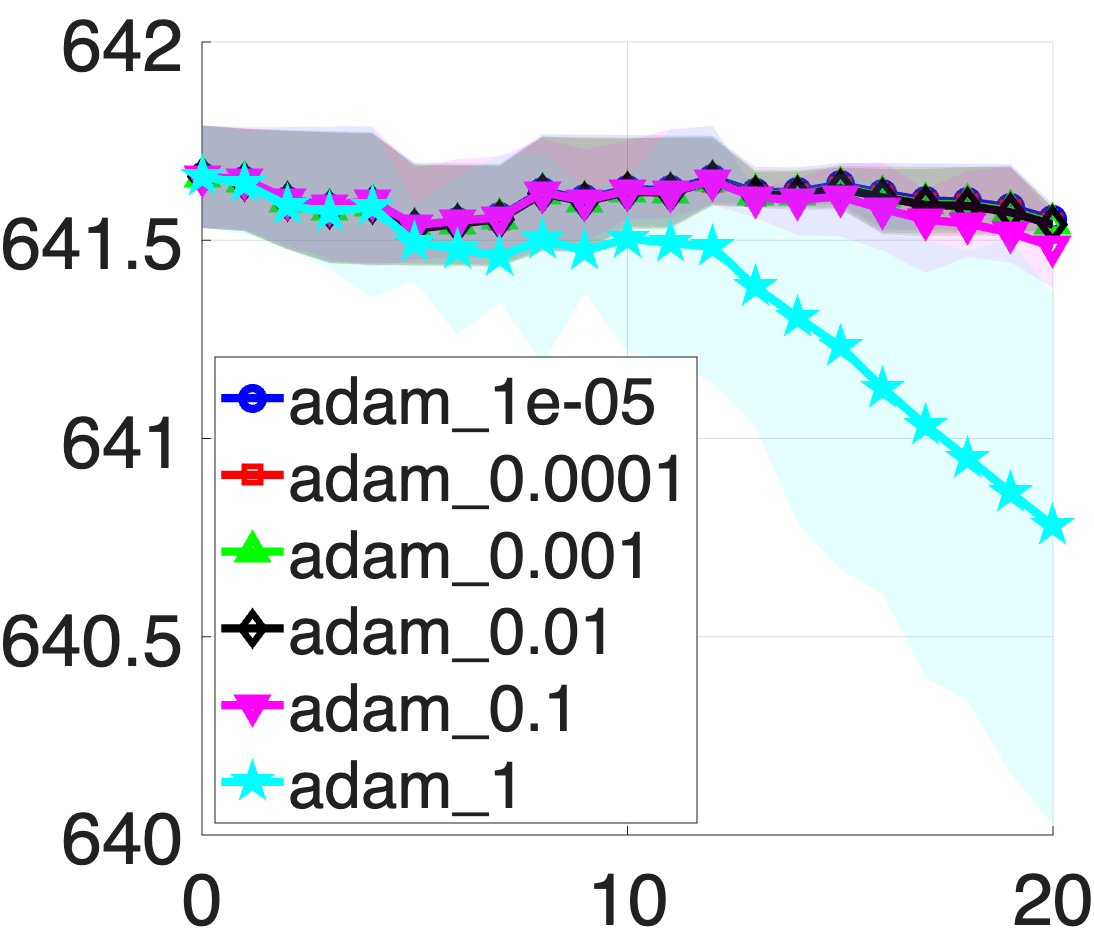}
    \end{tabular}
\caption{objective values (averaged over 5 runs, with bands indicating the minimum and maximum) of different initial stepsizes for \texttt{sgd\_const\_tuned} (top row) and \texttt{adam\_tuned} (bottom row) on three problems. Note that we select a problem-specific y-axis range that reflects the natural scale of the trajectories and displays the trajectories of all schemes as fully as possible. Consequently, the y-axis ranges may differ substantially across problems. The only exception occurs when the trajectory of a scheme diverges, in which case its full trajectory may fall outside the displayed range.} \label{fig:tunestepsize}
\end{figure}
Several points deserve emphasis when studying the results of tuning. First, the impact of tuning is quite dramatic. For instance, the choice of $s=1$ vs $s=0.1$ leads to significantly different behavior for \texttt{Adam} when applied to \texttt{rcv1}. In fact, for other choices of $s$, the resulting trajectories do not even emerge in the graph (implying far poorer performance). Second, the schemes have to be tuned for each problem and results from one problem may have little bearing on the tuning for another problem. Third, we employ the {\bf same} parameter choices of \texttt{SLAM} for all problems and at the cost of repeating ourselves, these choices do not necessitate knowledge of $L$. 

\subsection{Multivariate stochastic Rosenbrock function}\label{sec:5.3}
We consider the box-constrained minimization of the Rosenbrock function  with box constraints given by $X \, \triangleq \, [-10,10]^n$ and
\begin{equation}\label{eq:rosenbrock}
\begin{aligned}
    \phi(x) & \, = \, f(x) + {\bf 1}_X(x) 
 \text{ where}\\
f(x) & \triangleq \E_{\bxi\sim \mathcal{TN}(0,10^2)}[F(x, \bxi)], \mbox{ and } F(x, \bxi) \triangleq \sum_{i=1}^{n-1} \left[ (100 + \bxi)(x_{i+1} - x_i^2)^2 + (1-x_i)^2 \right],
\end{aligned}
\end{equation}
and $\mathcal{TN}(0,10^2)$ is a normal distribution with mean $0$ and standard deviation $10$, truncated to the interval $[-30, 30]$. Indeed, the expectation can be computed explicitly, which is 
\begin{align*}
    f(x) &= \sum_{i=1}^{n-1} \left[ 100(x_{i+1} - x_i^2)^2 + (1-x_i)^2 \right],
\end{align*}
and it is known that the global solution is $x^* = (1, \ldots, 1)^\top$ and
$\phi(x^*) = 0$ for any dimension $n$. We use the unbiased stochastic
gradient and objective estimator of the function $f$, defined as
\eqref{eq:rosenbrock}, to test the performance of our proposed algorithms.
The compactness of $X$ ensures that $\nabla f$ is bounded and $\phi$ is lower
bounded. Together with the bounded support of $\xi$, it also uniformly
bounds the variance of the sampled Hessian and gradient estimators computed over $X$.  
Thus,
Assumptions~\ref{ass:prox}, \ref{ass: F}, and \ref{ass: var} hold. We test
the scheme on dimension $n$ varying in the set $\{2, 10, 50, 100\}$.
Further, 
all schemes employ the same batch size schedule $N_k = 128$ for all $k$. The
initial point is set to $(x_0)_{i=1}^n = 6$. The iteration budget is given
by $K = 1500$ for $n=2,10$, $K=3000$ for $n=50$, and $K=6000$ for $n=100$.
The true objective $\phi(x_k)$ and the squared residual
$\|G_s(x_k)\|^2$ of each scheme are shown in Figure~\ref{fig:rosenbrock}.
Surprisingly, some of the schemes appear to perform poorly on most instances
of this problem (a relatively standard problem in nonlinear programming),
including tuned versions of \texttt{SGD} and \texttt{sls0\_1}.
As mentioned earlier, \texttt{sls0\_1} may lack convergence theory under the choice of $s=1$  since it is
unlikely to satisfy the stepsize restriction $s \le 2/(\rho L)$ in
\cite{Vaswani2019}. Our scheme, \texttt{slam\_1\_p50}, outperforms others
in terms of objective and gradient convergence in most cases, achieving
function values of the order of $1$e-$8$ in the larger problem settings,
except for $n = 100$, where \texttt{SLAM} is seen to be slightly slower than
\texttt{adam-tuned} when $n=100$.
\begin{figure}[h]
    \centering
    \begin{tabular}{cccc}
    $n=2$ & $10$ & $50$ & $100$ \\
    \includegraphics[width=0.23\textwidth]{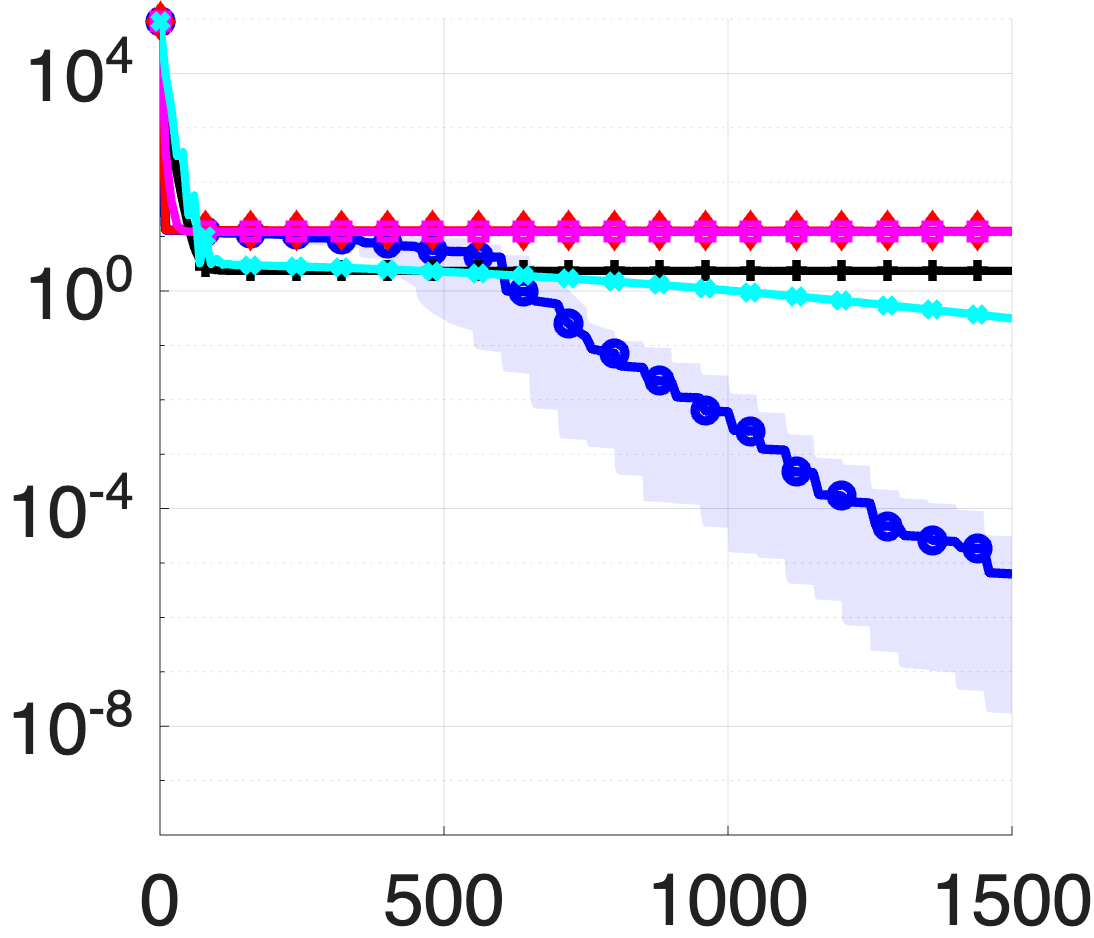}&
    \includegraphics[width=0.23\textwidth]{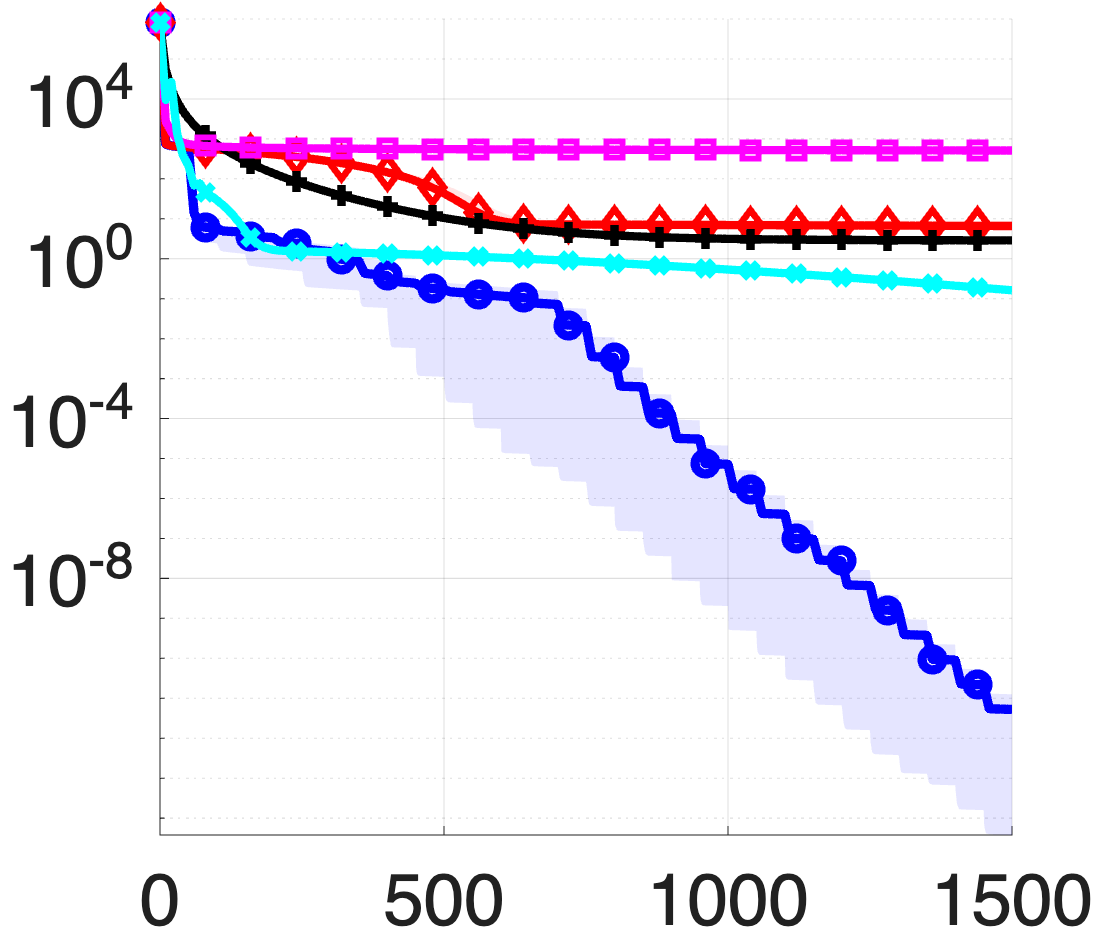}&
    \includegraphics[width=0.23\textwidth]{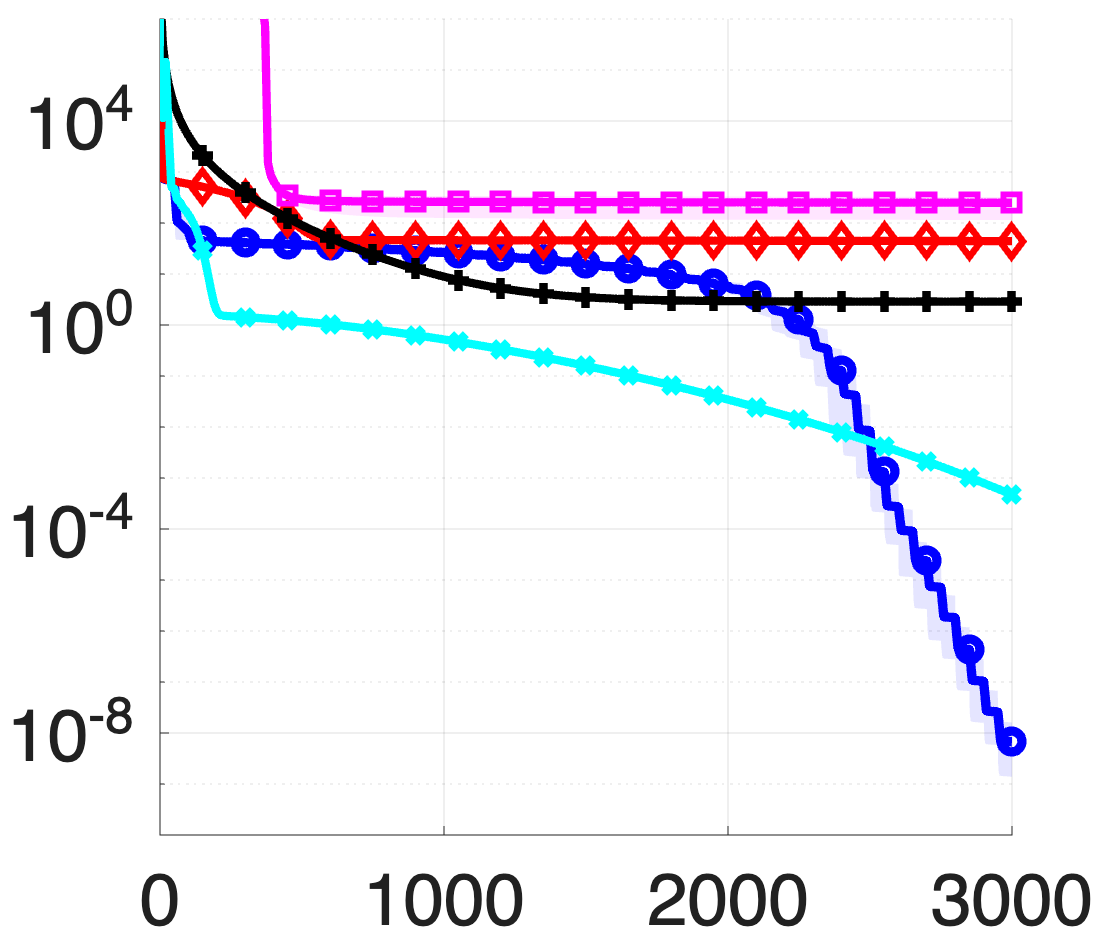}&
    \includegraphics[width=0.23\textwidth]{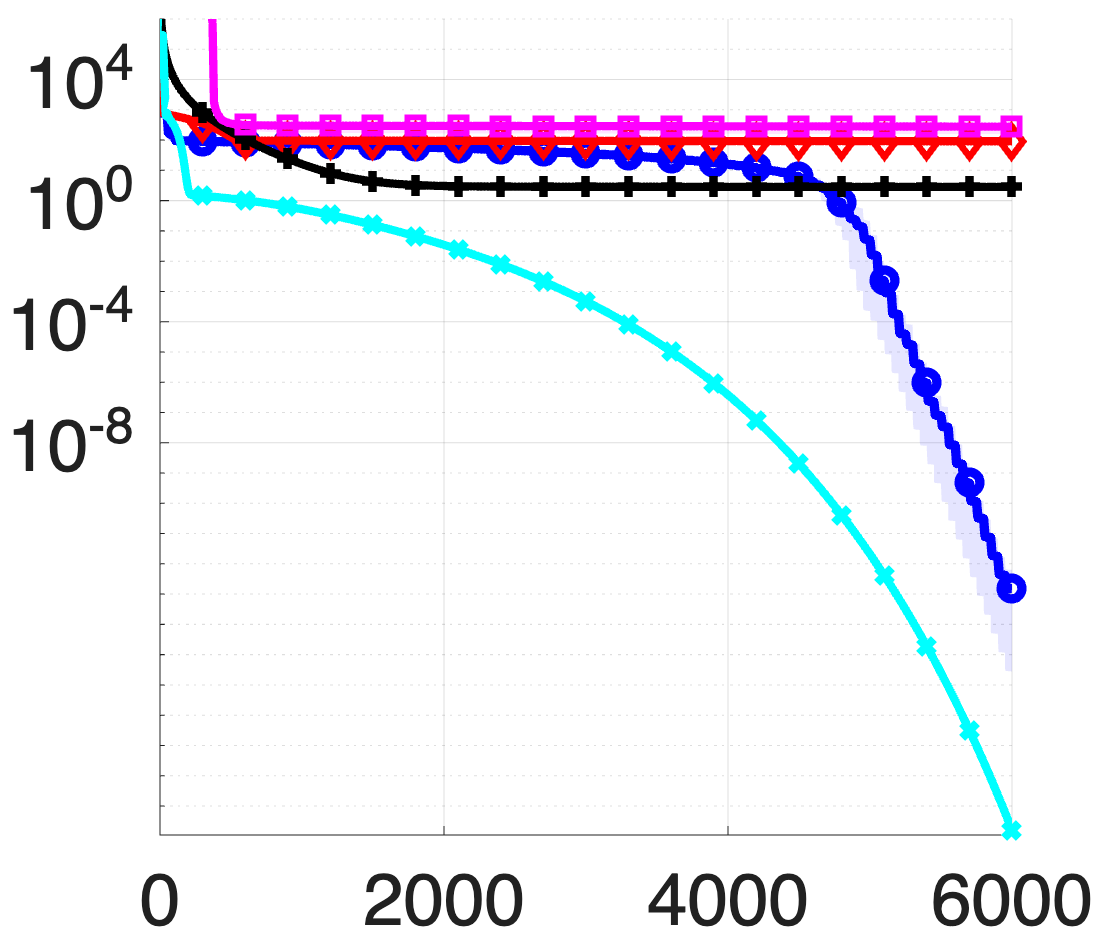}\\
    \includegraphics[width=0.23\textwidth]{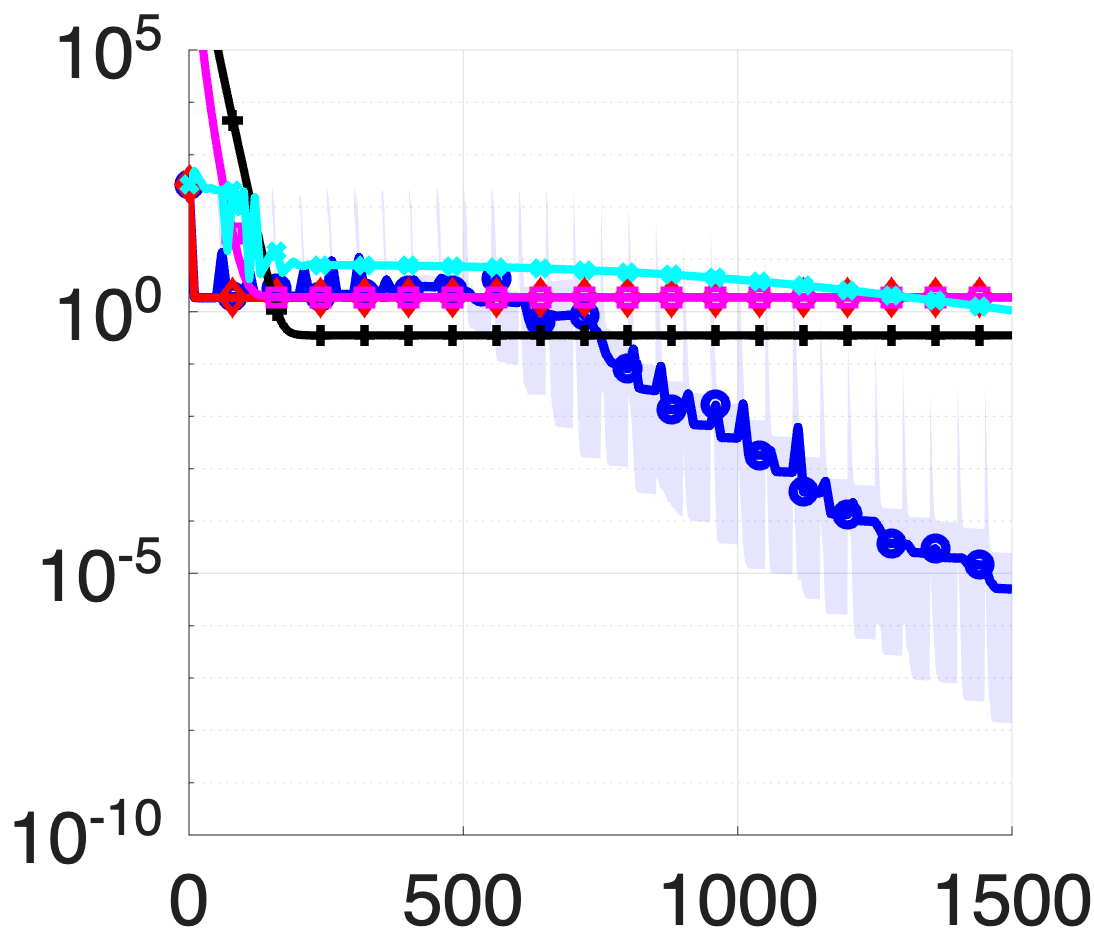}&
    \includegraphics[width=0.23\textwidth]{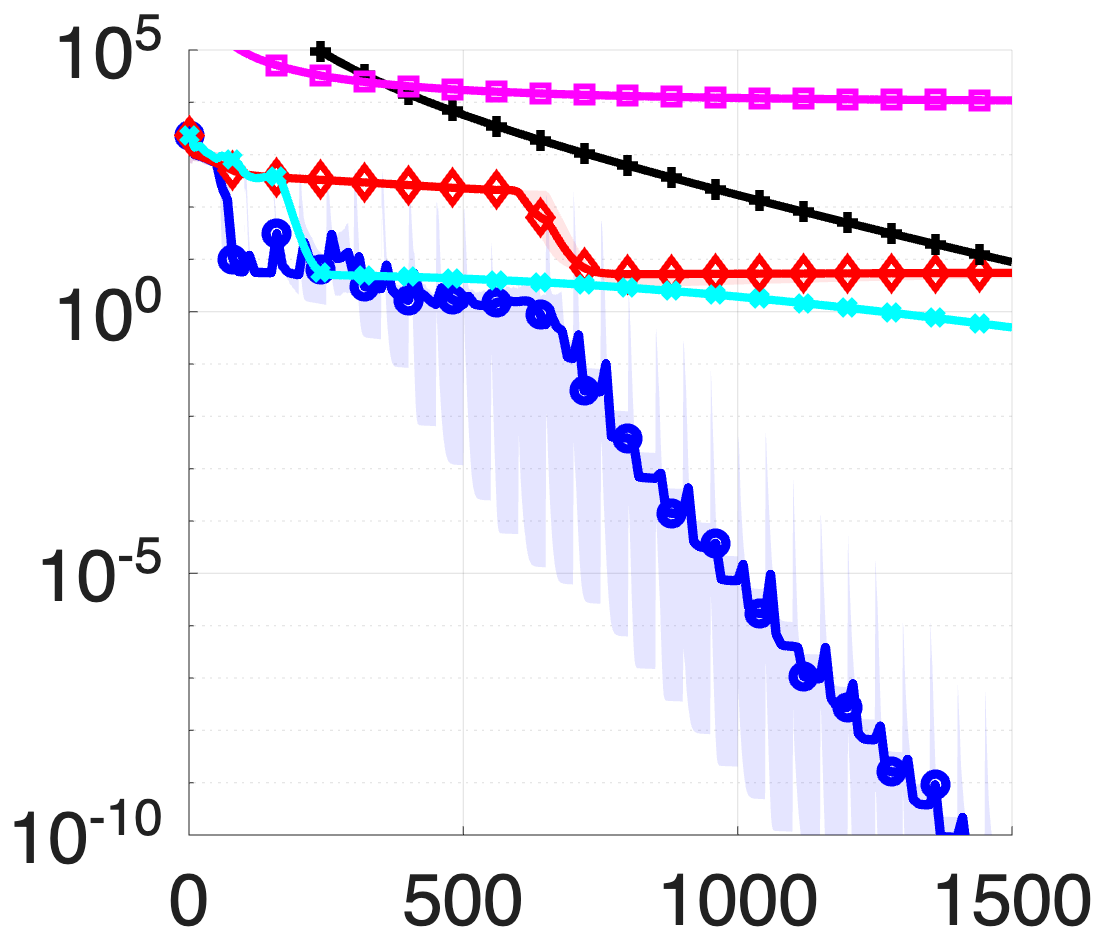}&
    \includegraphics[width=0.23\textwidth]{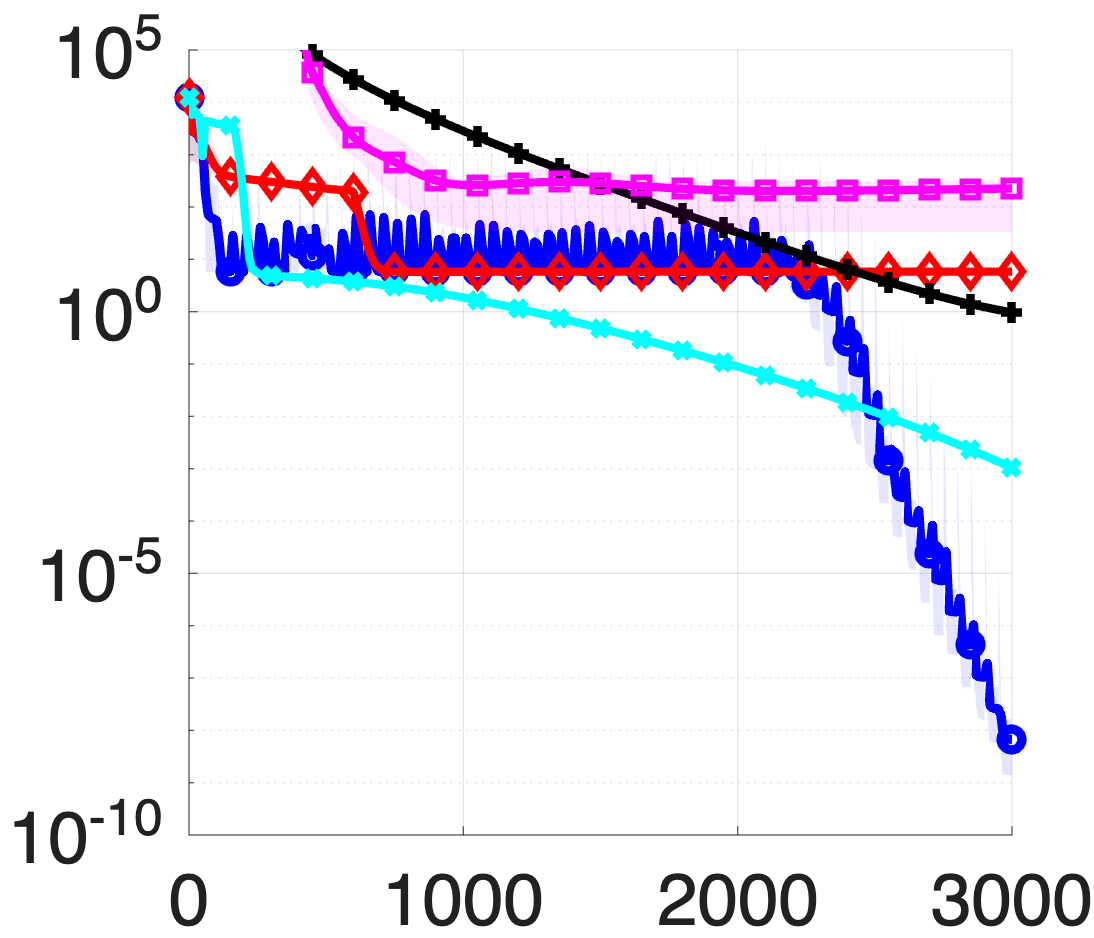}&
    \includegraphics[width=0.23\textwidth]{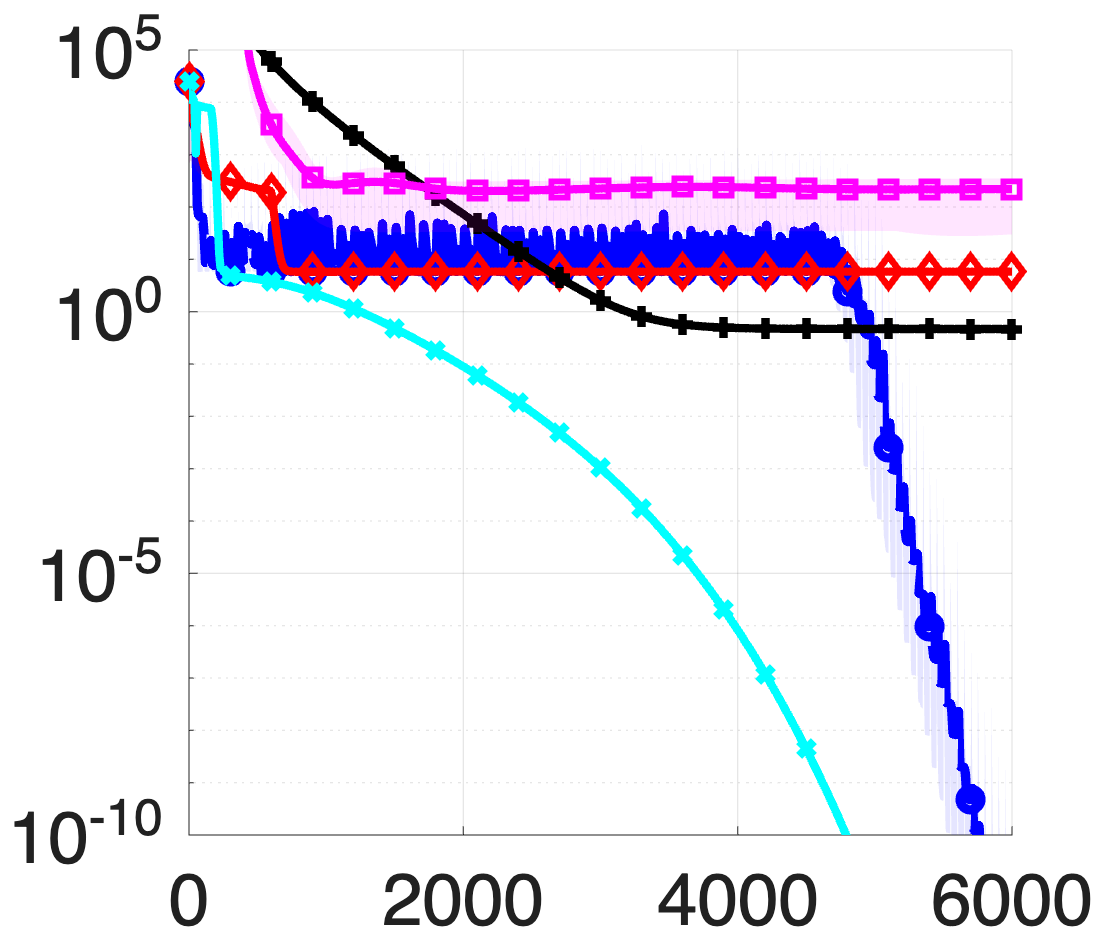}
    \end{tabular}
    \includegraphics[width=0.8\textwidth]{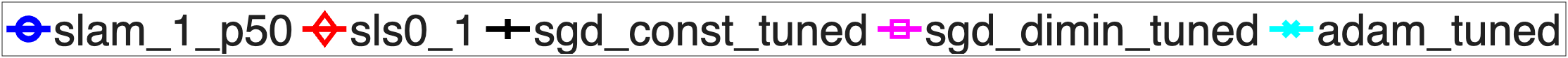}
    \caption{Numerical results of Rosenbrock function with different dimensions. The top and bottom rows are each for true $\phi(x_k)$, and $\|G_s(x_k)\|^2$ (average over 5 runs) versus iteration $k$. The scheme \texttt{sls2} is not included in the comparison since it is implemented for finite sample space problems.}
    \label{fig:rosenbrock}
\end{figure}

\subsection{Machine learning applications}\label{sec:5.4}
In this subsection, we evaluate the proposed method on the training of
machine learning models.  First, we test on regularized logistic regression
on three LIBSVM datasets \cite{Chang2011}   \texttt{duke}, \texttt{ijcnn1},
and \texttt{rcv1} for the binary classification task. The objective function
is $\phi(x) = f(x) + 0.001\|x\|_1$, where  $f(x) \triangleq
\tfrac{1}{m}\sum_{i=1}^m \log(1 + \exp(-y_i \cdot a_i^\top x))$, and
$y_i \in \{-1, 1\}$ and $\{y_i, a_i\}_{i=1}^m$ are the labels and features
from the datasets. Each sample loss is continuously differentiable and
Lipschitz smooth, and its gradient norm is bounded by $\|a_i\|$. Because the
dataset is finite, Assumptions~\ref{ass:prox}, \ref{ass: F}, and \ref{ass:
var} hold. A constant batch size of 16 is used for \texttt{duke} problem
(\texttt{duke}'s full batch is only 44), and $128$ for all other problems.
The iteration budget is $K = 1500$. The initial points are randomly
generated and fixed for each problem.

Second, we train a ResNet18~\cite{He2016} neural network on the
\texttt{CIFAR10} dataset \cite{Krizhevsky2009} for image classification. The
dataset contains 50000 training images of size $32\times32$ in 10 classes.
We use the cross-entropy loss as the expected loss $f$ and impose the box
constraint via $r(x)={\bf 1}_X(x)$, where $X \,\triangleq\,
[-10,10]^n$. Although the compact effective domain $X$ bounds the gradient
at differentiable points, the ReLU activations make the objective nonsmooth.
Thus, Assumptions~\ref{ass:prox}(ii) and~\ref{ass: F} do not hold, and we
include this experiment as an empirical stress test. We use a batch size of
$512$ and a budget of $150$ epochs. This experiment is run on a Google Colab
A100 GPU; all other experiments are implemented in MATLAB R2025b and run on
a MacBook Pro with an M4 processor.

The trajectories of $\phi(x_k)$, $\|G_s(x_k)\|^2$, and $t_k$, together with test accuracy for \texttt{CIFAR10}, of each scheme are shown in Figure~\ref{fig:ml}. Our scheme, \texttt{slam\_1\_p50}, achieves the best performance on \texttt{duke}, \texttt{ijcnn1}, and \texttt{rcv1}, and remains competitive in \texttt{CIFAR10} in objective convergence. Note that for \texttt{ijcnn1} and \texttt{rcv1}, three schemes of \texttt{slam\_1\_p50}, \texttt{sls0\_1}, and \texttt{sgd\_const\_tuned} indeed exhibit identical performance and tied for the best one, as they all adopt a unit step size $t_k = 1$ throughout the iterations, as shown in the bottom row of Figure~\ref{fig:ml}. Generally, \texttt{SLAM} is seen to be amongst the best performing over the entirety of the problem types, outperforming solvers such as \texttt{Adam} in the first three problem types. From Figure~\ref{fig:tunestepsize}, it is apparent that role of tuning is crucial in such settings. 

\begin{figure}[H]
    \centering
    \begin{subfigure}{\textwidth}
        \centering
        \begin{tabular}{ccc}
        \texttt{duke} & \texttt{ijcnn1} & \texttt{rcv1} \\
        \includegraphics[width=0.23\textwidth]{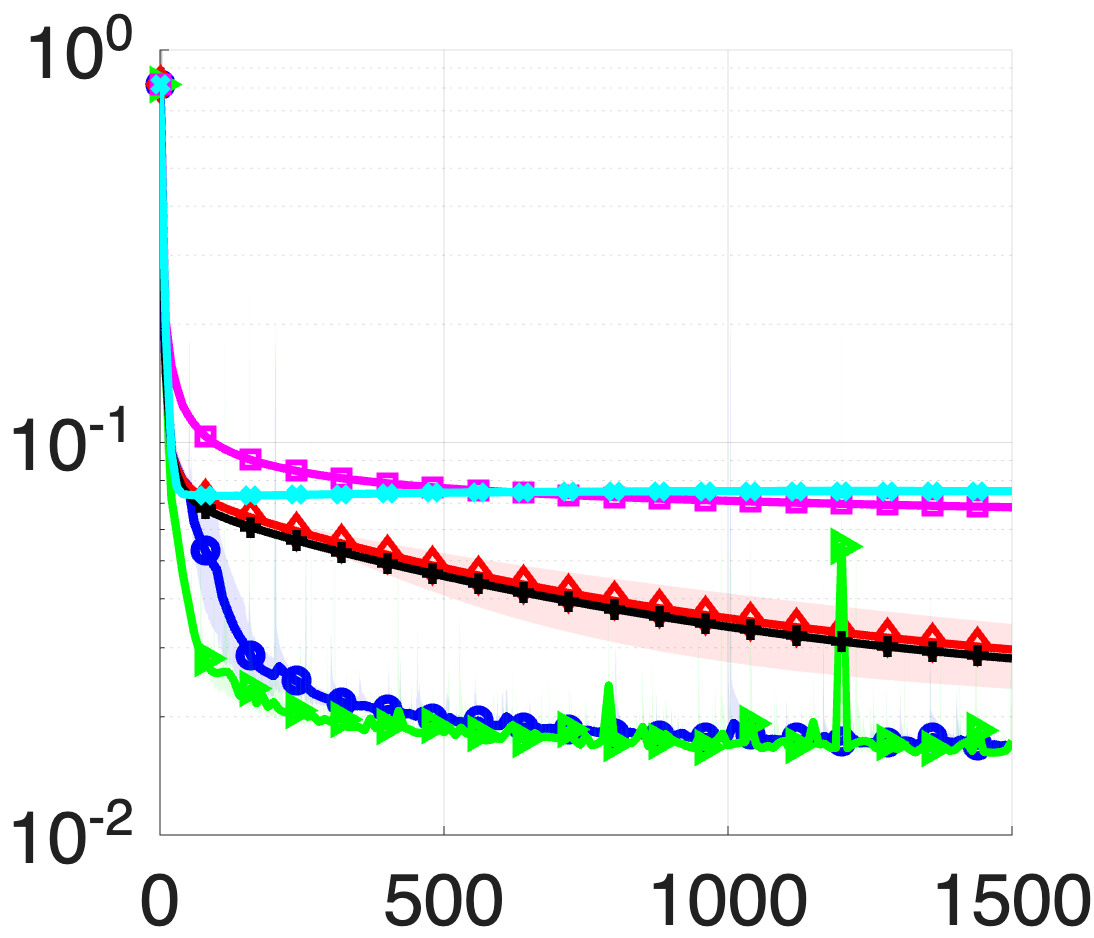}&
        \includegraphics[width=0.23\textwidth]{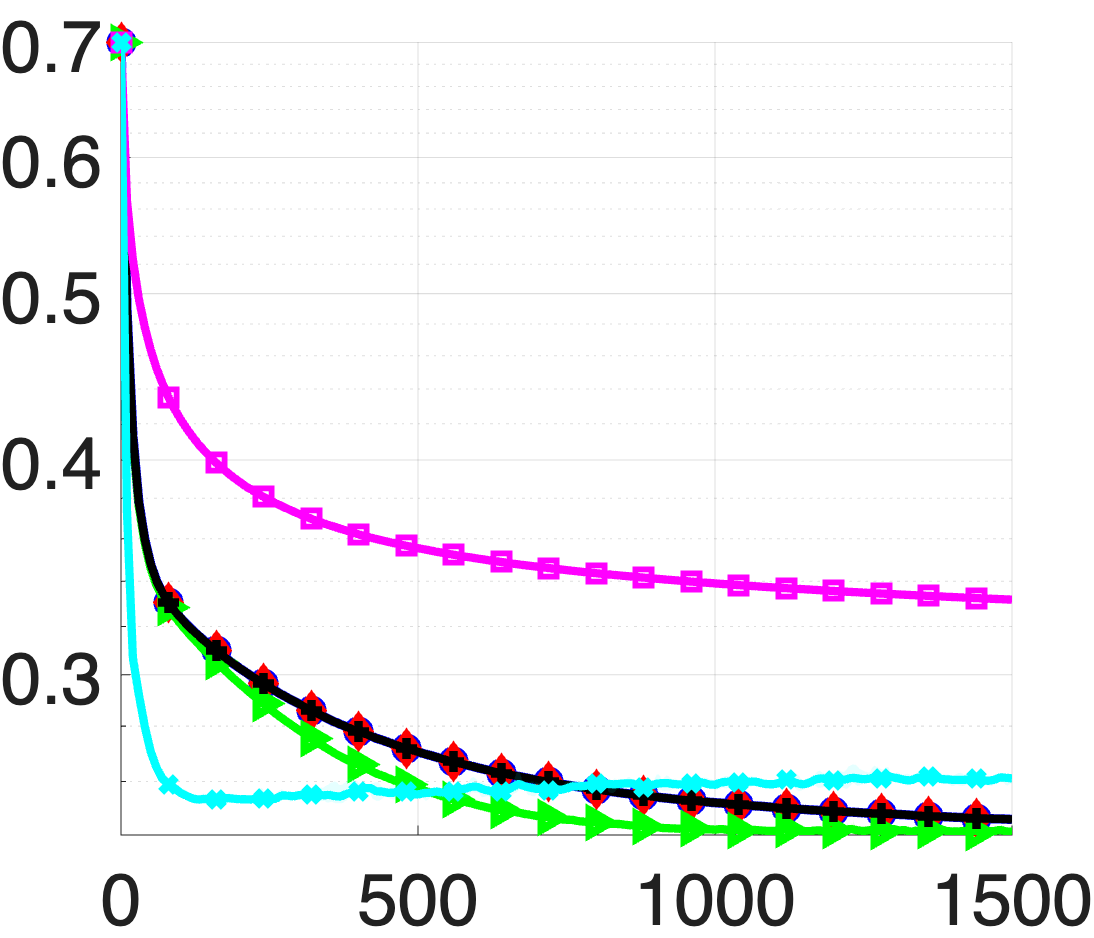}&
        \includegraphics[width=0.23\textwidth]{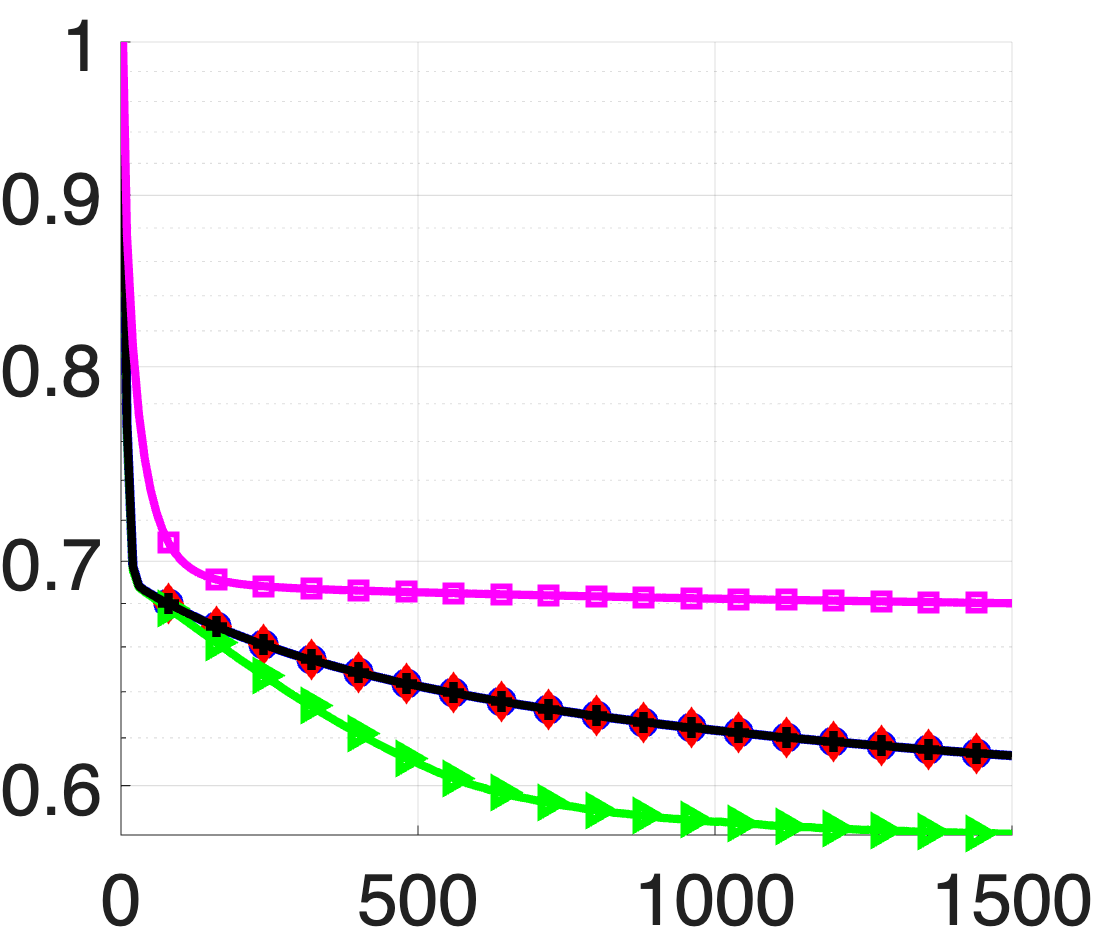}\\
        \includegraphics[width=0.23\textwidth]{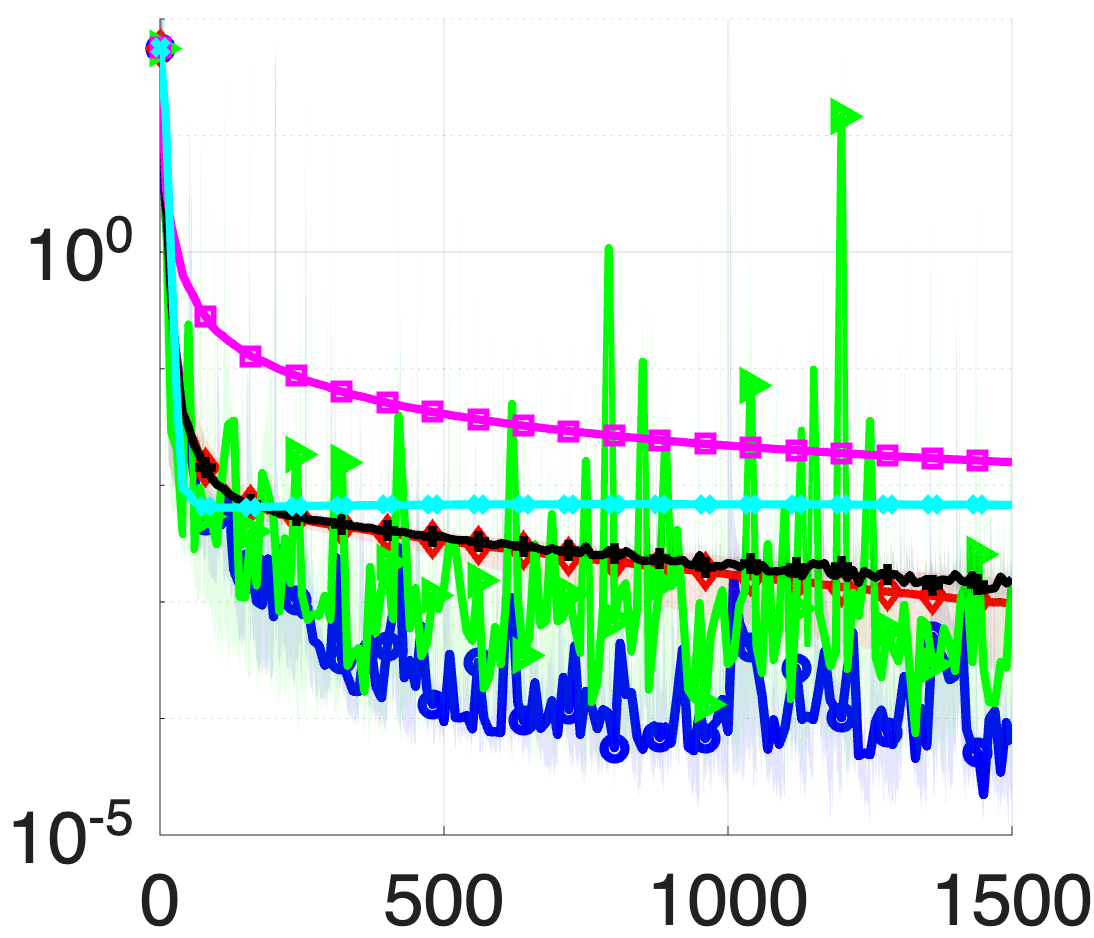}&
        \includegraphics[width=0.23\textwidth]{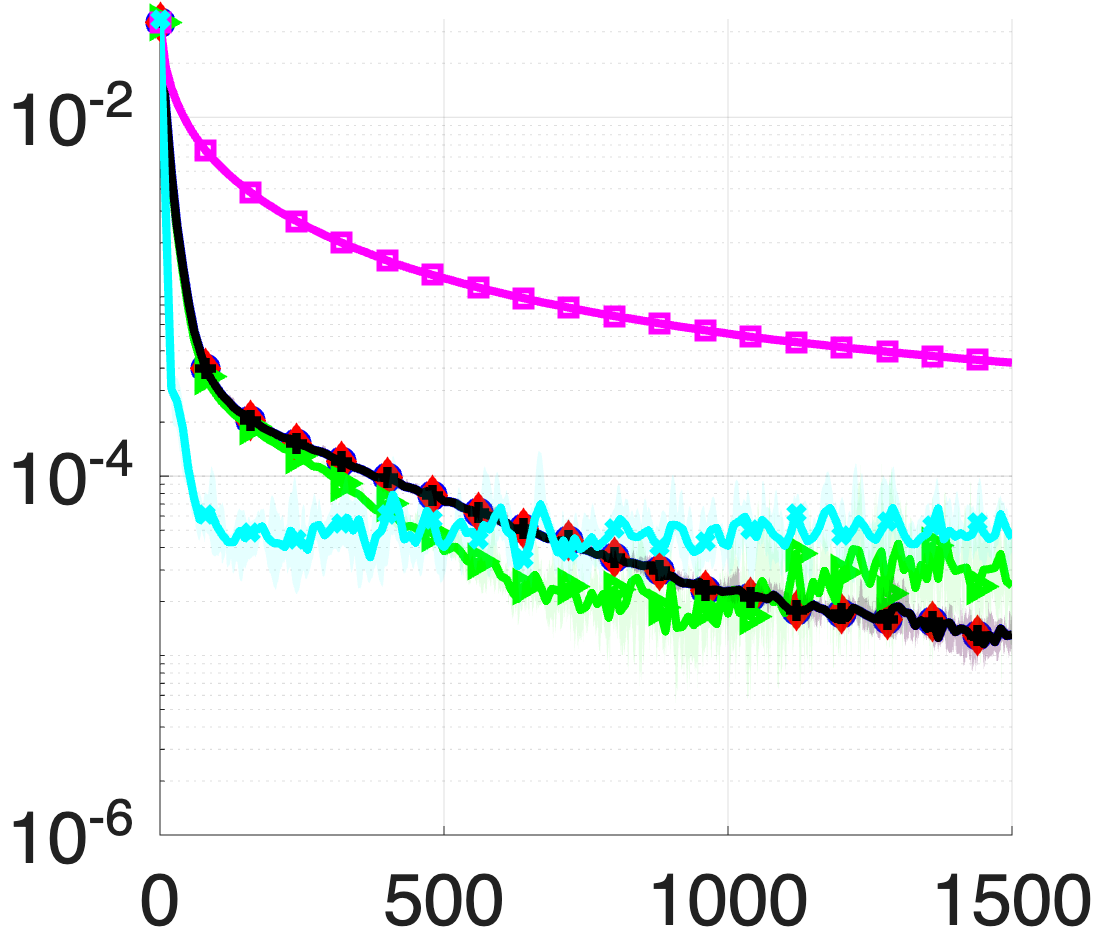}&
        \includegraphics[width=0.23\textwidth]{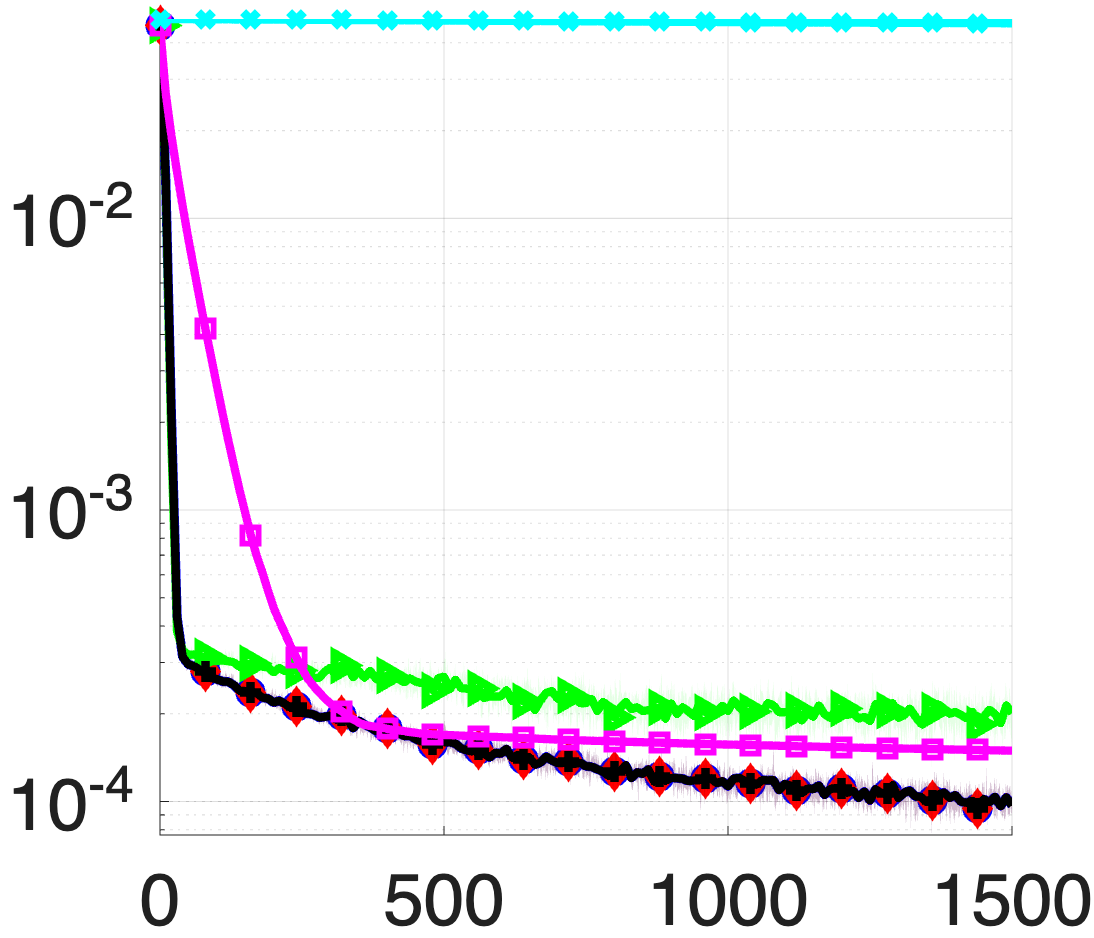}\\
        \includegraphics[width=0.23\textwidth]{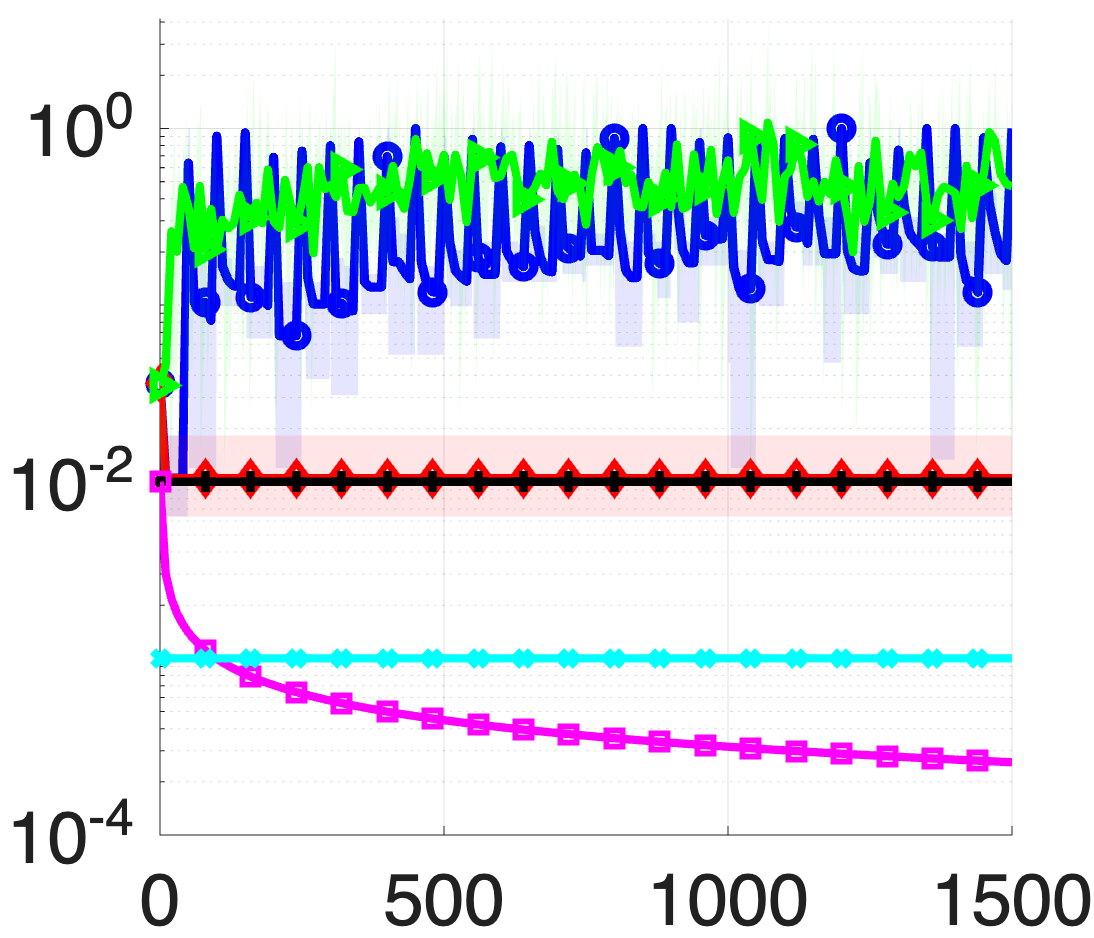}&
        \includegraphics[width=0.23\textwidth]{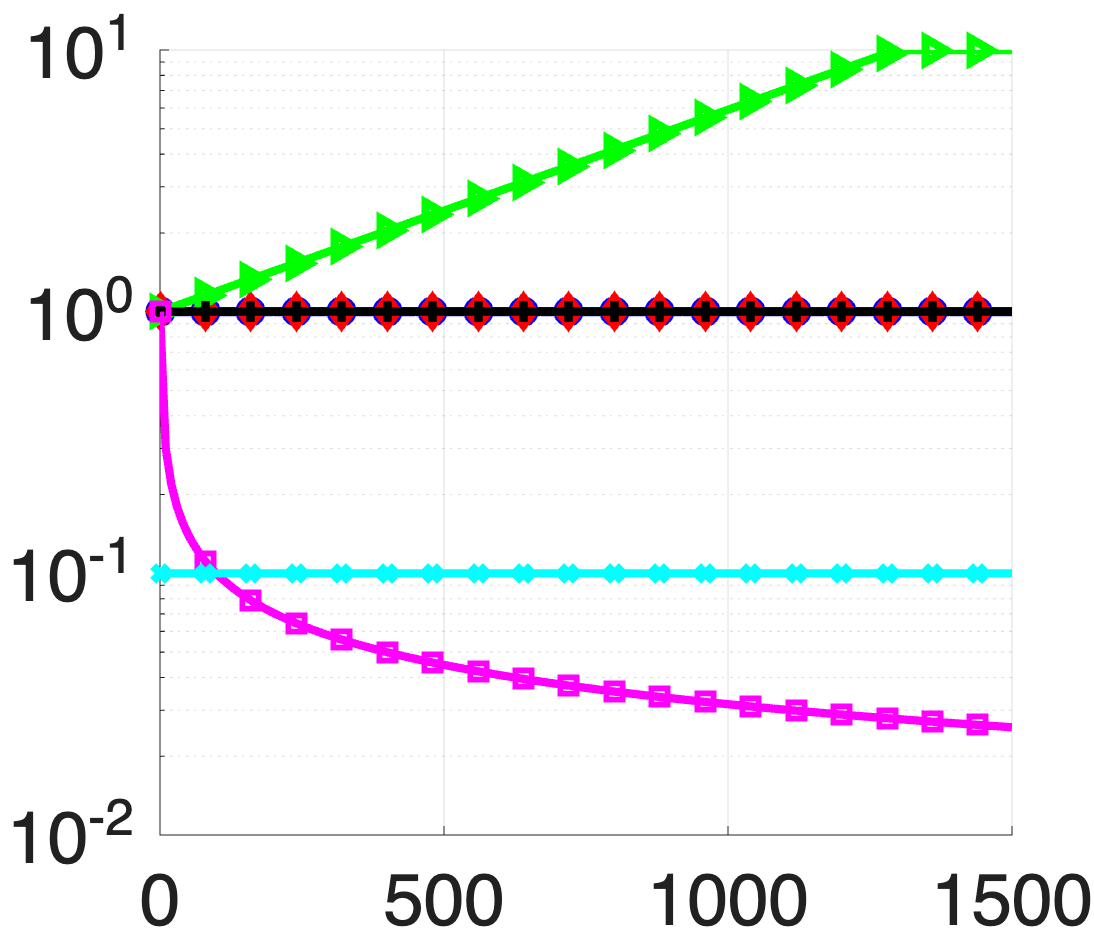}&
        \includegraphics[width=0.23\textwidth]{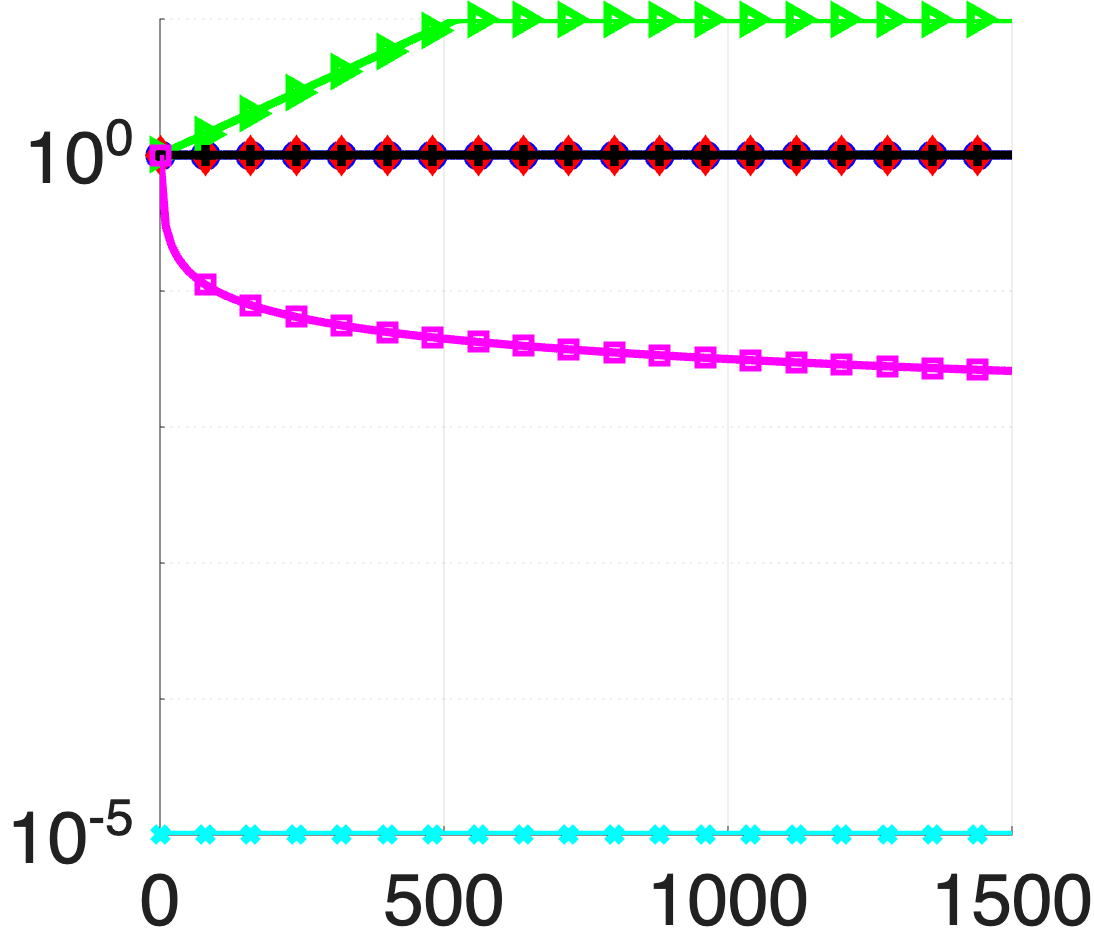}
        \end{tabular}
        \caption{Logistic regression on \texttt{duke}, \texttt{ijcnn1}, and \texttt{rcv1}.}
        \label{fig:ml-logistic}
    \end{subfigure}

    \begin{subfigure}{\textwidth}
        \centering
        \includegraphics[width=0.23\textwidth]{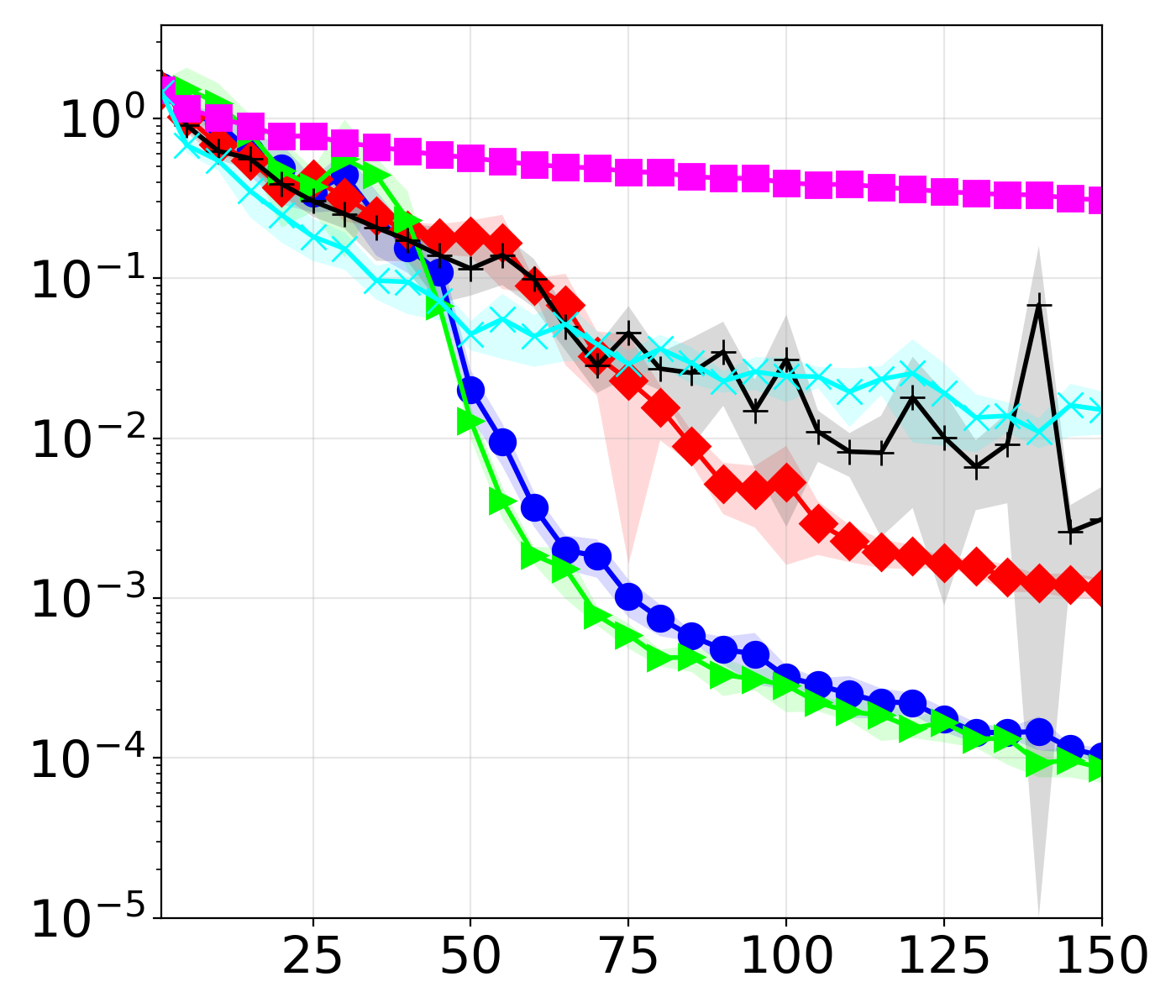}
        \includegraphics[width=0.23\textwidth]{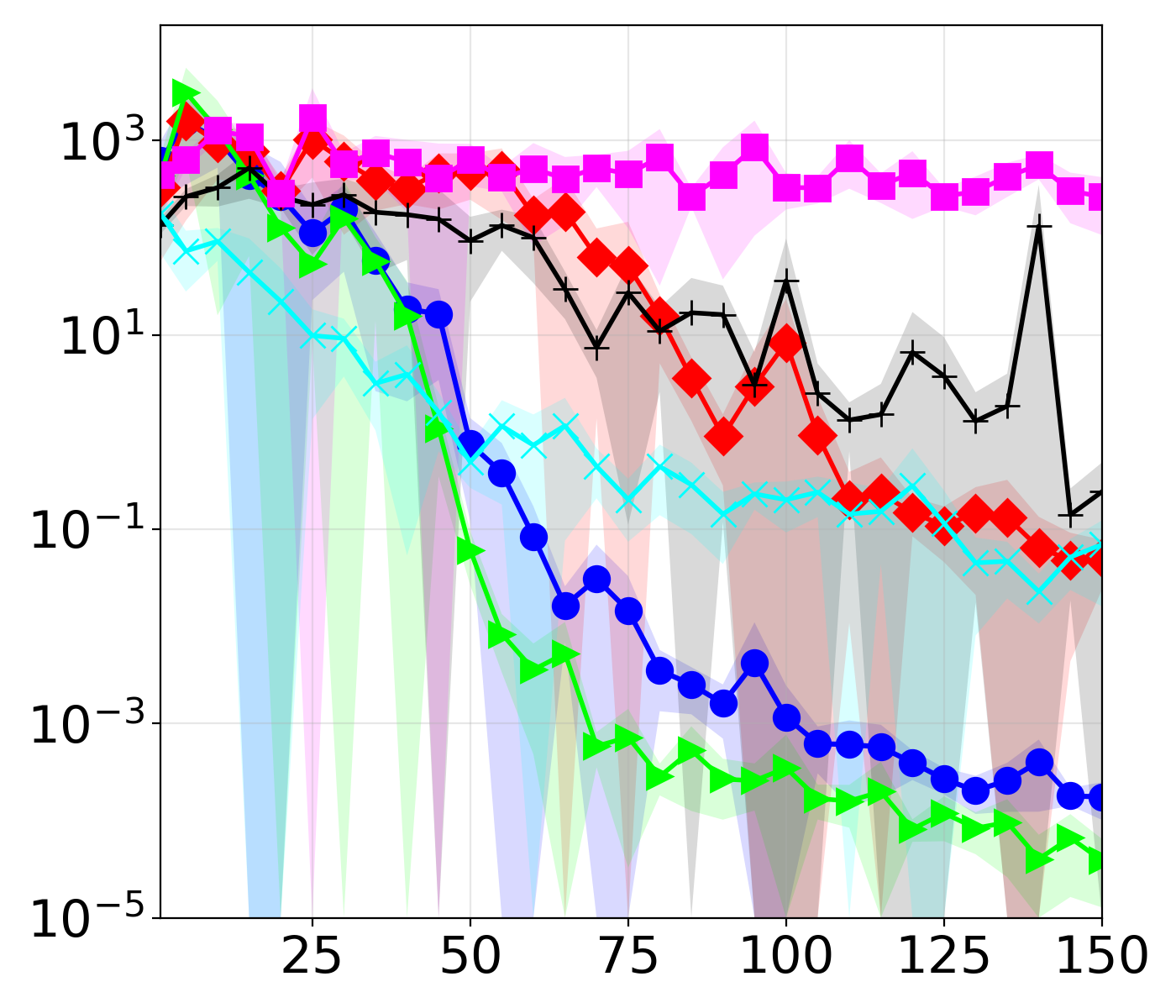}
        \includegraphics[width=0.23\textwidth]{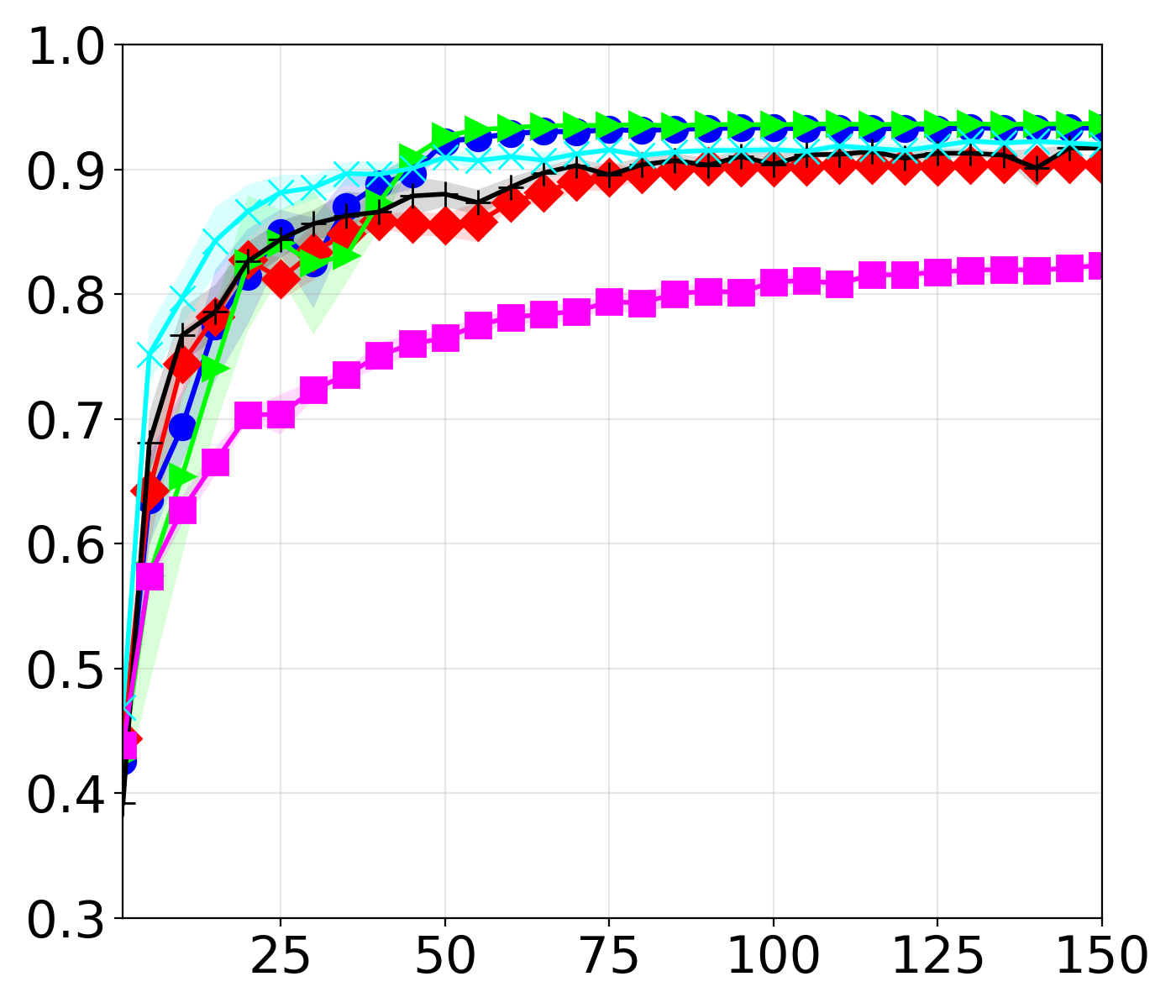}
        \includegraphics[width=0.23\textwidth]{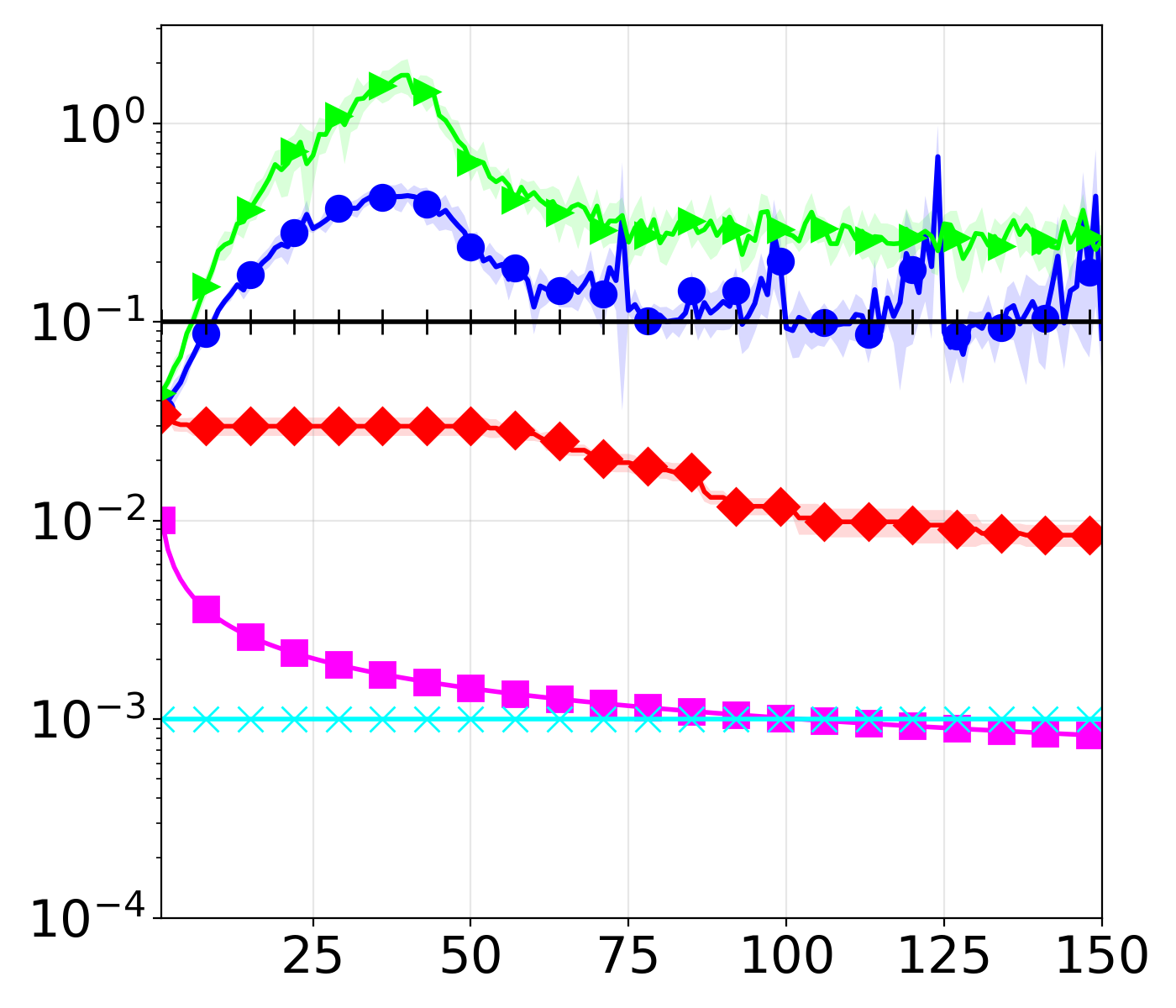}\\
        \caption{ResNet-18 on \texttt{CIFAR10}.
        \vspace{1em}}
        \label{fig:ml-resnet18}
    \end{subfigure}
    \includegraphics[width=0.8\textwidth]{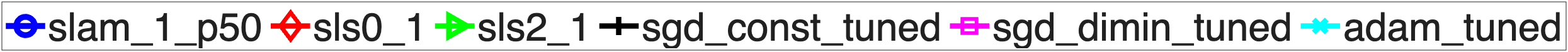}
    \caption{Numerical results for the machine learning problems. In (a), the rows from top to bottom show $\phi(x_k)$, $\|G_s(x_k)\|^2$, and stepsize $t_k$, averaged over five runs. For \texttt{ijcnn1} and \texttt{rcv1}, the trajectories of \texttt{slam\_1\_p50}, \texttt{sls0\_1}, and \texttt{sgd\_const\_tuned} overlap because all three use $t_k=1$ for every $k$. In (b), the columns from left to right show training loss $\phi(x_k)$, $\|G_s(x_k)\|^2$, test accuracy, and $t_k$, averaged over five runs.}
    \label{fig:ml}
\end{figure}

\subsection{Two-stage nonconvex stochastic programming}\label{sec:5.5}    
We now consider the  two-stage stochastic programming problem with nonconvex first-stage objectives and quadratic recourse problems, defined as
\begin{align}\label{2stage}
\begin{aligned}
    &\min_{x \in \R^n} \ f(x) \triangleq H(x) +  \E\left[ \, F(x, \bxi) \, \right]\\
    &\text{subject to} \ x \in \mathcal{X} \triangleq \{x \in \R^n \mid Ax = b, x \ge 0\}
\end{aligned}
\end{align}
where $H: \R^n \to \R$ is a smooth function and $F(x, \bxi)$ is the optimal objective value of the following second-stage problem
\begin{align*}
    F(x, \bxi) = & \min_{y \in \R^m} \ \frac{1}{2}y^\top Q(\bxi)y \\
    & \text{subject to} \ y \in \mathcal{Y}(x, \bxi) \triangleq \{y \in \R^m \mid T(\bxi)x + W(\bxi)y = h(\bxi),\  y\ge 0\}.
\end{align*}
Assuming for all $x \in \mathcal{X}$ and for almost all $\bxi$, $\mathcal{Y}(x, \bxi)$ is nonempty and $Q(\bxi)$ is positive definite. Suppose the dual solution of the  second-stage problem is unique. We revisit this uniqueness requirement as well as satisfaction of Assumptions~\ref{ass: F} and~\ref{ass: var} in Appendix~\ref{app: 2st}. Then $F(\bullet, \bxi)$ is differentiable on an open set containing  $\mathcal{X}$, see \cite[Corollary 2.23]{Shapiro2009}, and 
\begin{align*}
    \nabla F(x, \bxi) = T(\bxi)^\top \pi(x, \bxi),
\end{align*}
where $\pi(x, \bxi)$ is the unique optimal dual solution associated with the constraint $ T(\bxi)x + W(\bxi)y = h(\bxi)$ (Note that the sign of $\pi(x,\bxi)$ in $\nabla F(x, \bxi) = T(\bxi)^\top \pi(x, \bxi)$ depends on the Lagrangian function when adding equality constraints. Here we assume the Lagrangian of the second stage problem is $L(y, \pi, s) = \frac{1}{2}y^\top Qy + \pi^\top (Wy - h + Tx) - y^\top s$.). Therefore, we use 
\begin{align*}
   \tg_k(x) =  \nabla H(x) +  \frac{1}{N_k} \sum_{i = 1}^{N_k} \nabla F(x, \bxi_{k,i})  
\end{align*}
as the stochastic gradient estimator of $f(x)$. The constraint $x \in \mathcal{X}$ in the first-stage problem can be handled by the projection operator over the set $\mathcal{X}$. Empirically, we solve the following two-stage stochastic economic dispatch problem 

\begin{align*}
    \min_{x \in \R^n} \phi(x) &= f(x) + {\bf 1}_{\mathcal{X}}(x), \text{ where }
 f(x)  \triangleq \frac{1}{2} x^\top H_g x + c_g^\top x + \E[F(x, \bxi)] \text{ and }\\
 \mathcal{X} &\triangleq \{x \in \R^{n} \mid x^\top \mathbf{1} = \bar{D} , \ \delta  \le x_i \le (1-\delta)\mathrm{Cap}_i \quad \forall i = 1, \ldots, n \}
\end{align*}

and the second-stage problem is
\begin{align*}
    F(x, \bxi) = \min_{u^{+}, u^{-},\alpha^{+}, \alpha^{-}, s, v}  & \frac{1}{2} \big( (u^{+})^\top \mathrm{diag}(c_u)u^+ + (u^{-})^\top \mathrm{diag}(c_d)u^- \\
     + (\alpha^{+})^\top & \mathrm{diag}(10^{-4}c_u)\alpha^+  + (\alpha^{-})^\top \mathrm{diag}(10^{-4}c_d)\alpha^- + c_s s^2 + c_v v^2 \big)  \\
    \text{s.t. } (x + u^{+} - u^{-})^\top \mathbf{1} &= D(\bxi) - s + v \\
    x_i + u_i^{+} + \alpha^{+}_i &= \mathrm{Cap}_i \quad\  \forall i = 1, \ldots, n  \\
    x_i - u_i^{-} - \alpha^{-}_i& = 0 \quad \quad \quad \forall i = 1, \ldots, n  \\
    u^{+}, u^{-}, \alpha^{+}, \alpha^{-}, s, v &\ge 0,
\end{align*}    
where the decision variables are given by 

\begin{tabular}{ll}
   $\bullet$  $x \in \R^{n}$ & base generation in the first stage \\
   $\bullet$ $u^{+}, u^{-} \in \R^{n}$ & up and down redispatch in the second stage \\
   $\bullet$ $\alpha^{+}, \alpha^{-} \in \R^{n}$ & slack variables in the second stage \\
   $\bullet$ $s \in \R$ & load shedding in the second stage \\
   $\bullet$ $v \in \R$ & generation spillage in the second stage 
\end{tabular}

\noindent and the parameters are defined as

\begin{tabular}{ll}
   $\bullet$ $n \in \N$ & number of generators \\
   $\bullet$ $H_g \in \R^{n \times n}$, $c_g\in \R^{n}$ & generation cost parameters, $H_g$ is indefinite\\
   $\bullet$ $\mathrm{Cap} \in \R^{n}$ & capacity of generators \\
  $\bullet$  $\bar{D} \in \R$ & predicted demand \\
  $\bullet$  $\delta \in (0,1)$ & buffer to capacity limit for first-stage\\
  $\bullet$  $c_u, c_d \in \R^{n}$ & positive second-stage cost of up and down dispatch \\
   $\bullet$ $c_s, c_v \in \R$ & positive load shedding and spillage cost\\
  $\bullet$  $D(\bxi) \in \R$ & actual demand at scenario $\bxi$. 
\end{tabular}
\medskip

\noindent The parameters of the second-stage problem can be compactly defined as
\begin{align*}
y&=\begin{pmatrix}(u^{+})^\top &
(u^{-})^\top &
(\alpha^{+})^\top &
(\alpha^{-})^\top &
s &
v \end{pmatrix}^\top,\\
Q(\bxi) &=  \mathrm{diag}\begin{pmatrix}
    (c_u)^\top& (c_d)^\top& (10^{-4}c_u)^\top & (10^{-4}c_d)^\top & c_s & c_v
\end{pmatrix},\\
   W(\bxi) &= \begin{pmatrix}
    \mathbf{1}^\top & -\mathbf{1}^\top & \mathbf{0}^\top & \mathbf{0}^\top & 1 & -1 \\
    I_{n} & 0 & I_{n} & 0 & 0 & 0 \\
    0 & -I_{n}    & 0 & -I_{n} & 0 & 0 \\
\end{pmatrix} , \quad h(\bxi) = \begin{pmatrix}
    D(\bxi) \\
    \mathrm{Cap} \\
    0 \\
    \end{pmatrix}, 
    \quad \text{and } T(\bxi)= \begin{pmatrix}
    \mathbf{1}^\top \\
    I_{n} \\
    I_{n} \\
\end{pmatrix}.
\end{align*}
Note that $Q(\bxi)$ is positive definite, and the load shedding variables $s$ and spillage variable $v$ guarantee the feasibility of the second-stage problem for all $x \in \mathcal{X}$ and almost all $\bxi$ while their unit cost $c_s$ and $c_v$ are maintained as high to discourage load shedding and spillage. This ensures that the second-stage problem has a unique primal solution as well as a unique active set. Furthermore, the matrix $W(\bxi)$ is always full row rank. 
Furthermore, we choose $H_g$ to be indefinite to make the first-stage problem nonconvex. We generate $D(\bxi)$ from a normal distribution with mean $\bar{D}$ and standard deviation $\sigma$, truncated to the interval $\bar{D} \pm 3\sigma$. We set the value of problem parameters as follows. We conducted tests on $n \in \{10, 50, 100\}$. $H_g$ is a diagonal matrix with the diagonal elements being generated from uniform distribution between $-1$ and $1$, $c_g = [1, 1, \ldots, 1]$, $\bar{D} = 4n$, $\mathrm{Cap}_i  = 5+ 0.2i$ for $i \in \{1, \ldots, n\}$, $\delta = 0.1$, $c_u = [2,2,\ldots,2]$, $c_d = [2,2,\ldots,2]$, $c_s = 20n$, $c_v = 10n$, and $\sigma = 5$. Under these parameter choices, $\mathcal{X}$ is nonempty and compact; in particular, $x=4\mathbf{1}$ is feasible. Under the stated recourse regularity and integrability conditions,  $f$ is continuously differentiable on an open set containing $\mathcal{X}$, so $\nabla f$ is bounded on $\mathcal{X}=\mathrm{dom}(r)$. {Moreover, convexity of $F(x,\xi)$ for each $\xi$ and well-definedness of $\mathbb{E}[F(x,\bxi)]$ imply unbiasedness (see~\cite[Ch.~7]{Shapiro2009});} Theorem~\ref{thm: lip-2st} establishes a Lipschitz constant uniform over all realizations $\xi$, and the bounded support of $D(\bxi)$ together with \eqref{kktsol} yields a uniformly bounded stochastic gradient and hence bounded variance. Consequently, Assumptions~\ref{ass:prox}, \ref{ass: F}, and \ref{ass: var} hold for the two-stage stochastic economic dispatch problem.


The initial point $x_0$ is set to be the solution of the first-stage problem without the second-stage recourse, i.e., when $\E[F(x, \bxi)] \equiv 0$. The initial point and the second stage problem are solved by the solver \texttt{Gurobi} 12.0.3. Across all the compared schemes, we employ a batch size of $N_k = 16$,  while the iteration budget is $K = 100$. Unlike the stochastic Rosenbrock function, the expectation $\E[F(x, \bxi)]$ cannot be computed explicitly. Therefore, we employ a sample size of $128$ to estimate the true objective value and residual norm square $\|G_s(x_k)\|^2$ at each iteration for comparison. Figure~\ref{fig:tsp} shows the estimated true objective and residual norm square versus iteration $k$ for problems with different generators. In this setting, the scheme \texttt{slam\_1\_p50} and \texttt{sls0\_1} perform similarly to \texttt{sgd\_const\_tuned} for the first several iterations since \texttt{sgd\_const\_tuned} is indeed using $1$ as the constant stepsize while the linesearch schemes \texttt{slam\_1\_p50} and \texttt{sls0\_1} also use $1$ as the stepsize accepted by the Armijo condition, until convergence which appears to ensue within the first 20 or so iterations in most instances. 
\begin{figure}[!t]
    \centering
    \begin{tabular}{ccc}
    $n=10$ & $50$ & $100$ \\
    \includegraphics[width=0.23\textwidth]{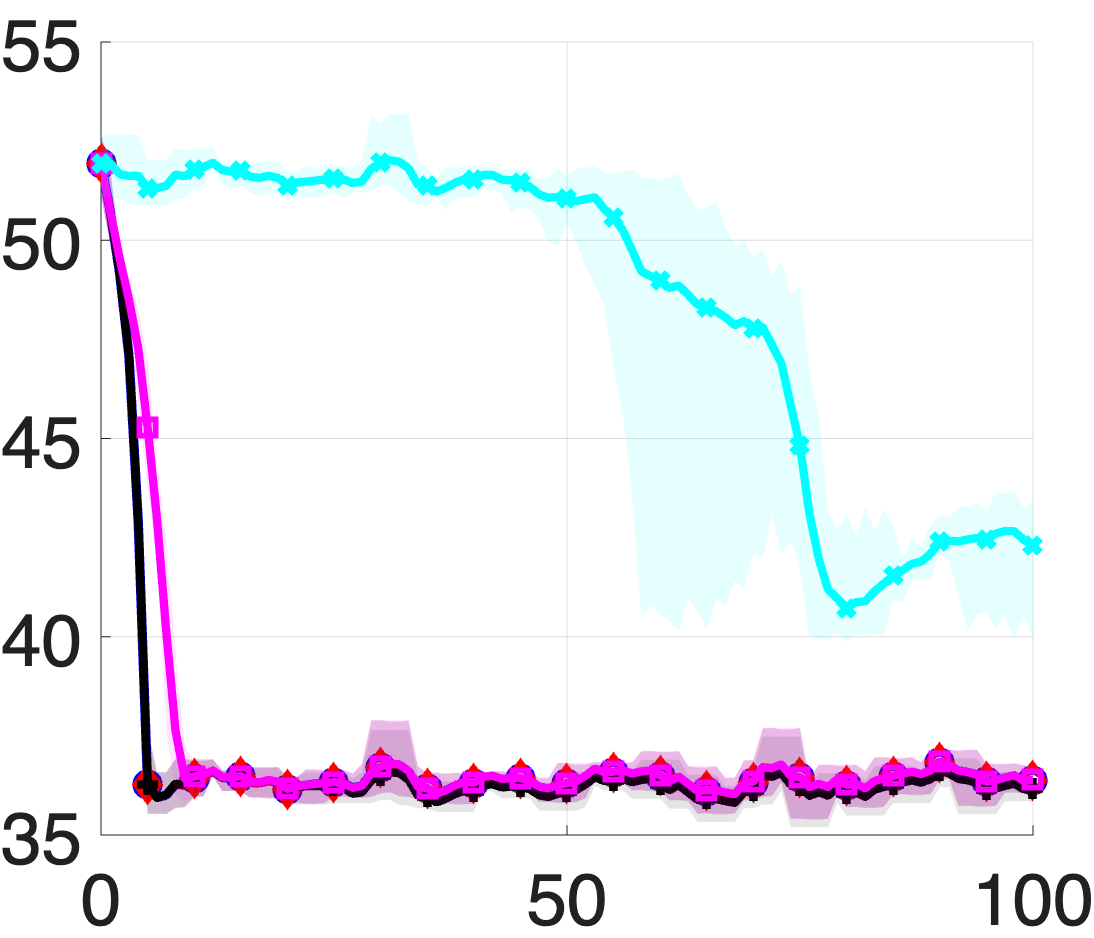}&
    \includegraphics[width=0.23\textwidth]{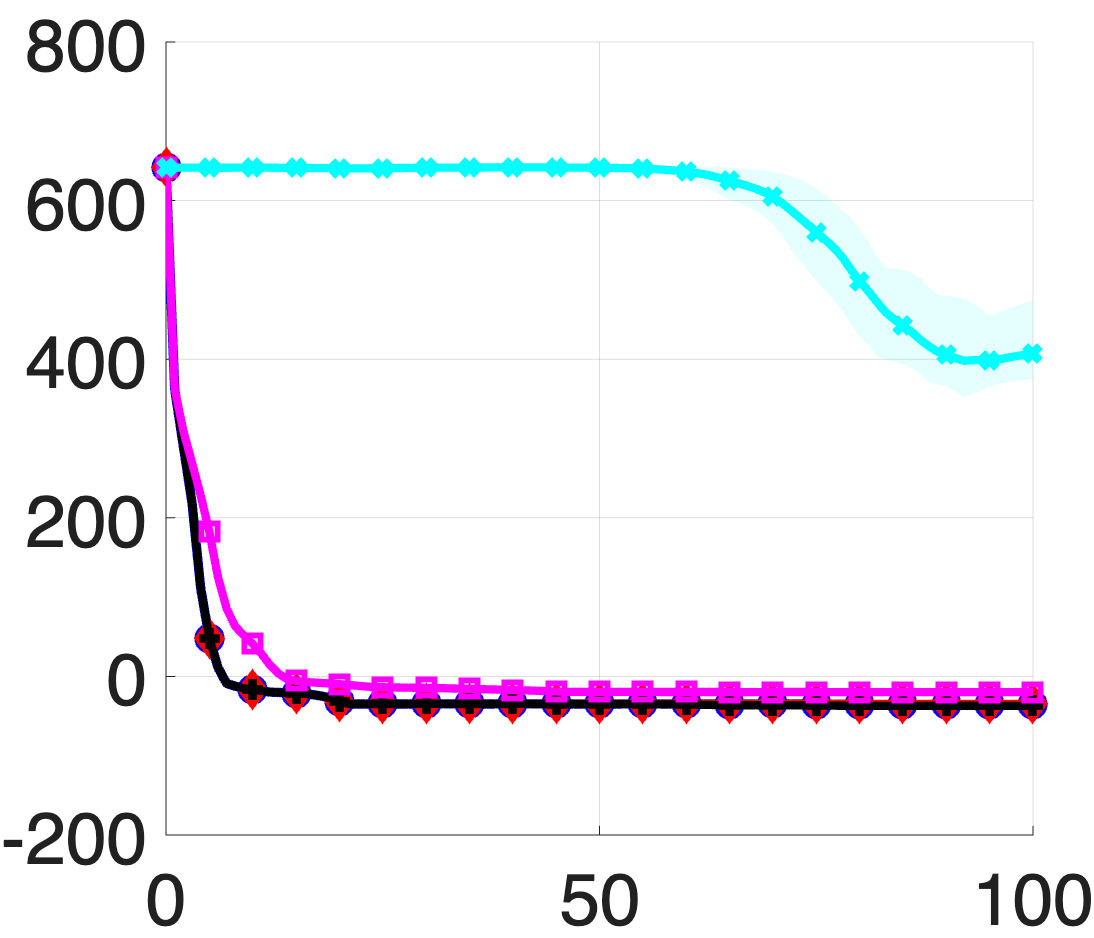}&
    \includegraphics[width=0.23\textwidth]{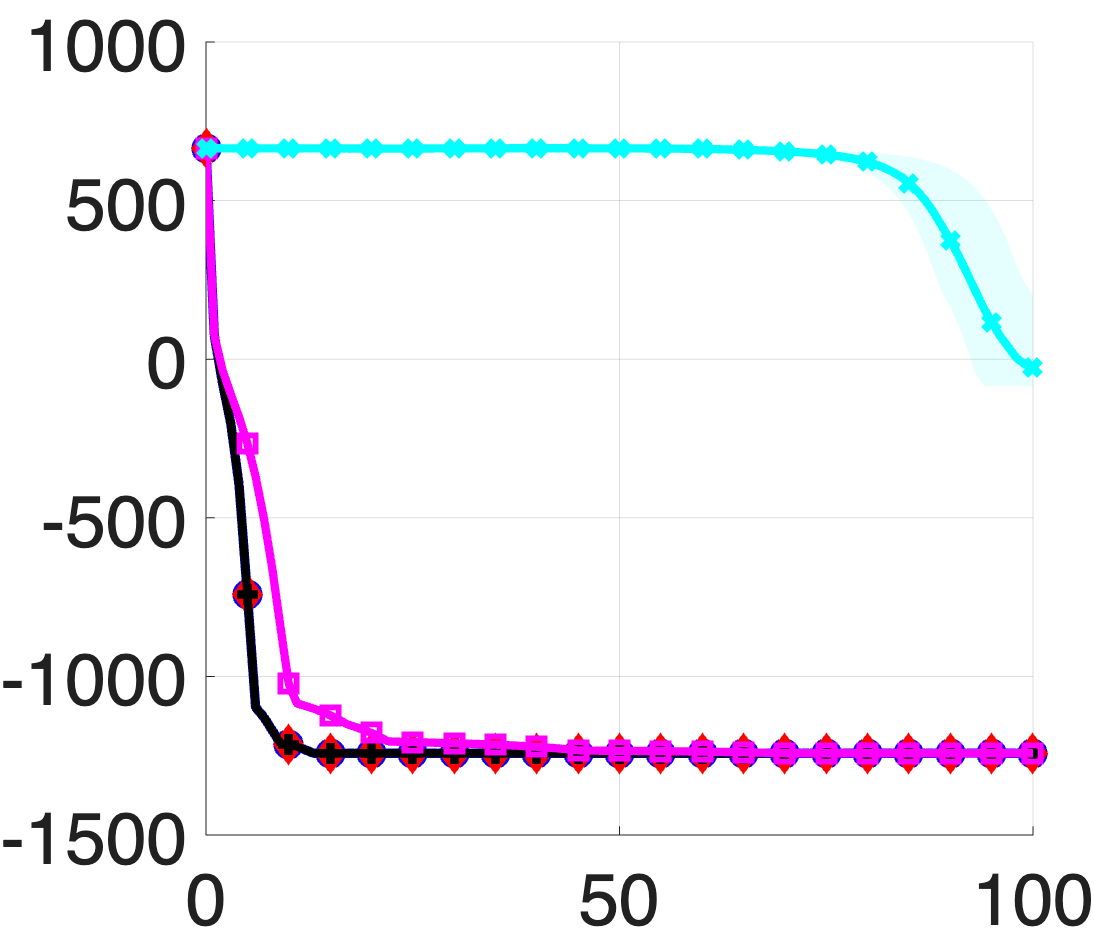}\\
    \includegraphics[width=0.23\textwidth]{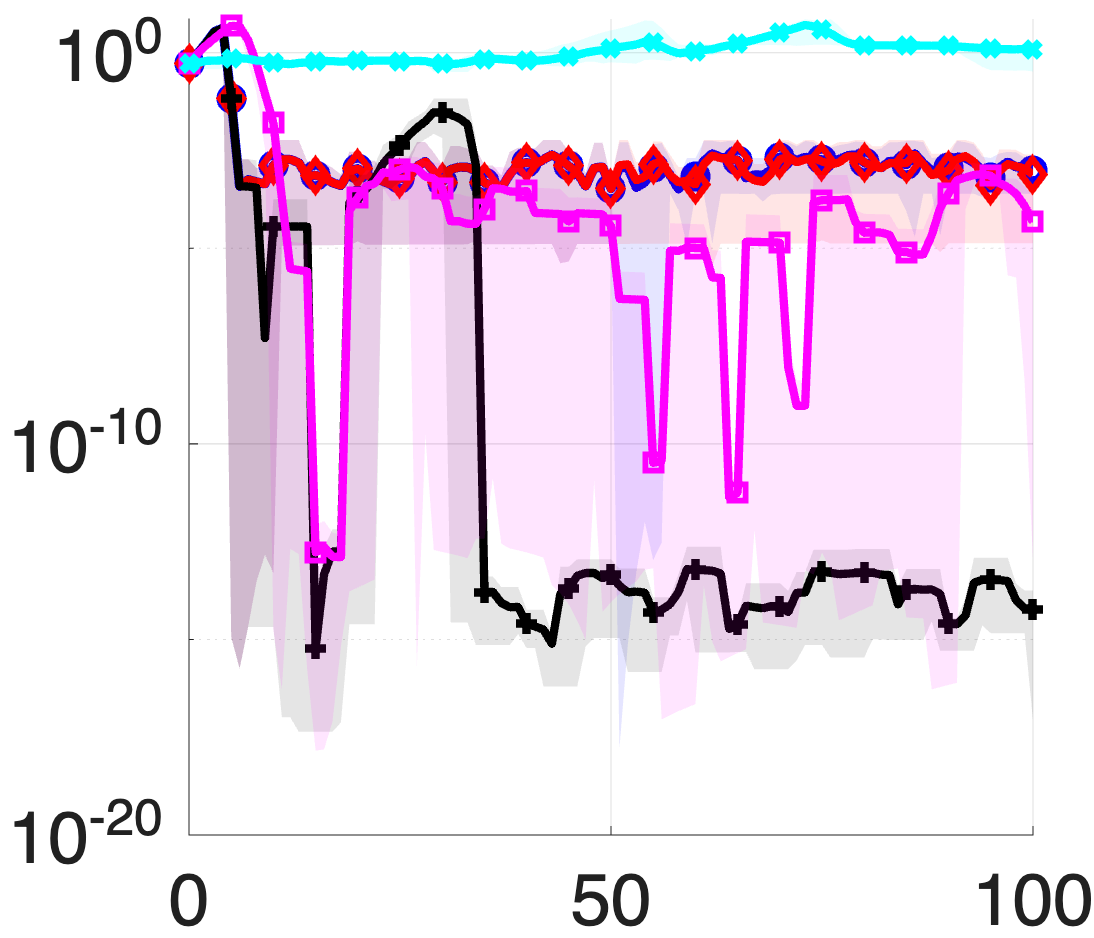}&
    \includegraphics[width=0.23\textwidth]{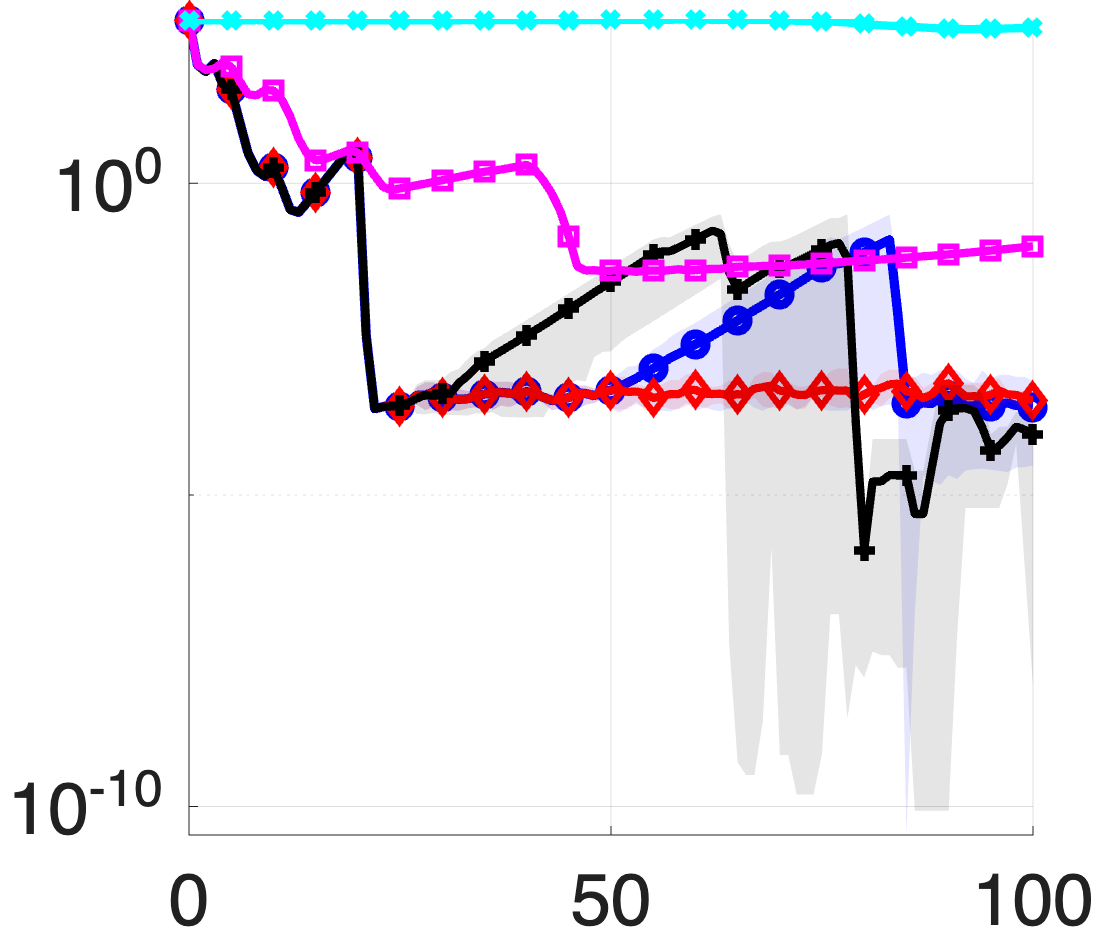}&
    \includegraphics[width=0.23\textwidth]{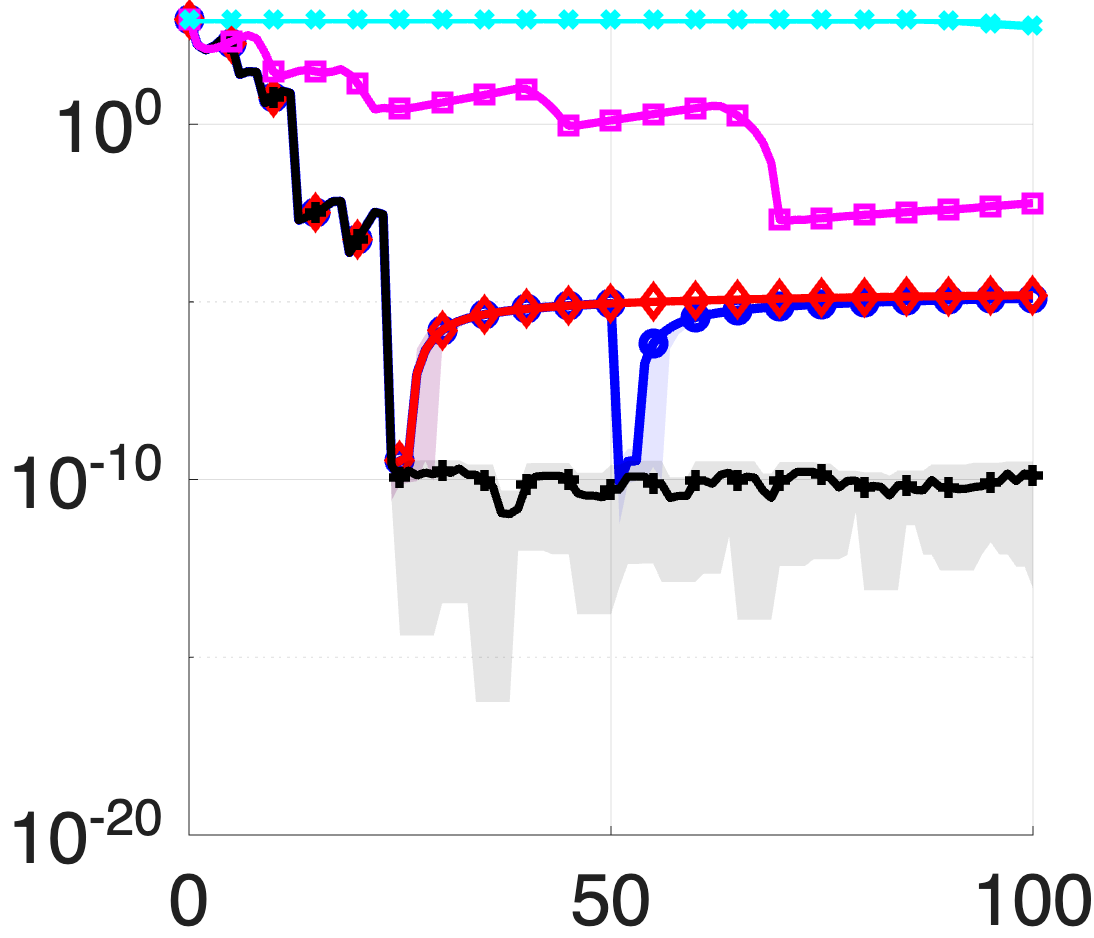}\\
    \includegraphics[width=0.23\textwidth]{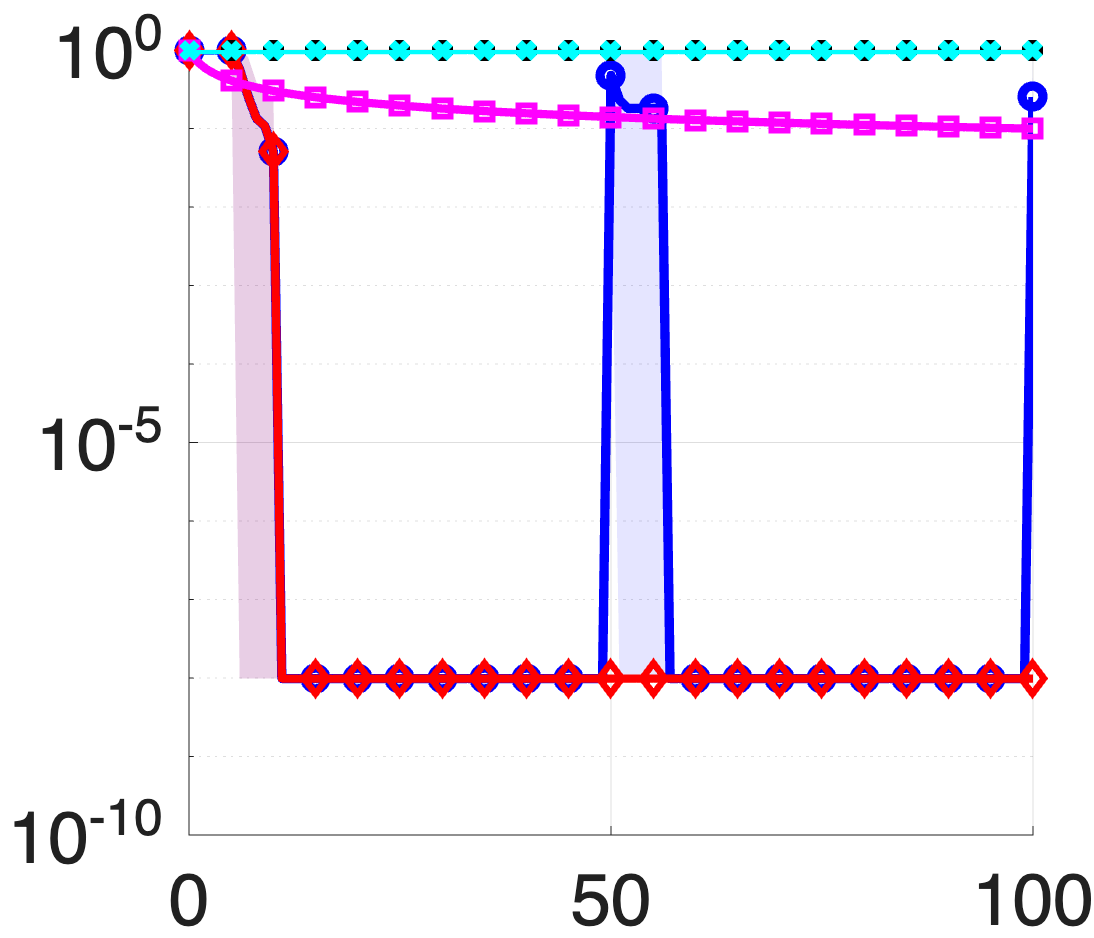}&
    \includegraphics[width=0.23\textwidth]{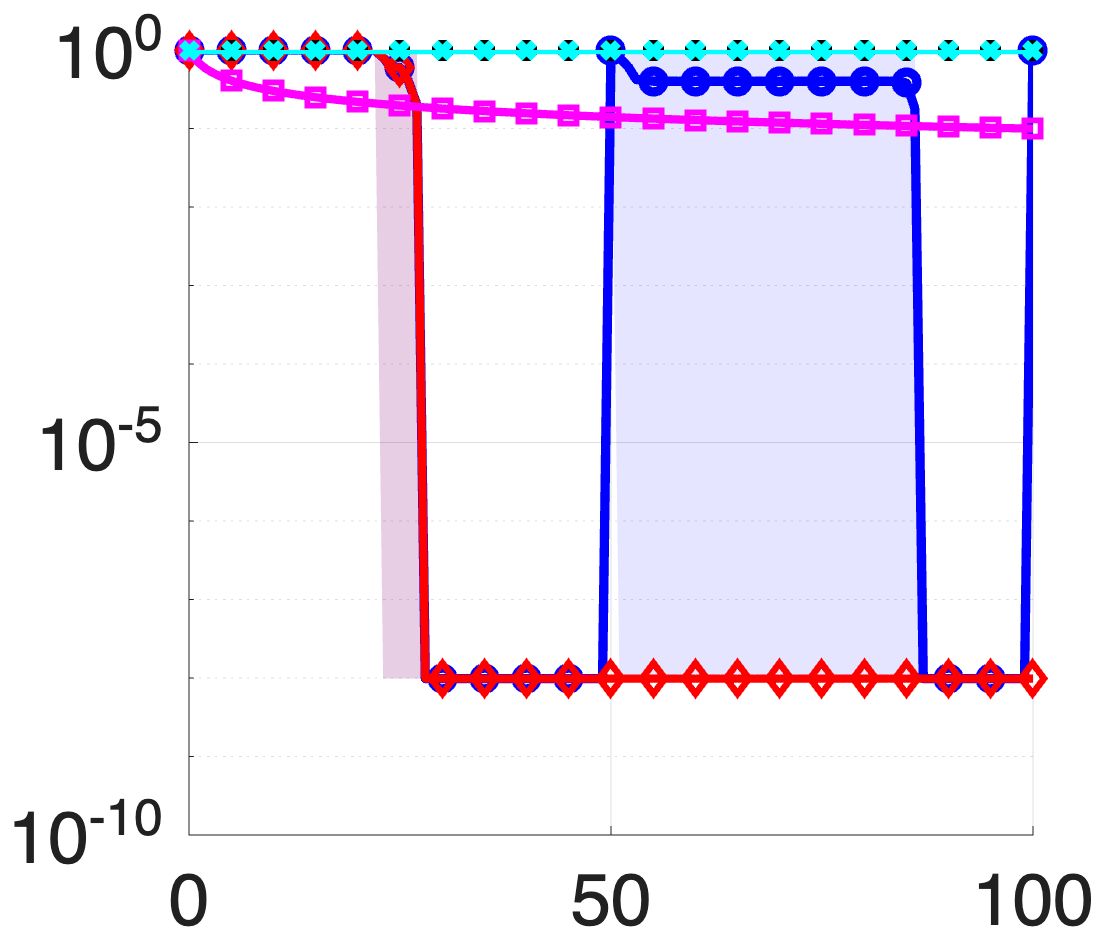}&
    \includegraphics[width=0.23\textwidth]{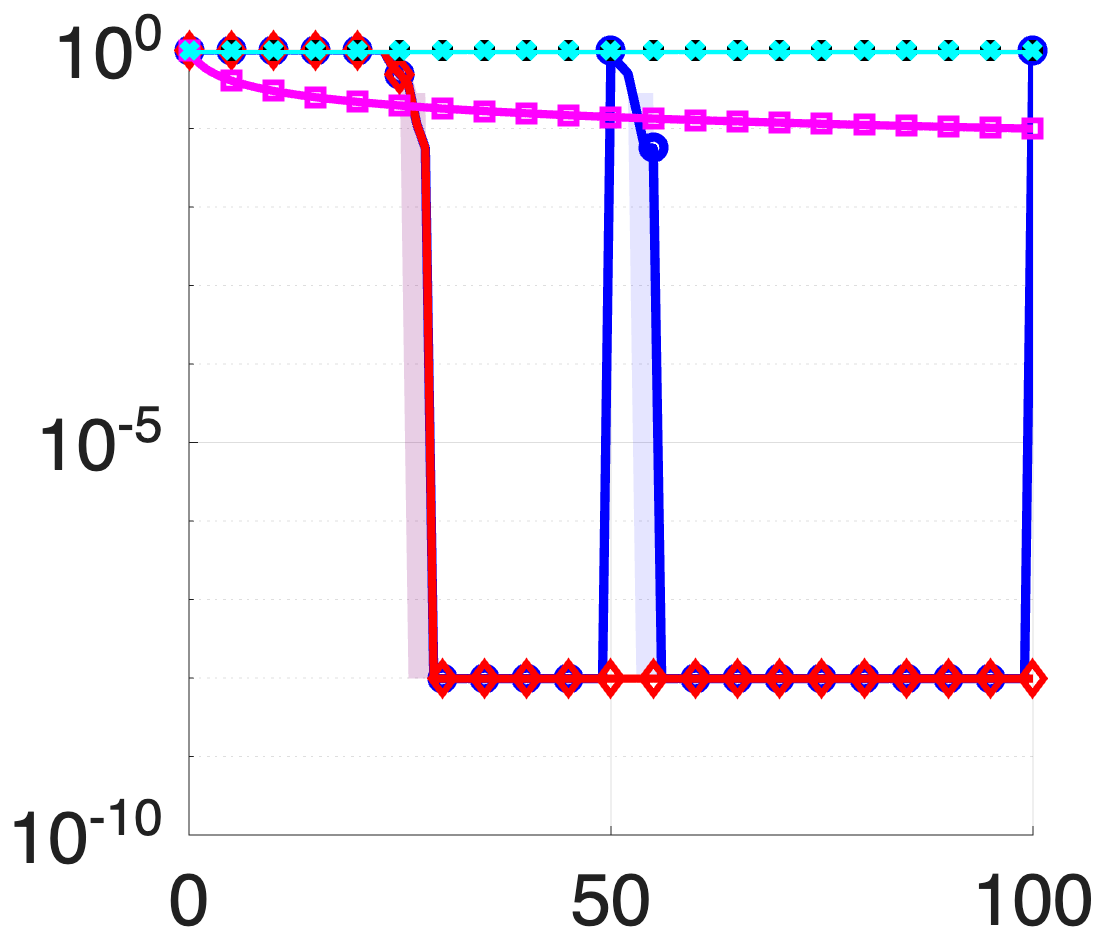}
    \end{tabular}
    \includegraphics[width=0.8\textwidth]{plot_legend_r1.png}
    \caption{Numerical results of two-stage stochastic dispatch problem for $n \in \{10, 50, 100\}$. The top, middle and bottom rows are each for the \emph{estimated} true $\phi(x_k)$, $\|G_s(x_k)\|^2$ {and $t_k$} (average over 5 runs) versus iteration $k$. The scheme \texttt{sls2} is not included in the comparison since it is implemented for finite sample space problems. 
    }
    \label{fig:tsp}
\end{figure}

\subsection{Robustness of SLAM to the period length} \label{sec:5.6}
Previous experiments with our proposed scheme are all conducted with a fixed
period parameter $p_j = 50$ for all $j$. In this subsection, we investigate
the performance of SLAM with different choices of period or cycle length on
the \texttt{duke} problem. Observe that this should not be viewed as tuning of the period  sequence for the previous experiments; instead, this effort is merely to show that performance is relatively robust to choice of period sequence while convergence theory holds for the the choice of period sequences as specified in the convergence guaranatees.  Notice that in the stepsize trajectory of SLAM of
\texttt{duke} problem shown in Figure~\ref{fig:ml}, the stepsizes are small
at the earlier stage and increase to $1$ quickly. Hence, the length of
period may affect the speed at which the stepsize increases. We test the
constant period lengths in $\{1, 10, 50, 100, 1500\}$ and the result is
shown in Figure~\ref{fig:duke_tunep}. Note that \texttt{slam\_1\_p1500} is
equivalent to \texttt{sls0\_1} scheme and represents a setting where the
steplength is never reinitialized to $s = 1$ and the steplength
sequence is non-increasing for the entire duration of the scheme.
\begin{figure}[!tbp]
    \centering
   \includegraphics[width=0.23\textwidth]{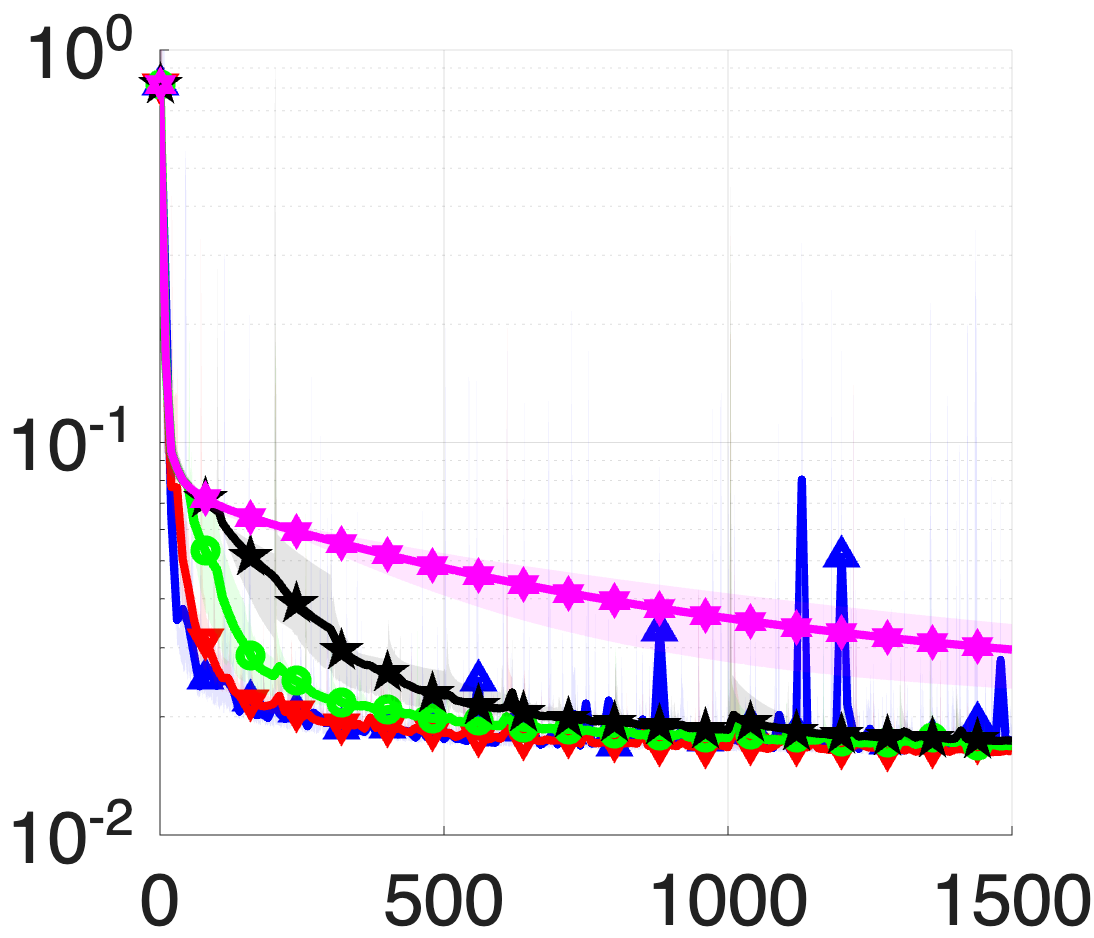}
   \includegraphics[width=0.23\textwidth]{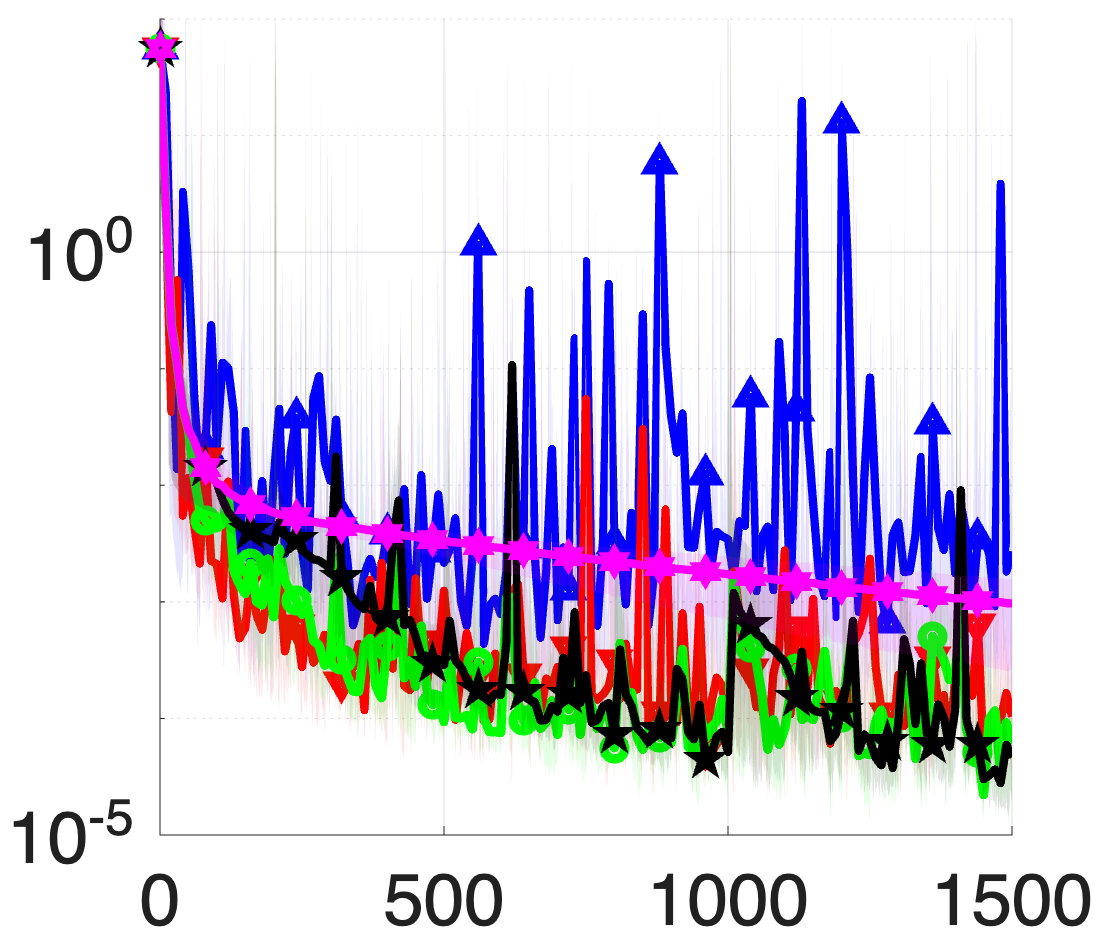}
   \includegraphics[width=0.23\textwidth]{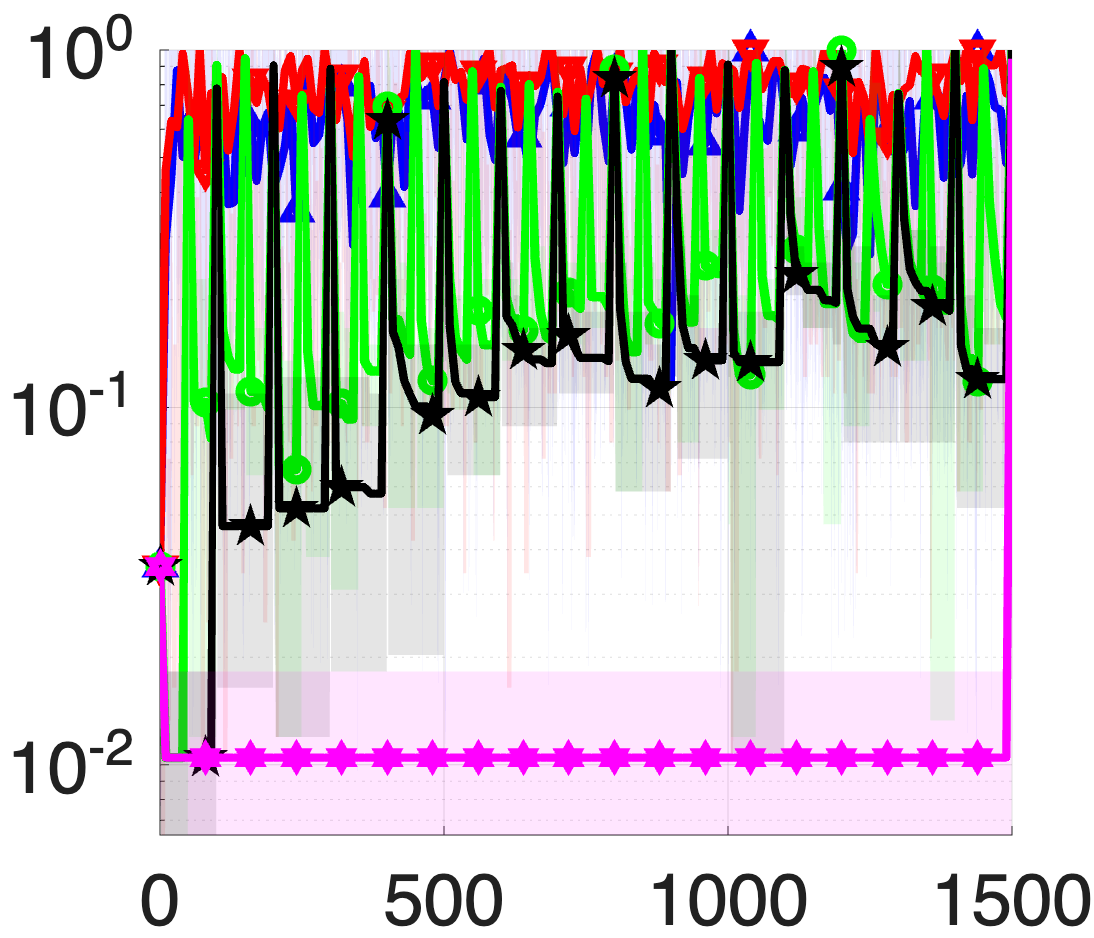}
   \includegraphics[width=0.23\textwidth]{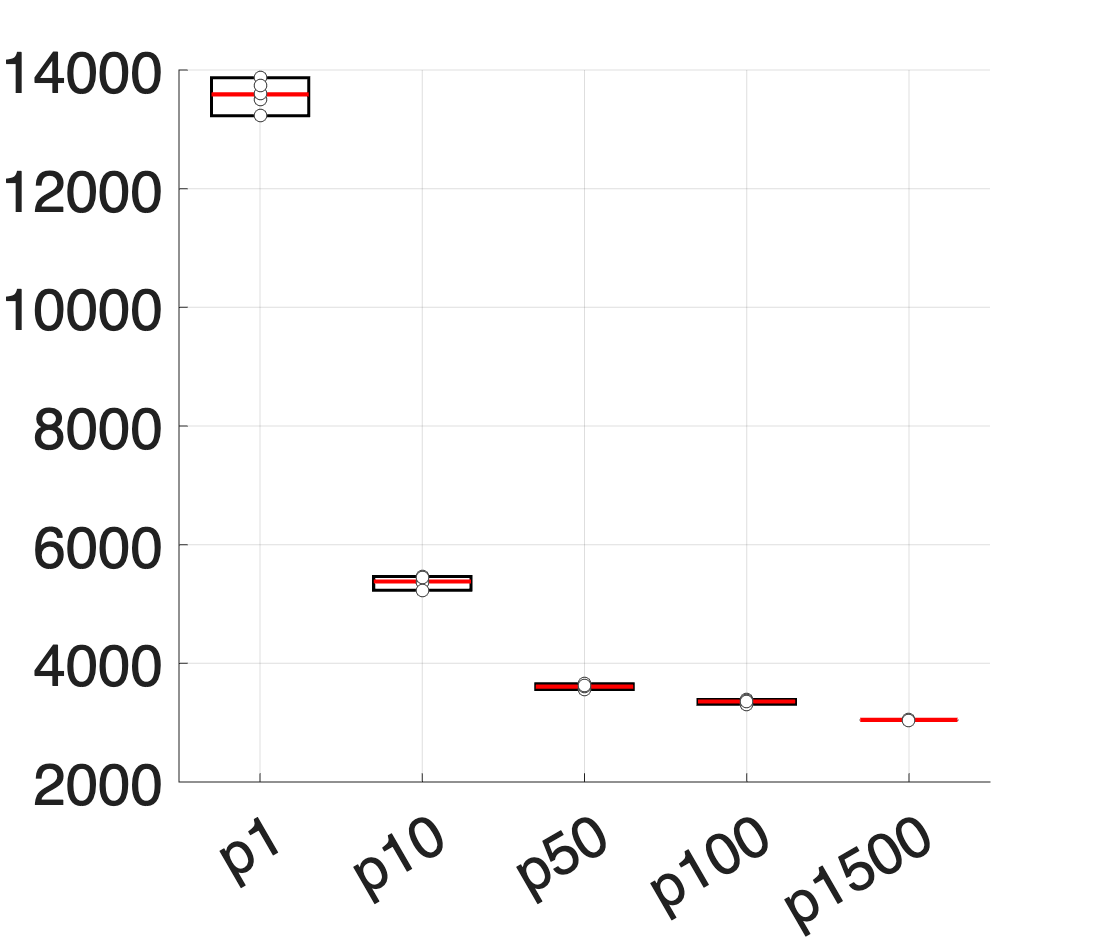}\\
   \vspace{1em}
   \includegraphics[width=0.8\textwidth]{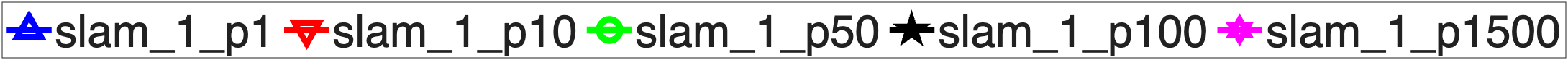}
\caption{Performance of SLAM with different constant period lengths on the \texttt{duke} problem. The first three plots show $\phi(x_k)$, $\|G_s(x_k)\|^2$, and $t_k$, averaged over five runs. The last reports the mean and range of the total number of zeroth-order oracle evaluations over five runs.} \label{fig:duke_tunep}
\end{figure}
One finds that the smaller the period length, the faster the stepsize increases to $1$ and the better the performance of \texttt{SLAM} in this problem. However, smaller period lengths may lead to more zeroth-order oracle evaluations in the linesearch loop and thus higher per-iteration cost, as shown in the rightmost plot of Figure~\ref{fig:duke_tunep}. Experiments show that a smaller period length such as $50$ generally works well in our tests. What  is apparent from the tests above is that for all periods, except $p = 1500$, the performance in terms of function value and residual norm is quite similar, implying that the scheme is relatively robust to changes in $p$.

\section{Conclusion} \label{sec:6}
The success of nonlinear programming solvers such as \texttt{snopt, knitro,} and \texttt{ipopt} hinges, in part, on the ability of such schemes to adapt to problem instances  without necessitating the need for providing problem parameters (such as Lipschitz constants of gradient maps, strong convexity or PL constants, etc.). In comparison, rigorous implementation of extant stochastic gradient-type schemes that adhere to theoretical guidelines, necessitates the knowledge of a host of problem-specific parameters. This often requires a computationally expensive tuning phase, where steplength choices are determined a priori, severely limiting the effectiveness and applicability of such techniques. Moreover, the resulting schemes employ either constant or non-increasing steplength sequences and often require stringent requirements which tend to rarely hold for general stochastic optimization problems. This motivates the need for developing stochastic gradient schemes under relatively mild conditions and free from any problem parameters. In this paper, we propose a novel stochastic linesearch algorithm ({\bf SLAM}) for solving composite stochastic optimization problems, comprised of a smooth expected-value and potentially nonconvex function and a simple nonsmooth convex function. Assuming the availability of CRN oracles and employing an Armijo linesearch on the sampled evaluations, our method adjusts the stepsize based on the local landscape of the objective function and incorporates a periodic reset of the initial steplength. We established convergence rates and complexity guarantees under standard smoothness assumptions and these claims are further specialized to convex and PL regimes. We emphasize that our algorithm parameters are free of any problem-dependent constants. In addition, we do not impose stringent bounds on the initial steplength of the form of $\mathcal{O}(1/L)$, do not require non-increasing steplengths sequences, and do not need to appeal to highly limiting requirements such as the strong growth or the interpolation condition. Our preliminary numerics conducted on a breadth of problems, suggest that the scheme adapts well to a broad class of problem types and instances and competes well with de facto schemes such as \texttt{SGD} and \texttt{Adam}. Future work will focus on extending such a focus to other globalization strategies such as trust-regions and widening the reach of such adaptive methods.
\medskip

\noindent{\bf Acknowledgements} The research of Qi Wang and Uday V. Shanbhag are supported by AFOSR Grant FA9550-24-1-0259 and DOE Grant DE-SC0023303. The research of Yue Xie is funded by Hong Kong RGC General Research Fund (Project number:  17300824),  and Guangdong Province Fundamental and Applied Fundamental Research Regional Joint Fund (Project number: 2022B1515130009).

\bibliographystyle{informs2014}
\bibliography{ref,references}

@article{iusem2019variance,
  title={Variance-based extragradient methods with line search for stochastic variational inequalities},
  author={Iusem, Alfredo N and Jofr{\'e}, Alejandro and Oliveira, Roberto I and Thompson, Philip},
  journal={SIAM Journal on Optimization},
  volume={29},
  number={1},
  pages={175--206},
  year={2019},
  publisher={SIAM}
}

@article{Beck2009,
  title = {A {{Fast Iterative Shrinkage-Thresholding Algorithm}} for {{Linear Inverse Problems}}},
  author = {Beck, Amir and Teboulle, Marc},
  year = 2009,
  journal = {SIAM Journal on Imaging Sciences},
  volume = {2},
  number = {1},
  pages = {183--202},
  issn = {1936-4954},
  doi = {10.1137/080716542}
}

@book{Beck2014,
  title = {Introduction to Nonlinear Optimization: Theory, Algorithms, and Applications with {{MATLAB}}},
  shorttitle = {Introduction to Nonlinear Optimization},
  author = {Beck, Amir},
  year = 2014,
  series = {{{MOS-SIAM}} Series on Optimization},
  address = {Philadelphia},
  isbn = {978-1-61197-364-8},
  lccn = {QA402.5 .B4224 2014}
}

@book{Beck2017,
  title = {First-Order Methods in Optimization},
  author = {Beck, Amir},
  year = 2017,
  series = {{{MOS-SIAM}} Series on Optimization},
  number = {25},
  address = {Philadelphia},
  doi = {10.1137/1.9781611974997},
  isbn = {978-1-61197-498-0 978-1-61197-499-7}
}

@article{Berahas2021,
  title = {Global {{Convergence Rate Analysis}} of a {{Generic Line Search Algorithm}} with {{Noise}}},
  author = {Berahas, A. S. and Cao, L. and Scheinberg, K.},
  year = 2021,
  journal = {SIAM Journal on Optimization},
  volume = {31},
  number = {2},
  pages = {1489--1518},
  issn = {1052-6234, 1095-7189},
  doi = {10.1137/19M1291832}
}

@article{Bertsekas1976,
  title = {On the {{Goldstein-Levitin-Polyak}} Gradient Projection Method},
  author = {Bertsekas, Dimitri P.},
  year = 1976,
  journal = {IEEE Transactions on Automatic Control},
  volume = {21},
  number = {2},
  pages = {174--184},
  issn = {1558-2523},
  doi = {10.1109/TAC.1976.1101194}
}

@article{Bertsekas2000,
  title = {Gradient {{Convergence}} in {{Gradient}} Methods with {{Errors}}},
  author = {Bertsekas, Dimitri P. and Tsitsiklis, John N.},
  year = 2000,
  journal = {SIAM Journal on Optimization},
  volume = {10},
  number = {3},
  pages = {627--642},
  issn = {1052-6234, 1095-7189},
  doi = {10.1137/S1052623497331063}
}

@article{Blanchet2019,
  title = {Convergence {{Rate Analysis}} of a {{Stochastic Trust-Region Method}} via {{Supermartingales}}},
  author = {Blanchet, Jose and Cartis, Coralia and Menickelly, Matt and Scheinberg, Katya},
  year = 2019,
  journal = {INFORMS Journal on Optimization},
  volume = {1},
  number = {2},
  pages = {92--119},
  issn = {2575-1484, 2575-1492},
  doi = {10.1287/ijoo.2019.0016}
}

@article{Bottou2018,
  title = {Optimization {{Methods}} for {{Large-Scale Machine Learning}}},
  author = {Bottou, L{\'e}on and Curtis, Frank E. and Nocedal, Jorge},
  year = 2018,
  journal = {SIAM Review},
  volume = {60},
  number = {2},
  pages = {223--311},
  issn = {0036-1445, 1095-7200},
  doi = {10.1137/16M1080173}
}

@incollection{Byrd2006,
  title = {Knitro: {{An Integrated Package}} for {{Nonlinear Optimization}}},
  shorttitle = {Knitro},
  booktitle = {Large-{{Scale Nonlinear Optimization}}},
  author = {Byrd, Richard H. and Nocedal, Jorge and Waltz, Richard A.},
  editor = {Pardalos, Panos and Di Pillo, G. and Roma, M.},
  year = 2006,
  volume = {83},
  pages = {35--59},
  address = {Boston, MA},
  doi = {10.1007/0-387-30065-1_4},
  isbn = {978-0-387-30063-4 978-0-387-30065-8}
}

@article{Cartis2018,
  title = {Global Convergence Rate Analysis of Unconstrained Optimization Methods Based on Probabilistic Models},
  author = {Cartis, C. and Scheinberg, K.},
  year = 2018,
  journal = {Mathematical Programming},
  volume = {169},
  number = {2},
  pages = {337--375},
  issn = {0025-5610, 1436-4646},
  doi = {10.1007/s10107-017-1137-4}
}

@book{Cartis2022,
  title = {Evaluation Complexity of Algorithms for Nonconvex Optimization: Theory, Computation, and Perspectives},
  shorttitle = {Evaluation Complexity of Algorithms for Nonconvex Optimization},
  author = {Cartis, Coralia and Gould, Nicholas I. M. and Toint, Philippe L.},
  year = 2022,
  series = {{{MOS-SIAM}} Series on Optimization},
  number = {30},
  address = {Philadelphia},
  doi = {10.1137/1.9781611976991},
  isbn = {978-1-61197-698-4 978-1-61197-699-1}
}

@article{Chang2011,
  title = {{{LIBSVM}} : A Library for Support Vector Machines},
  author = {Chang, Chih-Chung and Lin, Chih-Jen},
  year = 2011,
  journal = {ACM transactions on intelligent systems and technology (TIST)},
  volume = {2},
  number = {3},
  pages = {1--27}
}

@article{Chen2018,
  title = {Stochastic Optimization Using a Trust-Region Method and Random Models},
  author = {Chen, R. and Menickelly, M. and Scheinberg, K.},
  year = 2018,
  journal = {Mathematical Programming},
  volume = {169},
  number = {2},
  pages = {447--487},
  issn = {0025-5610, 1436-4646},
  doi = {10.1007/s10107-017-1141-8}
}

@article{Curtis2019,
  title = {A {{Stochastic Trust Region Algorithm Based}} on {{Careful Step Normalization}}},
  author = {Curtis, Frank E. and Scheinberg, Katya and Shi, Rui},
  year = 2019,
  journal = {INFORMS Journal on Optimization},
  volume = {1},
  number = {3},
  pages = {200--220},
  issn = {2575-1484, 2575-1492},
  doi = {10.1287/ijoo.2018.0010}
}

@article{Curtis2022,
  title = {A Fully Stochastic Second-Order Trust Region Method},
  author = {Curtis, Frank E. and Shi, Rui},
  year = 2022,
  journal = {Optimization Methods and Software},
  volume = {37},
  number = {3},
  pages = {844--877},
  issn = {1055-6788},
  doi = {10.1080/10556788.2020.1852403}
}

@inproceedings{Fan2023,
  title = {{{BiSLS}}/{{SPS}}: {{Auto-tune Step Sizes}} for {{Stable Bi-level Optimization}}},
  shorttitle = {{{BiSLS}}/{{SPS}}},
  booktitle = {Advances in {{Neural Information Processing Systems}}},
  author = {Fan, Chen and {Chon{\'e}-Ducasse}, Gaspard and Schmidt, Mark and Thrampoulidis, Christos},
  year = 2023,
  eprint = {2305.18666},
  primaryclass = {cs},
  doi = {10.48550/arXiv.2305.18666},
  archiveprefix = {arXiv}
}

@inproceedings{Galli2023,
  title = {Don't Be so {{Monotone}}: {{Relaxing Stochastic Line Search}} in {{Over-Parameterized Models}}},
  shorttitle = {Don't Be so {{Monotone}}},
  booktitle = {Advances in {{Neural Information Processing Systems}}},
  author = {Galli, Leonardo and Rauhut, Holger and Schmidt, Mark},
  year = 2023,
  eprint = {2306.12747},
  primaryclass = {math},
  doi = {10.48550/arXiv.2306.12747},
  archiveprefix = {arXiv}
}

@article{Ghadimi2013,
  title = {Stochastic {{First-}} and {{Zeroth-Order Methods}} for {{Nonconvex Stochastic Programming}}},
  author = {Ghadimi, Saeed and Lan, Guanghui},
  year = 2013,
  journal = {SIAM Journal on Optimization},
  volume = {23},
  number = {4},
  pages = {2341--2368},
  issn = {1052-6234, 1095-7189},
  doi = {10.1137/120880811}
}

@article{Ghadimi2016,
  title = {Mini-Batch Stochastic Approximation Methods for Nonconvex Stochastic Composite Optimization},
  author = {Ghadimi, Saeed and Lan, Guanghui and Zhang, Hongchao},
  year = 2016,
  journal = {Mathematical Programming},
  volume = {155},
  number = {1-2},
  pages = {267--305},
  issn = {0025-5610, 1436-4646},
  doi = {10.1007/s10107-014-0846-1}
}

@article{Gill2005,
  title = {{{SNOPT}}: {{An SQP Algorithm}} for {{Large-Scale Constrained Optimization}}},
  shorttitle = {{{SNOPT}}},
  author = {Gill, Philip E. and Murray, Walter and Saunders, Michael A.},
  year = 2005,
  journal = {SIAM Review},
  volume = {47},
  number = {1},
  pages = {99--131},
  issn = {0036-1445, 1095-7200},
  doi = {10.1137/S0036144504446096}
}

@article{Ha2025,
  title={Complexity of zeroth-and first-order stochastic trust-region algorithms},
  author={Ha, Yunsoo and Shashaani, Sara and Pasupathy, Raghu},
  journal={SIAM Journal on Optimization},
  volume={35},
  number={3},
  pages={2098--2127},
  year={2025},
  publisher={SIAM}
}

@inproceedings{He2016,
  title = {Deep {{Residual Learning}} for {{Image Recognition}}},
  booktitle = {2016 {{IEEE Conference}} on {{Computer Vision}} and {{Pattern Recognition}} ({{CVPR}})},
  author = {He, Kaiming and Zhang, Xiangyu and Ren, Shaoqing and Sun, Jian},
  year = 2016,
  pages = {770--778},
  address = {Las Vegas, NV, USA},
  doi = {10.1109/CVPR.2016.90},
  isbn = {978-1-4673-8851-1}
}

@article{Jalilzadeh2022,
  title = {Smoothed {{Variable Sample-Size Accelerated Proximal Methods}} for {{Nonsmooth Stochastic Convex Programs}}},
  author = {Jalilzadeh, Afrooz and Shanbhag, Uday V. and Blanchet, Jose and Glynn, Peter W.},
  year = 2022,
  journal = {Stochastic Systems},
  volume = {12},
  number = {4},
  pages = {373--410},
  issn = {1946-5238, 1946-5238},
  doi = {10.1287/stsy.2022.0095}
}

@inproceedings{Jiang2023,
  title = {Adaptive {{SGD}} with {{Polyak}} Stepsize and {{Line-search}}: {{Robust Convergence}} and {{Variance Reduction}}},
  shorttitle = {Adaptive {{SGD}} with {{Polyak}} Stepsize and {{Line-search}}},
  booktitle = {Advances in {{Neural Information Processing Systems}}},
  author = {Jiang, Xiaowen and Stich, Sebastian U.},
  year = 2023,
  eprint = {2308.06058},
  primaryclass = {cs, math, stat},
  archiveprefix = {arXiv}
}

@inproceedings{Jin2021,
  title = {High {{Probability Complexity Bounds}} for {{Line Search Based}} on {{Stochastic Oracles}}},
  booktitle = {Advances in {{Neural Information Processing Systems}}},
  author = {Jin, Billy and Scheinberg, Katya and Xie, Miaolan},
  year = 2021,
  volume = {34}
}

@article{Jin2024,
  title = {High {{Probability Complexity Bounds}} for {{Adaptive Step Search Based}} on {{Stochastic Oracles}}},
  author = {Jin, Billy and Scheinberg, Katya and Xie, Miaolan},
  year = 2024,
  journal = {SIAM Journal on Optimization},
  volume = {34},
  number = {3},
  pages = {2411--2439},
  issn = {1052-6234, 1095-7189},
  doi = {10.1137/22M1512764}
}

@article{Jin2025,
  title = {Sample Complexity Analysis for Adaptive Optimization Algorithms with Stochastic Oracles},
  author = {Jin, Billy and Scheinberg, Katya and Xie, Miaolan},
  year = 2025,
  journal = {Mathematical Programming},
  volume = {209},
  number = {1-2},
  pages = {651--679},
  issn = {0025-5610, 1436-4646},
  doi = {10.1007/s10107-024-02078-z}
}

@incollection{Karimi2020,
  title = {Linear {{Convergence}} of {{Gradient}} and {{Proximal-Gradient Methods Under}} the {{Polyak-\L ojasiewicz Condition}}},
  booktitle = {Joint {{European}} Conference on Machine Learning and Knowledge Discovery in Databases},
  author = {Karimi, Hamed and Nutini, Julie and Schmidt, Mark},
  year = 2020,
  eprint = {1608.04636},
  primaryclass = {cs},
  pages = {795--811},
  archiveprefix = {arXiv}
}

@inproceedings{Kingma2017a,
  author       = {Diederik P. Kingma and
                  Jimmy Ba},
  editor       = {Yoshua Bengio and
                  Yann LeCun},
  title        = {Adam: {A} Method for Stochastic Optimization},
  booktitle    = {3rd International Conference on Learning Representations, {ICLR} 2015,
                  San Diego, CA, USA, May 7-9, 2015, Conference Track Proceedings},
  year         = {2015},
  url          = {http://arxiv.org/abs/1412.6980},
  bibsource    = {dblp computer science bibliography, https://dblp.org}
}

@article{Krizhevsky2009,
  title = {Learning {{Multiple Layers}} of {{Features}} from {{Tiny Images}}},
  author = {Krizhevsky, Alex},
  year = 2009
}

@misc{Lan2024,
  title = {Projected Gradient Methods for Nonconvex and Stochastic Optimization: New Complexities and Auto-Conditioned Stepsizes},
  shorttitle = {Projected Gradient Methods for Nonconvex and Stochastic Optimization},
  author = {Lan, Guanghui and Li, Tianjiao and Xu, Yangyang},
  year = 2024,
  number = {arXiv:2412.14291},
  eprint = {2412.14291},
  primaryclass = {math},
  doi = {10.48550/arXiv.2412.14291},
  archiveprefix = {arXiv}
}

@misc{Lei2020a,
  title = {Asynchronous {{Variance-reduced Block Schemes}} for {{Composite Nonconvex Stochastic Optimization}}: {{Block-specific Steplengths}} and {{Adapted Batch-sizes}}},
  shorttitle = {Asynchronous {{Variance-reduced Block Schemes}} for {{Composite Nonconvex Stochastic Optimization}}},
  author = {Lei, Jinlong and Shanbhag, Uday V.},
  year = 2020,
  number = {arXiv:1808.02543},
  eprint = {1808.02543},
  primaryclass = {math},
  doi = {10.48550/arXiv.1808.02543},
  archiveprefix = {arXiv}
}

@inproceedings{Li2018,
  title = {A {{Simple Proximal Stochastic Gradient Method}} for {{Nonsmooth Nonconvex Optimization}}},
  booktitle = {Advances in {{Neural Information Processing Systems}}},
  author = {Li, Zhize and Li, Jian},
  year = 2018
}

@book{Nocedal2006,
  title = {Numerical Optimization},
  author = {Nocedal, Jorge and Wright, Stephen J.},
  year = 2006,
  series = {Springer Series in Operations Research},
  edition = {2nd ed},
  address = {New York},
  isbn = {978-0-387-30303-1},
  lccn = {QA402.5 .N62 2006}
}

@article{Paquette2020,
  title = {A {{Stochastic Line Search Method}} with {{Expected Complexity Analysis}}},
  author = {Paquette, Courtney and Scheinberg, Katya},
  year = 2020,
  journal = {SIAM Journal on Optimization},
  volume = {30},
  number = {1},
  pages = {349--376},
  issn = {1052-6234, 1095-7189},
  doi = {10.1137/18M1216250}
}

@incollection{Robbins1985,
  title = {A {{Convergence Theorem}} for {{Non Negative Almost Supermartingales}} and {{Some Applications}}},
  booktitle = {Herbert {{Robbins Selected Papers}}},
  author = {Robbins, H. and Siegmund, D.},
  editor = {Lai, T. L. and Siegmund, D.},
  year = 1985,
  pages = {111--135},
  address = {New York, NY},
  doi = {10.1007/978-1-4612-5110-1_10},
  isbn = {978-1-4612-9568-6 978-1-4612-5110-1}
}

@article{Salzo2017,
  title = {The {{Variable Metric Forward-Backward Splitting Algorithm Under Mild Differentiability Assumptions}}},
  author = {Salzo, Saverio},
  year = 2017,
  journal = {SIAM Journal on Optimization},
  volume = {27},
  number = {4},
  pages = {2153--2181},
  issn = {1052-6234, 1095-7189},
  doi = {10.1137/16M1073741}
}

@book{Shapiro2009,
  title = {Lectures on {{Stochastic Programming}}: {{Modeling}} and {{Theory}}},
  shorttitle = {Lectures on {{Stochastic Programming}}},
  author = {Shapiro, Alexander and Dentcheva, Darinka and Ruszczy{\'n}ski, Andrzej},
  year = 2009,
  doi = {10.1137/1.9780898718751},
  isbn = {978-0-89871-687-0 978-0-89871-875-1}
}

@article{Shashaani2018,
  title = {{{ASTRO-DF}}: {{A Class}} of {{Adaptive Sampling Trust-Region Algorithms}} for {{Derivative-Free Stochastic Optimization}}},
  shorttitle = {{{ASTRO-DF}}},
  author = {Shashaani, Sara and Hashemi, Fatemeh S. and Pasupathy, Raghu},
  year = 2018,
  journal = {SIAM Journal on Optimization},
  volume = {28},
  number = {4},
  pages = {3145--3176},
  issn = {1052-6234, 1095-7189},
  doi = {10.1137/15M1042425}
}

@article{Tsukada2024,
  author       = {Yuki Tsukada and
                  Hideaki Iiduka},
  title        = {Relationship between Batch Size and Number of Steps Needed for Nonconvex
                  Optimization of Stochastic Gradient Descent using Armijo-Line-Search
                  Learning Rate},
  journal      = {Trans. Mach. Learn. Res.},
  volume       = {2025},
  year         = {2025},
  url          = {https://openreview.net/forum?id=pqZ6nOm3WF},
  bibsource    = {dblp computer science bibliography, https://dblp.org}
}

@inproceedings{Vaswani2019,
  title = {Painless {{Stochastic Gradient}}: {{Interpolation}}, {{Line-Search}}, and {{Convergence Rates}}},
  shorttitle = {Painless {{Stochastic Gradient}}},
  booktitle = {Advances in {{Neural Information Processing Systems}}},
  author = {Vaswani, Sharan and Mishkin, Aaron and Laradji, Issam and Schmidt, Mark and Gidel, Gauthier and {Lacoste-Julien}, Simon},
  year = 2019,
  volume = {32}
}

@article{Vaswani2025,
  title = {Armijo {{Line-search Makes}} ({{Stochastic}}) {{Gradient Descent Go Fast}}},
  author = {Vaswani, Sharan and Babanezhad, Reza},
  year = 2025,
  journal = {arXiv:2503.00229},
  eprint = {2503.00229},
  primaryclass = {cs},
  doi = {10.48550/arXiv.2503.00229},
  archiveprefix = {arXiv}
}

@article{Yang2024a,
  title = {Proximal {{Gradient Method}} with {{Extrapolation}} and {{Line Search}} for a {{Class}} of {{Non-convex}} and {{Non-smooth Problems}}},
  author = {Yang, Lei},
  year = 2024,
  journal = {Journal of Optimization Theory and Applications},
  volume = {200},
  number = {1},
  pages = {68--103},
  issn = {0022-3239, 1573-2878},
  doi = {10.1007/s10957-023-02348-4}
}

@misc{dahlBenchmarkingNeuralNetwork2023,
  title = {Benchmarking {{Neural Network Training Algorithms}}},
  author = {Dahl, George E. and Schneider, Frank and Nado, Zachary and Agarwal, Naman and Sastry, Chandramouli Shama and Hennig, Philipp and Medapati, Sourabh and Eschenhagen, Runa and Kasimbeg, Priya and Suo, Daniel and Bae, Juhan and Gilmer, Justin and Peirson, Abel L. and Khan, Bilal and Anil, Rohan and Rabbat, Mike and Krishnan, Shankar and Snider, Daniel and Amid, Ehsan and Chen, Kongtao and Maddison, Chris J. and Vasudev, Rakshith and Badura, Michal and Garg, Ankush and Mattson, Peter},
  year = 2023,
  number = {arXiv:2306.07179},
  eprint = {2306.07179},
  primaryclass = {cs},
  doi = {10.48550/arXiv.2306.07179},
  archiveprefix = {arXiv}
}

@article{li2024problem,
  title={Problem-parameter-free decentralized nonconvex stochastic optimization},
  author={Li, Jiaxiang and Chen, Xuxing and Ma, Shiqian and Hong, Mingyi},
  journal={arXiv preprint arXiv:2402.08821},
  year={2024}
}

@article{zhai2026problem,
  title={Problem-Parameter-Free Decentralized Bilevel Optimization},
  author={Zhai, Zhiwei and Yan, Wenjing and Zhang, Ying-Jun},
  journal={Advances in Neural Information Processing Systems},
  volume={38},
  pages={144986--145053},
  year={2026}
}

@inproceedings{yan2025problem,
  title={Problem-parameter-free federated learning},
  author={Yan, Wenjing and Zhang, Kai and Wang, Xiaolu and Cao, Xuanyu},
  booktitle={International Conference on Learning Representations},
  volume={2025},
  pages={97350--97377},
  year={2025}
}

@article{mcmahan2017survey,
  title={A survey of algorithms and analysis for adaptive online learning},
  author={McMahan, H Brendan},
  journal={Journal of Machine Learning Research},
  volume={18},
  number={90},
  pages={1--50},
  year={2017}
}

@inproceedings{loizou2021stochastic,
  title={Stochastic polyak step-size for sgd: An adaptive learning rate for fast convergence},
  author={Loizou, Nicolas and Vaswani, Sharan and Laradji, Issam Hadj and Lacoste-Julien, Simon},
  booktitle={International conference on artificial intelligence and statistics},
  pages={1306--1314},
  year={2021},
  organization={PMLR}
}

@misc{orabona2020parameterfree,
  author = {Orabona, Francesco and Cutkosky, Ashok},
  title = {Parameter-Free Learning and Optimization},
  howfound = {Tutorial at the 37th International Conference on Machine Learning (ICML)},
  year = {2020}}

@inproceedings{medyakov2026signsgd,
title={Sign-{SGD} via Parameter-Free Optimization},
author={Daniil Medyakov and Stanko Sergey and Gleb Molodtsov and Philip Zmushko and Evseev Grigoriy and Egor Petrov and Aleksandr Beznosikov},
booktitle={The Fourteenth International Conference on Learning Representations},
year={2026},
url={https://openreview.net/forum?id=yDLD3D95w3}
}

@book{nocedal2006numerical,
  title={Numerical optimization},
  author={Nocedal, Jorge and Wright, Stephen J},
  year={2006},
  publisher={Springer}
}

@inproceedings{Carmon2022,
  title = {Making {{SGD Parameter-Free}}},
  booktitle = {Proceedings of {{Thirty Fifth Conference}} on {{Learning Theory}}},
  author = {Carmon, Yair and Hinder, Oliver},
  year = 2022,
  pages = {2360--2389},
  issn = {2640-3498}
}
\newpage
\begin{appendix}
\section{Proofs}\label{app: prfs0}
Proof of Lemma~\ref{lem.Nk sum}.
\begin{proof}
    (i). One has $\sum_{k=0}^{K-1} \frac{1}{N_k} \le \frac{K}{c_b K} = \frac{1}{c_b}$; (ii) We observe that $$\sum_{k=0}^{K-1} \frac{1}{N_k} =  \sum_{k=1}^K \frac{1}{k} \le 1 + \sum_{k=2}^K \int_{k-1}^k \frac{1}{u} du = 1 + \int_1^K \frac{1}{u}du = \ln(K) + 1.$$ 
    (iii) Under this setting, 
    \begin{align*}
        \sum_{k=0}^{K-1} \frac{1}{N_k} = \sum_{k=1}^K k^{-b} \le 1 + \int_{1}^{K} x^{-b} dx = \frac{K^{1-b} - b}{1-b} \le \frac{b}{b-1}.
    \end{align*}
\end{proof}

\noindent Proof of Lemma~\ref{lem.pj sum}.
\begin{proof}
\ref{item.pj1} We have 
\begin{align*}
    J(K) + 1 = \left\lceil \frac{K}{\lceil c_p K \rceil} \right\rceil \le \frac{K}{\lceil c_p K \rceil} + 1 \le \frac{K}{c_p K} + 1 = \frac{1}{c_p} + 1.
\end{align*}
\ref{item.pj2} We have that $J(K)$ is the smallest positive integer $l$ such that $\sum_{j=0}^l 2^j l_0 \ge K$, i.e., 
\begin{align*}
    && l_0(2^{l+1}-1) &\ge K \\
    \iff && l+1 &\ge \log_2\left(\, \frac{l_0 + K}{l_0}\, \right) \\
    \implies && J(K) + 1 & = \left\lceil \log_2\left(\frac{l_0 + K}{l_0}\right) \right\rceil.
\end{align*}
\ref{item.pj3} In this setting, $J(K)$ is the smallest positive integer $l$ such that $\sum_{j=0}^l (l_0 + j) \ge K$, i.e., 
\begin{align*}
    && (l+1)l_0 + \frac{l(l+1)}{2} &\ge K \\
    \iff && \tfrac{1}{2} l^2 + \left(\tfrac{1}{2} + l_0\right)l - (K-l_0) &\ge 0 \\
    \implies && l &\ge \sqrt{(\tfrac{1}{2} + l_0)^2 + 2(K-l_0)} - (\tfrac{1}{2} + l_0)\\
    \implies && J(K)+1 & = \left\lceil \sqrt{(\tfrac{1}{2} + l_0)^2 + 2(K-l_0)} - (\tfrac{1}{2} + l_0) + 1\right\rceil.
\end{align*}
\end{proof}

\section{Additional discussions on the two-stage problem in Section~\ref{sec:5.5}}\label{app: 2st}

 We discuss the uniqueness of the dual solution as well as whether Assumption~\ref{ass: F} and \ref{ass: var} hold for the economic dispatch problem. In general, if the dual solution is unique, the stochastic gradient is given by
    \begin{align*}
        \nabla H(x) + \nabla F(x, \bxi) = \nabla H(x) + T(\bxi)^\top \pi(x, \bxi).
    \end{align*}
   To show the Lipschitz continuity of the stochastic gradient in $x$ for a given $\xi$, it suffices to show the corresponding Lipschitz continuity of the dual solution $\pi(\bullet, \bxi)$. For a given $x$ and a given realization $\xi$, the second stage problem is 
    \begin{align}\label{SP-2nd}\tag{SP$_2(x,\xi)$}
    \begin{aligned}
        \min_{y \in \R^m} \ & \frac{1}{2}y^\top Q(\xi) y \\
        \text{s.t. } & W(\xi) y = h(\xi) - T(\xi) x, \\
        & y \ge 0.
        \end{aligned}
    \end{align}
For a given $x$ and $\xi$, we denote the primal and dual solutions as $y^*$ and $(\pi^*, s^*)$, respectively where we suppress the dependence on $x$ and $\xi$. These solutions satisfy the KKT conditions
    \begin{align}\label{KKT}\tag{KKT$_2(x,\xi)$}
    \begin{aligned}
        Q(\xi) y^* + W(\xi)^\top \pi^* - s^* &= 0, \\
        W(\xi) y^* &= h(\xi) - T(\xi) x, \\
        (y^*)^\top s^* &= 0, \\
        y^* &\ge 0, \\
        s^* &\ge 0.
        \end{aligned}
    \end{align}
By simplification of the KKT system and checking non-singularity, we may claim the uniqueness of the dual solution via the following Lemma, the proof of which is omitted.
\begin{lemma}\label{lm: kkt}\em
    Consider the second-stage problem \eqref{SP-2nd}, given a feasible $x$ and a realization $\xi$. 
    Suppose the LICQ holds at the unique primal solution $y^*(x,\xi)$ of \eqref{SP-2nd}. Then there exists a unique solution $(y^*(x,\xi),\pi^*(x,\xi)),s^*(x,\xi))$ to the KKT conditions, 
    where $(\pi^*(x,\xi),s^*(x,\xi))$ denotes the unique solution of the dual problem. In particular, we have 
    \begin{align}\label{kktsol}
        \begin{pmatrix}
            y^*(x,\xi) \\ \pi^*(x,\xi) \\ s^*(x,\xi) 
        \end{pmatrix}
        = \begin{pmatrix}
             Q(\xi) & W(\xi)^\top & -I \\ W(\xi) & \mathbf{0} & \mathbf{0} \\ E_{\mathcal{A}}(x,\xi) & \mathbf{0} & \mathbf{0} \\ \mathbf{0} & \mathbf{0} & E_{\mathcal{A}^c(x,\xi)}
         \end{pmatrix}^{-1}
        \begin{pmatrix}
            \mathbf{0}_m \\ h(\xi) - T(\xi) x \\ \mathbf{0}_m 
        \end{pmatrix},
    \end{align}
    where $\mathcal{A}(x,\xi) \triangleq \left\{ \, i \, \mid \, y^*_i(x,\bxi) = 0 \, \right\}$ (simplified as $\mathcal{A}$), and
    \begin{align*}
        E_{\mathcal{A}} = [\mathbf{e}_{i_1},\cdots, \mathbf{e}_{i_{|\mathcal{A}|}}]^\top,
    \end{align*}
    with $i_k \in \mathcal{A}$ for all $k$ and for any $i$, $\mathbf{e}_i$ is given by $[0, \cdots, 0,1,0,\cdots,0]$ with $1$ in the $i$th entry.
\end{lemma}
Note that we will routinely suppress the dependence on $x$ and $\xi$, when referring to the primal-dual triple $(y^*,\pi^*,s^*)$. By Lemma~\ref{lm: kkt}, $\nabla F(x,\bxi) = T^\top \pi(x,\bxi)$. Next we argue that $(y^*(x,\xi),\pi^*(x,\xi),s^*(x,\xi))$ is locally affine  in $x$ for a given $\xi$, where $(y^*(x.\xi),\pi^*(x,\xi),s^*(x,\xi))$ denotes the solution of \eqref{SP-2nd} for a given $x$ and some $\xi$.
\begin{lemma}\label{lm: la}
    Consider the two-stage problem \eqref{SP-2nd} in Section~\eqref{sec:5.5} for some feasible $\hat{x}$ and realization $\xi$. Suppose the LICQ holds at $y^*(x,\xi)$ for any $x$ and $\xi$.  Then there exists $\epsilon > 0$ such that if $\bar x$ is feasible and $\| \bar x - \hat x \| \le \epsilon$, then $(y^*,\pi^*,s^*)$ is affine in $x$ on the line segment connecting $\hat x$ and $\bar x$.
\end{lemma}
\noindent {\bf Proof sketch.}
    Denote $(\bar y, \bar \pi, \bar s)$ and $(\hat y,\hat \pi, \hat s)$ as the solutions to the KKT system for $\bar x$ and $\hat x$ respectively. We claim that there exists an $\epsilon>0$ such that whenever $\| \bar x - \hat x \| \, \le \, \epsilon$, then for any $i \in \{1,\cdots, m\}$, one of the following three cases holds: (i) $\hat y_i = 0$, $\bar y_i = 0$ ;(ii) $\hat y_i > 0$, $\bar y_i > 0$; (iii) $\hat y_i = 0$, $\bar y_i > 0$, $\hat s_i = 0$, $\bar s_i = 0$. These cases are the basis for the KKT solution curve to be affine in $x$ for a given $\xi$. They ensure that complementarity is maintained on the line segment $\theta (\bar y, \bar \pi, \bar s) + (1-\theta) (\hat y,\hat \pi, \hat s)$, $\theta \in (0,1)$. Indeed, if this claim is true then it can be checked that $(y_\theta, \pi_\theta, s_\theta) = \theta (\bar y, \bar \pi, \bar s) + (1-\theta) (\hat y,\hat \pi, \hat s)$ is the unique solution of the KKT system associated with $x_\theta = \theta \bar x + (1-\theta) \hat x$, $\theta \in (0,1)$, i.e, $(y^*,\pi^*,s^*)$ is affine in $x$ on $[\hat x, \bar x]$. Now we prove this claim. We proceed by contradiction and assume that the claim is false.  Then there exists a sequence $\{ x^k \}$, where $x^k \to \hat x$ such that its associated KKT solution $(y^k,\pi^k,s^k)$ satisfies either (I) $\hat y_i > 0, y^k_i = 0$ for some $i$ or (II) $\hat y_i = 0, y^k_i > 0$, $\hat s_i > 0$, $s^k_i = 0$ for some $i$. First we suppose that case (I) happens infinitely often. Since $i \in \{1,\cdots, m \}$ has finite cardinality, there exists $i_0$ such that (I) happens infinitely often with $i = i_0$. We take the subsequence $k \in \mathcal{K}_1$ such that (I) is true for every $k \in \mathcal{K}_1$ and $i = i_0$. Note that by \eqref{kktsol}, the sequence $\{(y^k,\pi^k,s^k)\}$ is bounded since $\{ x^k \}$ is a bounded sequence and $\mathcal{A}$ can assume a finite number of values. 
    Therefore, we have a convergent subsequence $k \in \mathcal{K}_2 \subseteq \mathcal{K}_1$ such that $(y^k,\pi^k,s^k) \to (Y,\Pi,S)$, $k \in \mathcal{K}_2$. By continuity, $(Y,\Pi,S)$ is a KKT solution associated with $\hat x$. Therefore, $(Y,\Pi,S) = (\hat y, \hat \pi, \hat s)$ by uniqueness of the KKT solution. However, this implies that $\hat y_{i_0} = \lim_{k \in \mathcal{K}_2} y^k_{i_0} = 0$, contradicting the fact that $\hat y_{i_0} > 0$. Analogously, we can show that if (II) happens infinitely often, a contradiction also emerges. The claim is proved. $\hfill \Box$

\begin{theorem}\label{thm: lip-2st}
    Consider the economic dispatch problem in Section~\ref{sec:5.5}. Suppose LICQ holds for any $x$ and any realization $\xi$. There exists a uniform Lipschitz constant $L$ such that for any $\xi$, $\nabla F(x,\xi)$ is $L$-Lipschitz continuous.
\end{theorem}
\noindent{\bf Proof sketch.}
    Given a realization $\xi$, suppose we are given an $x^1$ and $x^2$ in the domain. Denote $x(t) = x^1 + t(x^2 - x^1)$, $0 \le t \le 1$. Suppose $(y(t),\pi(t),s(t))$ denotes the KKT solution associated with $x(t)$. For any $\hat t \in [0,1)$, by Lemma~\ref{lm: la}, we know that there exists $\bar t > \hat t$ such that if $\bar t$ is close enough to $\hat t$ then $(y^*,\pi^*,s^*)$ is affine on $[x(\hat t),x(\bar t)]$. Moreover, similar to the proof of Lemma~\ref{lm: la}, we have for any $i \in \{1,\cdots, m\}$, one of the following three cases holds: (i) $y_i(\hat  t) = 0$, $y_i(\bar t) = 0$; (ii) $y_i(\hat t) > 0$, $y_i(\bar t) > 0$; (iii) $y_i(\hat t) = 0$, $y_i(\bar t) > 0$, $s_i(\hat{t})=0$, $s_i(\bar{t}) = 0$. Therefore, on the interval $(\hat t, \bar t)$, the sign of $y_i(t)$ remains unchanged and the active set $\mathcal{A}$ remains stable. As a result, \eqref{kktsol} holds with a constant $\mathcal{A}$ on this small interval and there exists a uniform Lipschitz constant $\bar L$ independent of $x$ or $\xi$ since $Q$, $W$, $T$ are independent of $\xi$ for the economic dispatch problem and $\mathcal{A}$ can take on only a finite number of values.  If $\hat t \in (0,1]$ and $\hat t > \bar t$ with $\bar t$ close enough, we can prove a similar statement. As a result, for any $t \in [0,1]$, there is a small neighborhood of $t$ intersected with $[x^1,x^2]$ such that $(y^*,\pi^*,s^*)$ is $\bar L$-Lipschitz continuous. Due to finite cover property of compact set, we can take finite number of intervals that cover $[x^1,x^2]$ and on each interval $(y^*,\pi^*,s^*)$ is $\bar L$-Lipschitz. By concatenating these intervals we have $\| (y^*,\pi^*,s^*)(x^1) - (y^*,\pi^*,s^*)(x^2) \| \le \bar{L} \| x^1 - x^2 \|$ $\implies$ $\| \pi^*(x^1) - \pi^*(x^2) \| \le \bar{L} \| x^1 - x^2 \|$. Since $x^1$ and $x^2$ is arbitrary, we have $\pi(x,\xi)$ is $\bar L$-Lipschitz uniformly. Hence, $\nabla F(x,\xi) = T^\top \pi(x,\xi)$ is $L$-Lipschitz uniformly in $x$ for a given $\xi$, where $L = \bar L \| T \|$. $\hfill \Box$

According to Theorem~\ref{thm: lip-2st}, Assumption~\ref{ass: F} holds for the economic dispatch problem. As for the Assumption~\ref{ass: var}, again from \eqref{kktsol} and the settings in the economic dispatch problem, the stochastic gradient $\nabla F(x,\xi) = T(\xi)^\top \pi^*(x,\xi)$ is bounded and thus Assumption~\ref{ass: var} holds.
\end{appendix}

\end{document}